\documentclass[a4paper,12pt]{article}
\usepackage[a4paper,margin=2.4cm]{geometry}
\usepackage{amsthm,amsmath,amsfonts,amssymb}%
\usepackage{textcomp}%
\usepackage{mathrsfs}%
\usepackage[baseline]{euflag}%
\usepackage{url}
\newcommand{\tstY}{{\mathscr Y}}

\newcommand{\tstM}{{\mathscr M}}
\newcommand{\tstL}{{\mathscr L}}
\newcommand{\tstN}{{\mathscr N}}
\newcommand{\tstD}{{\mathscr D}}

\newcommand{\TCmap}{{\nu}}
\usepackage{tikz-cd}
\usepackage{graphicx}
\usepackage{caption}
\usepackage{subcaption}
\usepackage{hyperref}% % links to references
\hypersetup{bookmarksdepth=3,% depth of PDF section list
            colorlinks = true,%
            citecolor = {blue},%
            linkcolor = {red}%
}%
\usepackage{enumerate}%
\usepackage{graphicx}%
\usepackage{mathtools}%
\mathtoolsset{showonlyrefs=true} % to number only equations that are referenced in the text
\usepackage{mathabx}%
\usepackage{hyperref}%
\hypersetup{colorlinks = true,%
            citecolor = {blue},%
            linkcolor = {red}%
}%
\numberwithin{equation}{section}%
\newtheorem{proposition}{Proposition}[section]%
\newtheorem{lemma}[proposition]{Lemma}%
\newtheorem{theorem}[proposition]{Theorem}%
\newtheorem{corollary}[proposition]{Corollary}%
\newtheorem{definition}[proposition]{Definition}%
\newtheorem{definitionproposition}{Definition/Proposition}[section]%
\newtheorem{notationremark}{Notation/Remark}[section]%
\theoremstyle{definition}%
\newtheorem{remark}[proposition]{Remark}%
\newtheorem{example}[proposition]{Example}%
\newtheorem{notation}[proposition]{Notation}%
\newtheorem{convention}[proposition]{Convention}%

\DeclareMathOperator{\End}{End}

\DeclareMathOperator{\Ric}{Ric}

\DeclareMathOperator{\V}{V}

\renewcommand{\phi}{\varphi}

\DeclareMathOperator{\DF}{DF}

\newcommand{\N}{\mathbb{N}}
\newcommand{\R}{\mathbb{R}}
\newcommand{\C}{\mathbb{C}}

\newcommand{\Z}{\mathbb{Z}}

\newcommand{\pr}{\mathbb{P}}
\renewcommand{\epsilon}{\varepsilon}

\DeclareMathOperator{\Vol}{Vol}

\newcommand{\holSAS}{\Xi}
\newcommand{\Cmet}{\mathcal{Z}}

\newcommand{\Reebmap}{\mathrm{R}}

\def\Ds{\mathrm{D}}
\def\Sas{{\rm{Sas}}}

\def\DF{\mbox{{\bf DF}}}

\def\EH{{\rm{EH}}}

\def\bfV{{\bf{V}}}
\def\bfS{{\bf{S}}}

     \def\Sas{\mathcal{S}}  \def\mR{\mathcal{R}} \def\Sas{\mathcal{S}}            \def\mA{\mathcal{A}}
\def\mO{\mathcal{O}}

 \def\mL{\mathcal{L}}
\def\bT{\mathbb T}

\def\bS{\mathbb S}

 \def\kt{\mathfrak{t}} 

 \def\actscal{\mathcal{A}}
  \def\actvol{\mathcal{B}}

\def\Ric{{\rm Ric}}

\def\ra{\rightarrow }

\def\Hf{{\bf{F}}}

\def\CR{{\mathrm{J}}}

\def\grad{{\mathrm{grad}}}
\def\scal{{\mathrm{Scal}}}
\def\Sas{{\rm{Sas}}}
\def\con{{\rm{con}}}

\def\proj{{\rm p}}

\def\uI{\underline{I}}

\newcommand{\leftrarrows}{\mathrel{\raise.75ex\hbox{\oalign{%
  $\scriptstyle\leftarrow$\cr
  \vrule width0pt height.5ex$\hfil\scriptstyle\relbar$\cr}}}}
\newcommand{\lrightarrows}{\mathrel{\raise.75ex\hbox{\oalign{%
  $\scriptstyle\relbar$\hfil\cr
  $\scriptstyle\vrule width0pt height.5ex\smash\rightarrow$\cr}}}}
\newcommand{\Rrelbar}{\mathrel{\raise.75ex\hbox{\oalign{%
  $\scriptstyle\relbar$\cr
  \vrule width0pt height.5ex$\scriptstyle\relbar$}}}}

\newcommand{\intprod}{\mathbin{\raisebox{\depth}{\scalebox{1}[-1]{$\lnot$}}}}

\def\cst{\overline{{\bf c}}}
\def\biho{\Pi}

\def\uzeta{\underline{\zeta}}

\def\coneMetric{\Theta}
\def\euler{{\bf e} }

\DeclareRobustCommand{\gobblefive}[5]{}
\newcommand*{\SkipTocEntry}{\addtocontents{toc}{\gobblefive}}% to hide a subsection from toc

\title{The Einstein-Hilbert functional and the Donaldson-Futaki invariant}

\author{
    Abdellah Lahdili\thanks{Department of Mathematics, University of Aarhus. lahdili.abdellah@gmail.com} ,
    Eveline Legendre\thanks{\textit{Corresponding author.} Universit\'e Claude Bernard Lyon 1, Institut Camille Jordan, \'equipe AGL. legendre@math.univ-lyon1.fr} ,
    Carlo Scarpa\thanks{CIRGET, Universit\'e du Qu\'ebec \`a Montr\'eal. scarpa.carlo@courrier.uqam.ca}
    }
\date{\vspace{-1cm}}%

\begin{document}

\maketitle

\noindent \textbf{Abstract.}
    Given a K\"ahler manifold $M$ polarised by a holomorphic ample line bundle $L$, we consider the circle bundle associated to the polarisation with the induced transversal holomorphic structure. The space of contact structures compatible with this transversal structure is naturally identified with a bundle, of infinite rank, over the space of K\"ahler metrics in the first Chern class of $L$. We show that the Einstein-Hilbert functional of the associated Tanaka-Webster connections is a functional on this bundle, whose critical points are constant scalar curvature Sasaki structures. In particular, when the group of automorphisms of $(M,L)$ is discrete, these critical points correspond to constant scalar curvature K\"ahler metrics in the first Chern class of $L$. We show that the Einstein-Hilbert functional satisfies some monotonicity properties along some one-parameter families of CR-contact structures that are naturally associated to test configurations, and that its limit on the central fibre of a test configuration is related to the Donaldson-Futaki invariant through an expansion in terms of an extra real parameter. As a by-product, we obtain an original proof that the existence of cscK metrics on a polarized manifold implies K-semistability. We also show that the limit of the Einstein-Hilbert functional on the central fibre coincides with the ratio of the equivariant index characters pole coefficients of the central fibre.

\tableofcontents

\section{Introduction}
In this paper, we propose a new approach to the problem of whether or not the first Chern class of a given ample line bundle $L$, over a compact complex manifold $M$, contains a constant scalar curvature K\"ahler (cscK) metric. This problem, its extension to extremal Calabi metrics and the conjecture surrounding its relation to K-stability, has driven great part of the research in K\"ahler geometry for the last few decades. Out of these efforts, a beautiful and rich theory has emerged.
 
The best-known, but not trivial, case is when the ample line bundle $L$ is the anti-canonical bundle $-K_M$ over a (Fano) complex manifold $M$. In this case, the cscK metrics in the class $2\pi\mathrm{c}_1(M)=2\pi\mathrm{c}_1(-K_M)$, if any exist, are K\"ahler-Einstein metrics of positive curvature. It is now known that such a metric exists if and only if $(M,-K_M)$ is K-stable, and we even have a variety of proofs of this remarkable result, starting with \cite{CDS, Tian_KstabFano}. Since then, other stability notions have been introduced for Fano manifolds, that have been shown to be equivalent to K-stability \cite{fujita-valuative,fujita-odaka}. Each of these sheds light on a different aspect of the problem, and they have all been proven to be very useful in the study of the moduli space of K-stable Fano varieties. Arguably, the main reasons why the theory of cscK metrics has been so successful in the Fano case, is that the cscK equation reduces to a second-order complex Monge-Ampère equation, and that from the algebraic point of view the problem can be reduced to test K-stability, or its variations (uniform K-stability/semistability), on special test configurations \cite{lixu-special}, where one can use Ding stability \cite{Berman}.

For general polarized manifolds $(M,L)$, the Yau-Tian-Donaldson conjecture predicts an equivalence between the existence of cscK metrics in $2\pi c_1(L)$ and K-stability of $(M,L)$. On the differential side, the existence of cscK metrics is now well-understood thanks to the celebrated works of Chen-Cheng \cite{chen-cheng}, after \cite{BDL, DarvasRubinstein_propernessKE} and extending \cite{Tian}, showing that the existence of a cscK (including K\"ahler-Einstein) metric is equivalent to the coercivity (or \emph{properness}) of the K-energy functional on the space of K\"ahler potentials introduced by Mabuchi. The algebro-geometric side of the conjecture is now investigated via different approaches, notably \cite{BJ-nonarch,li-models1}, based on relating the asymptotic behavior of the K-energy to K-stability.
 
As the case of Fano manifolds exemplifies, the study of such a hard problem can only benefit from the introduction of new avenues of investigation. With this in mind, we introduce a new and rather old ingredient in the study of cscK metrics on a polarized manifold $(M,L)$, a version of the Einstein-Hilbert functional, in the hope that it will provide a fruitful new perspective on the problem. The Einstein-Hilbert functional was a crucial tool to solve the Yamabe problem, searching for constant scalar curvature Riemannian metrics in a given conformal class on a compact Riemannian manifold. A modified version of this functional, the one we are going to exploit, was instead first introduced to understand the existence of special contact forms compatible with a given CR structure, those with constant \emph{Tanaka-Webster} scalar curvature. These CR-contact forms coincide with the critical points of this Einstein-Hilbert functional on the CR manifold, and their existence on a given CR manifold is known as the \emph{CR-Yamabe problem}.

The main contributions of this paper is to highlight a link between a family (actually, a bundle) of CR-Yamabe problems and the cscK problem on a polarized manifold. We also explore an algebro-geometric counterpart of this picture and exhibit an explicit relation between the Einstein-Hilbert functional on a space of CR-contact structures and the ratio of the equivariant index character pole coefficients, establishing a direct link with K-stability. We now present the results in more detail.

\subsection{The CR-Yamabe problem and cscK metrics on polarized manifolds}

Let $M$ be a compact complex manifold, of complex dimension $n$. Any complex line bundle $L\to M$ classically corresponds to a principal $\mathbb{S}^1$-bundle $\pi:N\ra M$. Having fixed (now and for the rest of the paper) a generator $\xi$ of the $\mathbb{S}^1$-action on $L$, the complex structure on $M$ defines a transversal holomorphic structure $\underline{I}$ on $N$. Fixing a holomorphic structure on $L$, which is assumed to be ample from now on, any positively curved Hermitian metric $h$ determines a unique connection $1$-form $\eta_h$ on the principal bundle $N$, which turns out to be a contact structure on $N$ with Reeb vector field $\xi$. Moreover, $\eta_h$ satisfies a compatibility condition with the underlying transversal holomorphic structure $\underline{I}$ on $N$: it is CR-contact with respect to the unique CR-structure $I_h$ lifting $\uI$ on $\ker\eta_h$.

We consider the space $\holSAS(M,L)$ of CR-contact structures obtained as above when varying the Hermitian metric on $L$ while keeping the rest of the structure fixed,
\begin{equation*}
    \Xi(M,L)=\left\lbrace\eta_h\;\middle|\; h\mbox{ positively curved Hermitian metric on }L\right\rbrace.
\end{equation*}
Note that for any $\eta\in \holSAS(M,L)$, there is a unique K\"ahler form $\omega\in 2\pi c_1(L)$ on $M$ such that 
\begin{equation}\label{eq:correspSas=Kahl}
    \pi^*\omega = d\eta,
\end{equation}
and this provides a one-to-one correspondence between $\holSAS(M,L)$ and the space of K\"ahler metrics in $2\pi c_1(L)$. Now, we expand the set $\holSAS(M,L)$ within the space of $\xi$-invariant contact structures on $N$ by considering all the possible conformal transformations
\begin{equation*}
    \Cmet(M,L)\coloneqq  \left\lbrace f^{-1}\eta \;\middle|\; f\in C^{\infty}(N,\R_{>0})^\xi, \eta\in \holSAS(M,L)\right\rbrace,
\end{equation*}
where $C^{\infty}(N,\R_{>0})^\xi$ denotes the $\xi$-invariant positive real functions on $N$. In other words, $C^{\infty}(N,\R_{>0})^\xi=\pi^* C^{\infty}(M,\R_{>0})$. Any $\alpha \in \Cmet(M,L)$ is compatible with the CR structure induced by the transversal holomorphic structure on the distribution $\ker \alpha$, thus $\Cmet(M,L)$ is a space of CR-contact structures. Note that in general a CR-contact structure $\alpha\in\Cmet(M,L)$ is not Sasakian, and the Reeb vector field associated to $\alpha$ coincides with $\xi$ if and only if $\alpha\in \holSAS(M,L)$.

A crucial observation of Tanaka~\cite{tanaka} and Webster~\cite{webster2} is that any CR-contact structure $(\alpha,I)$ on a manifold $N$ determines a unique affine connection on $TN$, the so-called {\it Tanaka-Webster connection}, as well as a Riemannian metric on $N$, from which one can define the so-called \emph{Tanaka-Webster scalar curvature} $\scal(\alpha,I) \in C^\infty(N,\R)$ and thus an Einstein-Hilbert functional $\EH:\Cmet(M,L)\to\R$ given by the (normalized) total Tanaka-Webster scalar curvature,
\begin{equation}\label{eq:EH_TanakaWebster}
    \EH(\alpha) = \frac{\int_N \scal(\alpha,I)\, {\alpha}\wedge( d{\alpha})^{[n]} }{\bfV({\alpha})^{\frac{n}{n+1}}},
\end{equation}
where $\bfV({\alpha})=\int_N {\alpha}\wedge (d{\alpha})^{[n]}$ is the volume of $N$ with respect to $\alpha$, and for any differential form $\theta$ and natural number $k$, we let $\theta^{[k]}\coloneqq \theta^k/k!$. The exponent of the volume term in \eqref{eq:EH_TanakaWebster} is chosen so that $\EH$ is invariant under dilations on $\Cmet(M,L)$.

Our idea is to consider the Einstein-Hilbert functional on the space $\Cmet(M,L)$ we introduced and we prove the following.
\begin{theorem}[{\bf Theorem}~\ref{lem:CRITEH}]\label{theoEH=MabOnZ}
    Via the correspondence~\eqref{eq:correspSas=Kahl}, a K\"ahler metric in $2\pi c_1(L)$ has constant scalar curvature if and only if it is a critical point of the Einstein-Hilbert functional.    Moreover, the critical points of $\EH$ in $\Cmet(M,L)$ coincide with the Sasakian structures of constant scalar curvature.
    
    In particular, if the reduced automorphism group of $M$ is discrete, then the critical points of $\EH$ on $\Cmet(M,L)/\R_+$ coincide with the constant scalar curvature K\"ahler metrics in $2\pi c_1(L)$.
\end{theorem}

This should be compared with the recent work \cite{AfeltraMalchiodi_EH} where $\EH$ is considered as a functional on the set of CR contact structures $(\eta,I)$ that are compatible with a fixed contact distribution $D$ (see Section \ref{ss:PreliminairiesCONTACT} for a definition of these structures). There also the critical points are Sasakian structures with constant Tanaka-Webster scalar curvature. This hints to an interesting duality between the transversal holomorphic and contact point of views.

By definition, $\Cmet(M,L)$ is a (trivial) bundle over $\holSAS(M,L)$ or the space of K\"ahler metrics in $2\pi c_1(L)$, and we denote by $\kappa :\Cmet(M,L) \longrightarrow \holSAS(M,L)$ the bundle map. Each fibre of $\kappa$ is isomorphic to $C^\infty(M,\R_{>0})$ and coincides with a conformal class of $\xi$-invariant CR-contact structures compatible with the same CR structure. When restricted on a fibre, $\EH_{|_{\kappa^{-1}(\omega)}}$ coincides with a functional introduced by Jerison and Lee \cite{JerisonLee} to track contact forms with constant Tanaka-Webster scalar curvature (cscTW, from now on) on a given CR-structure.

\begin{remark}
Theorem~\ref{theoEH=MabOnZ} is the regular version of a more general result we prove for arbitrary Sasaki manifolds, see Theorem \ref{lem:CRITEH}.
\end{remark}

\begin{remark}\label{rem:WeightedKahl}
We can see the Einstein-Hilbert functional as being defined on the space of pairs $(\omega,f)$, for a K\"ahler metric $\omega$ in $2\pi\mathrm{c}_1(L)$ and $f\in C^{\infty}(M,\R_{>0})$. From this point of view, the Einstein-Hilbert functional is similar to a version of Perelman's entropy functional studied by Inoue~\cite{inoue-ent} to characterize $\mu$-cscK metrics, a generalisation of cscK metrics. We refer to Remark \ref{rmk:Perelmanentropy} for a discussion of some of the differences and similarities between the two functionals. It is likely that, more generally, weighted extremal K\"ahler metrics (including extremal Calabi metrics) are critical points of similar but modified weighted Einstein-Hilbert functionals, of which the one we study here and Perelman's entropy are special cases. For example, we point out in Remark \ref{rmk:EinsteinHilbert_weighted} that our work can be applied to a family of equations that includes the classical Yamabe problem, leading in particular to a characterisation of conformally K\"ahler Einstein-Maxwell metrics \cite{ApostolovMaschler_EinsteinMaxwell}, among others.
\end{remark}

\subsection{Polarized test configurations and ribbons of CR-contact structures}

Consider a smooth ample test configuration $(\tstM,\tstL)$ over $(M,L)$,  and denote by $\nu:\tstM\to \pr^1$ the projection map, equivariant with respect to the linearised $\mathbb{C}^*$-action on $(\tstM,\tstL)$. This test configuration gives rise to a test configuration $(\tstN, \uI, \xi)$ in the category of transversally holomorphic manifolds of Sasaki type, as in in \cite{ACL}, and let $\zeta$ be the generator of the induced $\mathbb{S}^1$-action on $\tstN$. Choosing a $\bS^1_\zeta$-invariant Sasaki-contact form $\tilde{\eta} \in \holSAS(\tstM,\tstL)$, for each $\tau \in \pr^1\backslash\{0\}$, we get a regular Sasaki manifold, determined by a form $\eta_\tau \in \holSAS(M_\tau, L_{|_{M_\tau}})$ on the circle bundle $N_\tau$ over the fibre $(M_\tau, L_{|_{M_\tau}})$ where $M_\tau= \TCmap^{-1}(\tau)$.  Moreover, for all $s\in \R$ with $|s| < \lVert (\tstM,\tstL) \rVert^{-1}_\infty$, we have $f_s \coloneqq  1-s\tilde{\eta}(\zeta)>0$ on $\tstN$ as observed in Lemma~\ref{lemma:NORMtc}. This allows us to consider a $2$-parameter family, that we call a \emph{ribbon of CR-contact structures}, 
\begin{equation}\label{eq:IntroPencil}
 f_s^{-1}\eta_\tau \in \Cmet(M_\tau, L_{|_{M_\tau}})
\end{equation}
where $f_s$ denotes also the pull-back of $f_s$ on the transversally holomorphic circle bundle $N_\tau$ over $(M_\tau,L_{|_{M_\tau}})$. The second main result of this paper comes from investigating the limit of the Einstein-Hilbert functional on \eqref{eq:IntroPencil} when $\tau \mapsto 0 \in \pr^1$.  

\begin{theorem}\label{theINTRO_TCdl}
    Let $(\tstM,\tstL)$ be a smooth ample test configuration over $(M,L)$ with reduced central fibre $(M_0,L_0)$. Then for any $\bS^1_\zeta$-invariant Sasaki-contact form $\tilde{\eta} \in \holSAS(\tstM,\tstL)$ and associated ribbon \eqref{eq:IntroPencil} there exists $\epsilon >0$ such that the limit $$\EH_s(\tstM,\tstL)\coloneqq \lim_{\tau\ra 0} \EH(f_s^{-1}\eta_\tau)$$ is a smooth function of $s\in (-\epsilon,\epsilon)$ and does not depend on the choice of $\tilde{\eta} \in \holSAS(\tstM,\tstL)$. Moreover, $$\left( \frac{d}{ds}\EH_s(\tstM,\tstL)\right)_{s=0} = \Vol(L)^{-n/n+1} \DF(\tstM,\tstL)$$ where $\DF(\tstM,\tstL)$ is a positive multiple of the Donaldson-Futaki invariant of $(\tstM,\tstL)$.
\end{theorem}
An explicit global formula for $\EH_s(\tstM,\tstL)$ is given in Theorem~\ref{theo:Globalformula} below. Moreover, the upper bound for the parameter $\epsilon$ is $\lVert (\tstM,\tstL) \rVert_{\infty}^{-1}$, see Lemma~\ref{lemma:NORMtc}.

Recall that a polarized manifold $(M,L)$ is said to be K-semistable (respectively K-stable) if the Donaldson-Futaki invariant of any (normal) test configuration for $(M,L)$ is non-negative (respectively positive and zero only on product test configurations), see \S \ref{ss:testconfigSASAK}.  While the proof of Theorem~\ref{theINTRO_TCdl} uses crucially that $(\tstM,\tstL)$ is smooth and ample, the coefficients in the expansion above are continuous with respect to the polarisation $\tstL$, see Theorem~\ref{theo:Globalformula}. This allows us to adapt a strategy of Dervan-Ross \cite{DR} and prove

\begin{corollary}\label{coro:Kss=minEH} A polarized manifold $(M,L)$ is K-semistable if and only for any smooth ample test configuration $(\tstM,\tstL)$ we have $$\left( \frac{d}{ds}\EH_s(\tstM,\tstL)\right)_{s=0}\geq 0.$$
\end{corollary}

To exploit these results, one needs to control the behavior of the Einstein-Hilbert functional over a ribbon of the form \eqref{eq:IntroPencil} obtained from a test configuration $(\tstM,\tstL)$ and a choice of CR-structure $\tilde{\eta} \in \holSAS(\tstM,\tstL)$. To this end, we use the well-established theory of weak geodesics associated to test configurations \cite{chen-tang,phong-sturm,CTW1,CTW2,zakarias} to study the total Tanaka-Webster scalar curvature and the volume functional as weighted K\"ahler functionals (see Remark~\ref{rem:WeightedKahl} above). This allows us to rely more directly on the works of \cite{BoB,DR,Lahdili,zakarias} and we get easily that the (contact) volume functional $\alpha\mapsto\bfV(\alpha)$ is well defined on the space of $C^{1,1}$ CR-contact structures $\Cmet^{1,1}(M, L)$.

The total Tanaka-Webster scalar curvature requires more care and, as in Inoue's work \cite{inoue-ent}, we rather study its primitive $t\mapsto \actscal(s,t)$ along paths of the form \eqref{eq:IntroPencil}, where $e^{-t}=\tau \mapsto d\eta_\tau$ is the lift through \eqref{eq:correspSas=Kahl} of a $C^{1,1}$-geodesic of K\"ahler currents. We prove that so defined, $t\mapsto \actscal(s,t)$ is continuous and pointwise convex, see Theorem~\ref{theo:ActionPointwiseCvxC0}, from which, together with Theorem~\ref{theINTRO_TCdl}, we re-obtain the following classical result of \cite{Donaldson_lowerbounds}.

\begin{corollary}[{\bf Corollary}~\ref{cor:Ksemistab}]\label{coro:Crit=>Kss}
If the Einstein-Hilbert functional on $\Cmet(M,L)$ admits a critical point in $\holSAS(M,L)$, i.e.\ if there exists a cscK metric in $2\pi c_1(L)$, then $(M,L)$ is K-semistable. 
\end{corollary}

\subsection{An affine version of the Einstein-Hilbert functional}

When a compact torus $\bT$ acts linearly on $L^{-1}$, and contains as a subgroup the fibrewise action of $\xi$, one can define the Sasaki-Reeb cone, sometimes called the Reeb cone or the reduced Reeb cone, see \S\ref{s:towards} for a reminder of the definition.

In \cite{CS}, Collins and Sz\'ekelyhidi give an alternative definition of the Donaldson-Futaki invariant of a test configuration $(\tstM,\tstL)$ in terms of the Sasaki-Reeb cone $\kt_{0,+}\subset \mbox{Lie} \bT$ of the central fibre $(M_0,L_0)$. They consider the equivariant index character, that is
\begin{equation*}
    \kt_{0,+} \ni x\mapsto \sum_{k\in \N,\alpha \in \kt^*_0 } e^{-t \langle \alpha,x \rangle} \dim H^0(M_0,L_0^k)_\alpha
\end{equation*}
where
\begin{equation*}
    \bigoplus_{k\in \N} H^0(M_0,L_0^k)=\bigoplus_{k\in \N} \bigoplus_{\alpha \in \kt^*_0 } H^0(M_0,L_0^k)_\alpha.
\end{equation*}
is the weight space decomposition with respect to the induced action of $\bT$. Collins and Sz\'ekelyhidi show then that, when $\mbox{Re}(t) >0$ and $x \in \kt_{0,+}$, the equivariant index character admits an asymptotic expansion
\begin{equation}\label{eq:ExpEICintro}
    \sum_{k\in \N,\alpha \in \kt^* } e^{-t \langle \alpha, x \rangle} \dim H^0(M_0,L_0^k)_\alpha = \frac{a_0(x) n!}{t^{n+1}} + \frac{a_1(x) (n-1)!}{t^{n}} + O(t^{1-n})
\end{equation}
where $a_0, a_1:\kt_{0,+}\ra \R$ are smooth functions, \cite[Theorem 3]{CS}. Moreover, the $\mathbb{C}^*$-action coming from the test configuration determines a segment $s\mapsto \xi-s\zeta \in \kt_{0,+}$ passing through $\xi$ at $s=0$, and Collins and Sz\'ekelyhidi show that the Donaldson-Futaki invariant of the test configuration can be expressed in terms of $a_0(\xi)$, $a_1(\xi)$, and the derivatives of $a_0(\xi-s\zeta)$ and $a_1(\xi-s\zeta)$ at $s=0$, see \cite[Definition 5.2]{CS}.  As noticed for example in \cite{BHLTF1,li}, the expression of the Donaldson-Futaki invariant given by Collins-Sz\'ekelyhidi is, in more concise form,
\begin{equation*}
    \DF(\tstM,\tstL) = \frac{a_0(\xi)^{n/n+1}}{n}\left(\frac{d}{ds} \frac{a_1(\xi-s\zeta)}{a_0(\xi-s\zeta)^{n/n+1}}\right)_{s=0}.
\end{equation*}

This observation leads us to compare the normalized ratio of the equivariant index character pole coefficients with the limit of the Einstein-Hilbert functional along a test configuration, and we show that they agree on smooth test configurations up to a positive dimensional constant.
\begin{theorem}\label{theo:EHaff=limEH}
Let $(\tstM,\tstL)$ be a smooth ample test configuration over $(M,L)$ with reduced central fibre, then  
    \begin{equation}\label{eq:EHaff=limEH}
        \EH_s(\tstM,\tstL) = 16\pi \frac{a_1(\xi-s\zeta)}{a_0(\xi-s\zeta)^{n/n+1}}.
    \end{equation}
\end{theorem}
To prove this result we provide a global formula for the equivariant index character pole coefficients of the central fibre in terms of those of $(\tstM,\tstL)$ and the volume and total scalar curvature of $(M,L)$, see Proposition~\ref{prop:GlobalExpAffEH} and Proposition~\ref{prop:5.5followup}. Then Theorem~\ref{theo:EHaff=limEH} follows by comparing this with the expression for $\EH_s(\tstM,\tstL)$ obtained in Theorem~\ref{theo:Globalformula}.  Note that the right hand side of \eqref{eq:EHaff=limEH}, that we call the {\it affine Einstein-Hilbert functional} of $(\tstM,\tstL)$, is defined as soon as the equivariant index character of the central fibre admits an expansion of the form \eqref{eq:ExpEICintro}, which happens for any ample test configuration, even non-smooth ones.

The affine Einstein-Hilbert functional is also well-defined for more general polarized cones, those having only quasi-regular or even irregular Reeb vector fields. This suggests that the identity \eqref{eq:EHaff=limEH} should hold for arbitrary Sasaki test configuration, if one interprets the left-hand side appropriately. In order to extend our results (Theorem~\ref{theINTRO_TCdl}, Corollaries~\ref{coro:Kss=minEH},~\ref{coro:Crit=>Kss}) to this setting, one could either develop the theory of $C^{1,1}$-geodesics in Sasaki geometry or apply the strategy of \cite{ACL} and work with a certain class of weighted cscK on a regular quotient, assuming one exists. We leave such generalisations for future work.

\paragraph{Relation to the normalized volume}
Recall that when $L= -K_M$, Chi Li introduced the {\it normalized volume} $\widehat{V}$ on the space of $\C^*$-invariant valuations and shows that K-semistability of $(M,L)$ is equivalent to $M$, seen as a divisorial valuation $M\subset L^{-1}$, being the minimiser of the normalized volume \cite{li}. It follows from \cite[Proposition 6.4]{CS2} that, on the Sasaki-Reeb cone of the central fibre of a special test configuration over a Fano variety, the affine Einstein-Hilbert functional coincides with the normalized volume up to a positive dimensional constant. Recall that, as highlighted in \cite{liMathZ}, when $M$ is a irreducible variety, the Sasaki-Reeb cone is canonically identified with the subcone of $\C^*$-invariant valuations of the affine cone $\hat{Y}\subset \C^N$ whose regular part corresponds to $L^{-1}\backslash M$ through a $\bT$-equivariant holomorphic embedding.

Chi Li's Theory can be understood as a wide generalization of an important result of Martelli, Sparks, and Yau~\cite{MSY2} showing that the volume, seen as a function on a fixed transversal polytope of the Sasaki-Reeb cone, detects the Sasaki-Reeb vector fields with vanishing transversal Futaki invariant. The normalized volume, as defined in \cite{liMathZ}, does not admit a direct generalization to arbitrary polarized cones since it uses the log discrepancy which is only defined on cones that are log terminal (cones over a Fano manifolds).

On the other hand, the Einstein-Hilbert functional on the Sasaki-Reeb cone \emph{also} detects the Sasaki-Reeb vector fields with vanishing transversal Futaki invariant \cite{BHLTF1}, is defined on arbitrary polarised manifolds, and could be thought as an extension of the normalized volume on general polarized cone. To reveal the Einstein-Hilbert functional as an accurate extension of the normalized volume, one needs to propose an extension of this functional on a (connected) topological space containing as a subset the space of test configurations.

We leave this for future work but notice that a direct consequence of this would be a brand new interpretation of $K$-semistability of a polarized variety as a variational problem, mirroring \cite[Theorems 4.1, 6.1]{li}, in which K-stability is naturally understood in terms of the K-stability of the associated polarized complex cone detecting, for example, the, not necessarily regular, K\"ahler-Ricci flat conical metrics. Compared to the Mabuchi functional, the Einstein-Hilbert functional has the advantage of detecting K-stability in a wider space of deformations and could, eventually, provide a useful tool to study open problems like the isolation problem for conical cscK metrics \cite{BHLTF3} and deformation theory in the spirit of \cite{DonaldsonSun_GHlimits}.

\SkipTocEntry\subsection*{Plan of the paper}
The next section contains a review of the basic notions and results on Sasaki and contact geometry for a fixed CR-structure. This corresponds to studying the geometry of a fibre of the bundle map $\kappa : \Cmet(M,L) \ra \holSAS(M,L)$ and is the appropriate setting in which to apply the works of Tanaka, Webster, Jerison, Lee and Tanno, that we will recall briefly. In Section~\ref{s:Cmet} we construct the space $\Cmet(M,L)$, together with a more general version for non-regular Sasaki manifolds. We compute the variation of $\EH$ under deformations of the CR-structure, which are transversal to the fibres of $\kappa$, and thus prove Theorem~\ref{theoEH=MabOnZ}. Up to this point, each result and claim holds for general (non-regular) Sasaki manifolds, but in Section \ref{s:EHtestconfig} we must restrict our attention to regular Sasaki manifolds, or, equivalently, to polarized K\"ahler manifolds, in order to employ the variational methods often used in the theory surrounding the Yau-Tian-Donaldson conjecture. The main results of Section \ref{s:EHtestconfig} are Theorem ~\ref{theo:Globalformula} and the study of the primitive $\actscal$ along $C^{1,1}$-geodesics, as outlined above. The final section~\ref{sectionTOWalgebraic} has a more algebraic flavor and is devoted to the proof of Theorem~\ref{theo:EHaff=limEH}, after a brief recap of the Collins-Sz\'ekelyhidi theory leading to the definition of the affine Einstein-Hilbert functional.

\section[The CR Einstein-Hilbert functional]{The Einstein-Hilbert functional of a pseudoconvex CR manifold} 

\subsection{Brief recap on Sasaki geometry}\label{ss:PreliminairiesCONTACT}

Throughout this section, $(N,\Ds)$ stands for a co-oriented contact manifold, that is $N$ is a smooth manifold which will be assumed compact and connected, unless otherwise stated, and $\Ds \subset TN$ is a contact distribution. The space of compatible contact $1$-forms is denoted $\Gamma(\Ds_+^0)$. This is the set of nowhere vanishing sections of the annihilator $\Ds^0 \subset T^*N$ (i.e.\ $\Ds=\ker \eta$) such that $\eta\wedge (d\eta)^n$ is a volume form on $N$, positive with respect to the chosen orientation. Note that if $\eta,\tilde{\eta}\in\Gamma(\Ds_+^0)$ then there exists a positive function $f$ such that
\begin{equation}
    \tilde{\eta}=\eta_f\coloneqq f^{-1}\eta.
\end{equation}
Given $\eta \in \Gamma(\Ds_+^0)$, there is a unique {\it associated Reeb vector field} $\Reebmap^\eta \in \Gamma(TN)$, determined by the relations 
\begin{equation} \label{e:contactFvsReeb}
\eta(\Reebmap^\eta)=1 \quad \mL_{\Reebmap^\eta}\eta =0.
\end{equation}
This implies that $\Reebmap^\eta \in \mbox{con}(N,\Ds)$, the Lie algebra of the group of contactomorphisms of $(N,\Ds)$. The infinite dimensional cone of Reeb vector fields is often denoted $\mbox{con}_+(N,\Ds)$.

\smallbreak

Now suppose $I\in \End(\Ds)$ is a CR structure on $(N,\Ds)$, i.e.\ a (pointwise) complex structure on $\Ds$ such that $\Ds^{1,0}\coloneqq  \{a-i I a\,|\, a\in \Ds\}$ is closed under Lie bracket in $TN\otimes\C$. We also assume throughout that $(\Ds,I)$ is a \emph{strictly pseudoconvex} CR-structure in the sense that it admits \emph{positive} (or \emph{compatible}) sections, i.e.\ sections $\eta\in \Gamma(\Ds^0)$ for which $d\eta(\cdot,I\cdot)$ is positive definite on $\Ds$. Note that $\eta([\cdot,I\cdot]) = (d\eta)(\cdot, I\cdot)$ since $\ker\eta=\Ds$ and thus $\eta$ is positive if and only if $f^{-1}\eta$ is positive for any $f\in C^\infty(N,\R_{>0})$.  Therefore a  strictly pseudoconvex CR structure imposes a co-orientation on $(N,\Ds)$ for which the set of positive sections is a space of contact forms of the form $\Gamma(\Ds_+^0)$ on $\Ds$. To emphasize this relation we denote $\Cmet{(\Ds,I)} = \Gamma(\Ds_+^0)$ the space of {\it contact forms} compatible with the co-orientation of $(\Ds,I)$.

\begin{definitionproposition}\label{defnPropDEFsasak}[e.g.\ \cite{BG:book}]
Let $(N, \Ds, I)$ be a strictly pseudoconvex CR manifold with a compatible contact $1$-form $\eta \in \Cmet{(\Ds,I)}$ and associated Reeb vector field $\Reebmap = \Reebmap^\eta$. Then we define the pointwise endormorphism $I^\Reebmap\in \End(TN)$ by $I^\Reebmap(a)\Reebmap =0$. The following tensor is a Riemannian metric\footnote{The second term is divided by $2$ here like in for example \cite{tanno,ACL,FOW,MSY2} but unlike \cite[(6.4.2)]{BG:book}. The convention we use is determined by the associated K\"ahler cone metric.}
\begin{equation}\label{eqRiemCRcontacMetric}
    g_{\Ds,I,\eta} \coloneqq  \eta \otimes \eta +\frac{1}{2}d\eta (\cdot, I^\Reebmap\cdot).
\end{equation}
Such structure $(\Ds, I, \eta)$, which determines uniquely the Reeb vector field $\Reebmap$ via~\eqref{e:contactFvsReeb} and the metric $g_{\Ds,I,\eta}$ above, is called a \emph{CR-contact} structure. Whenever $\Reebmap$ is a Killing vector field, i.e.\ $\mL_{\Reebmap} g_{\Ds,I,\eta} =0$, then the structure $(\Ds, I, \eta)$ is a \emph{Sasaki structure} and, in this case, $\Reebmap$ is called a \emph{Sasaki-Reeb vector field} (and typically denoted $\xi$ below).
\end{definitionproposition}

\begin{notation}\label{notation:xi_invariant}
    For any space of smooth sections of tensors, say $\mA$, over $N$, and any vector field $\xi$ inducing the action of a compact torus, we denote by $\mA^\xi$ the space of $\xi$-invariant elements of $\mA$, i.e.\ $\mA^\xi\coloneqq \{\psi \in \mA\;|\; \mL_\xi \psi =0\}$.
\end{notation}

\begin{convention}\label{conventionTK}
The (local) K\"ahler quotient of a Sasaki structure $(\Ds, I,\eta)$ refers to the K\"ahler structure $(\omega_\eta,\uI,g_\omega)$ obtained on the local quotient by the flow of the associated vector field with the following convention. Denote $\pi : U_N\ra U_M$ such local quotient map, then $\pi_* :\Ds \simeq TU_M$, $\pi^*\omega_\eta \coloneqq  d\eta$ and $\uI\circ\pi_* =\pi_* \circ I$. In particular, $\pi^*g_\omega = 2(g_{\Ds,I,\eta})_{|_\Ds}$.
\end{convention}

\subsection{The Tanaka-Webster connection and its scalar curvature}\label{sec:TWconnection} Curvatures of CR-contact (or pseudo-hermitian) structures have been studied a few decades ago e.g.\ \cite{JerisonLee, tanaka, webster2}. We recall in this subsection some of their results.     
 
Tanaka proved in~\cite{tanaka} that given a strictly pseudoconvex CR manifold $(N, \Ds, I)$ and a compatible contact $1$-form $\eta \in \Cmet{(\Ds,I)}$ with associated Reeb vector field $\Reebmap=\Reebmap^\eta\in \con_+(N, \Ds)$, there exists a unique connection $\nabla=\nabla^{(\Ds,I,\eta)}$ on $TN$, \emph{the Tanaka-Webter connection}, preserving the whole structure  $(N, \Ds, I,\eta,\Reebmap)$ and whose torsion $T^\nabla$ satisfies 
\begin{equation}\label{eq:torsion}
    T^\nabla(a,b) =d\eta (a,b) \,\Reebmap \qquad T^\nabla(\Reebmap,Ia) =I T^\nabla(\Reebmap,a)
\end{equation}
for any horizontal vector fields $a,b\in \Gamma(\Ds)$. A key feature of this connection is the following.

\begin{theorem}[Webster \cite{webster2}, David \cite{davidtanaka}]\label{theoWEBSTER}
Let $(N,\Ds,I, \eta)$ be a quasi-regular Sasaki manifold with K\"ahler quotient $(M,\omega, J)$. Then the pull-back of the Levi-Civita connection of $(M,\omega, J)$ is the Tanaka-Webster connection of $(N,\Ds,I, \eta)$.
\end{theorem}

The trace, with respect to \eqref{eqRiemCRcontacMetric}, of the Ricci curvature of the Tanaka-Webster connection is the \emph{Tanaka-Webster scalar curvature}, denoted $$\scal(\Ds,I,\eta)$$ or simply $\scal(\eta)$ if no confusion is possible, that is if $(\Ds,I)$ is fixed. Note that by Theorem \ref{theoWEBSTER}, if $(N, \Ds, I,\eta)$ is Sasaki and quasiregular then $\scal(\eta)$ is the scalar curvature of the K\"ahler quotient.

Lee \cite{Lee} and Tanno \cite{tanno}, independently proved a nice formula relating any two conformal CR-contact structures $\eta,\tilde{\eta} \in  \Cmet{(\Ds,I)}$ which will be the key for our approach. We will use this formula assuming one of the Reeb vector field, say $\xi\coloneqq  \Reebmap^\eta$ is Killing (thus CR-real holomorphic) and that the other contact form $\tilde{\eta}$ is $\xi$-invariant. This implies that $\tilde{\eta}=f^{-1}\eta$ for some $f\in C^\infty(N,\R_{>0})^\xi$. In that case, the Lee-Tanno formula is     
\begin{equation}\label{eq:TannoFormula}
    \scal(f^{-1}\eta) = f\,\scal(\eta) -2(n+1)\Delta^\eta f -(n+1)(n+2)f^{-1}\lvert df\rvert^2_\eta 
\end{equation}
where $\Delta^\eta$ denotes the Laplacian, and $\lvert\,\cdot\,\rvert_\eta$ the norm, with respect to the metric $g_{\Ds,I,\eta}$.

\begin{remark}\label{rem:NONinvTANNO}
    Formula \eqref{eq:TannoFormula} resembles the weighted scalar curvature introduced by the first author \cite{Lahdili} for a weight function associated to a Sasaki manifold \cite{AC}. However, note that here the function $f$ is not in general the pullback by a moment map of a function on a moment polytope.
    
    Assume for a moment that this is the case, i.e.\ that we are in the setting of \cite{AC,ACL}: $(\Ds,I,\eta)$ is a Sasaki structure with regular Reeb vector field $\xi$, inducing a holomorphic circle action $\bS_\xi^1$ for which the K\"ahler quotient $(M=N/\bS_\xi^1, \omega_\eta)$ is smooth, and $\eta_f=f^{-1}\eta$ is a contact form for which $f : N \ra \R$ is a $\xi$-invariant \emph{Killing potential} (i.e.\ the associated contact vector field $\Reebmap^{\eta_f}$ \S\ref{ss:PreliminairiesCONTACT} is CR). Thus, $\Reebmap^{\eta_f}\in \mbox{Lie}(\bT)$ where $\bT$ is a compact torus leaving $(\Ds,I,\eta)$ invariant and $f$ is a Hamiltonian function on $(M,\omega)$ for the induced vector field $\underline{\Reebmap^{\eta_f}} \in \mbox{Lie}(\bT/\bS_\xi^1)$. Then, putting $\mathrm{v}(x) = x^{-n-1}$ the direct calculation gives
    \begin{equation}
        \scal_{\mathrm{v\circ f}}(\omega_\eta)= f^{-n-2} \scal(\eta_f)/2
    \end{equation}
    where $\scal_{\mathrm{v}}(\omega)$ is the weighted scalar curvature of \cite{Lahdili}. Note that since we are using Convention \ref{conventionTK} for the metric, the scalar curvatures of the transversal K\"ahler structures and the one of the Tanaka-Webster connections are related by $\pi^*\scal(\omega_\eta) =\scal(\eta)/2$.
\end{remark}

Let $\dim N= 2n+1$, then the form
\begin{equation}\label{eq:volform}
    \eta_f\wedge (d\eta_f)^{[n]} =f^{-1-n}\eta\wedge (d\eta)^{[n]}
\end{equation}
is, up to a multiplicative constant (depending only on the dimension), the volume form of the Riemannian metric $g_{\Ds,I,\eta_f}$. Recall that for any $2$-form $\vartheta$, we set $\vartheta^{[k]}=\vartheta^k/k!$. Using \eqref{eq:volform} we get that the \emph{total Tanaka-Webster scalar curvature} of $(N,D,I,\eta_f\coloneqq f^{-1}\eta)$, defined by
\begin{equation}
    \bfS_f \coloneqq  \bfS(\Ds,I,\eta_f) \coloneqq  \int_N \scal(\eta_f)\, \eta_f\wedge (d\eta_f)^{[n]}
\end{equation}
can be written, as noticed in \cite{JerisonLee}, as
\begin{equation}\label{eq:positiveTOTcscK}
    \bfS(\Ds, I,f^{-1}\eta) =\int_N \left(f^{-n} \scal(\eta)   + (n+1)n f^{-n-2}|df|^2_\eta\right) \eta\wedge (d\eta)^{[n]}.
\end{equation}

The volume of $(N,\Ds,I,\eta_f)$ will be denoted $\bfV(\eta_f)$ or $\bfV_f$ if no confusion is possible, that is $\bfV_f = \int_N \eta_f\wedge (d\eta_f)^{[n]}$. The Einstein-Hilbert functional is then
\begin{equation}\label{eq:EH_definition}
\EH(f)\coloneqq \EH(D,I,\eta_f)\coloneqq  \frac{\bfS_f}{\bfV_f^{\frac{n}{n+1}}},
\end{equation} and
its variation along a conformal transformation is
\begin{equation}\label{eq:Var_EH}
\left(\frac{d}{ds} \EH_{f_s}\right)_{|_{s=0}} =  -\frac{n}{\bfV_f^{\frac{n}{n+1}}} \int_N \frac{\dot{f}}{f} \left(\scal(\eta_f)-\frac{\bfS_f}{\bfV_f}\right) \eta_f\wedge (d\eta_f)^{[n]}.
\end{equation} 
Given any $\tilde{\eta}=\eta_f\in \Gamma_+(\Ds^0)^{\xi}$, the TW scalar curvature of the CR-contact structure, say $(\Ds,I,\tilde{\eta})$, associated to $\tilde{\eta}$ and $(\uI,\xi)$, is invariant by the torus induced by $\xi$ since the whole structure is invariant. Therefore, the TW scalar curvature $\scal(\tilde{\eta})$ is constant if and only if \eqref{eq:Var_EH} vanishes for all $\dot{f} \in C^\infty(N,\R)^\xi$. From this, the following slight generalization of Jerison-Lee's result \cite{JerisonLee} holds.

\begin{corollary} \label{coroJLcscTW}
    Let $(N, \Ds, I)$ be a strictly pseudoconvex CR manifold and a compatible contact $1$-form $\eta \in \Cmet(\Ds,I)$ with Reeb vector field $\xi=\Reebmap^\eta \in \con_+(N, \Ds)$. For any $\tilde{\eta}\in \Cmet(\Ds,I)^{\xi}$, $(N,\Ds,I,\tilde{\eta})$ has constant Tanaka-Webster scalar curvature if and only if $\tilde{\eta}$ is a critical point of the Einstein-Hilbert functional on $\Cmet(\Ds,I)^{\xi}$.
\end{corollary}

\begin{remark}
    Given a strictly pseudoconvex CR manifold $(N,D,I)$, the \emph{CR Yamabe problem} refers to the existence of compatible contact forms  with constant Tanaka-Webster scalar curvature, i.e.\ critical points of $f\mapsto\EH(f^{-1}\eta)$ for a \emph{fixed} CR contact form $\eta$. As mentioned above, this problem was first considered by Jerison and Lee \cite{JerisonLee,JerisonLee2}, with important contributions by Gamara and Yacoub \cite{gamaran1,gamarayacoub_cryamabeFLAT}. Note that their results do not directly apply to our case, as we only consider conformal factors that are invariant under a certain $\mathbb{S}^1$-action. There is however a solution of this \emph{equivariant} CR Yamabe problem \cite{Zhang_Kcontact}, showing that $\EH$ restricted to $\Cmet(\Ds,I)^\xi$ always admits a minimum. In fact, the results of \cite{Zhang_Kcontact} are even more general and apply to the non-integrable, non-regular case as well.
\end{remark}

\section[The bundle of contact CR structures]{The bundle of contact CR structures defined by a polarized manifold}\label{s:Cmet}

\subsection{Contact CR structures on a transversal holomorphic manifold}

In this section we consider a $2n+1$ dimensional transversal holomorphic manifold $(N,\uI,\xi)$ of Sasaki type. By definition, this means that $\xi$ is a nowhere-vanishing vector field and $\uI$ is an integrable almost complex structure on the local quotient by the flow of $\xi$, such that the following space of compatible Sasaki structures is not empty 
\begin{equation*}
\Sas(N,\uI,\xi)\coloneqq  \left\{ (\Ds,J, \eta)\,\middle|\, 
\begin{array}{c}
    (\Ds,J, \eta) \mbox{ is Sasaki with Sasaki-Reeb vector field } \xi\mbox{ and}\\[-.5pt]
    \mbox{the quotient map } \Ds \ra TN/\langle\xi\rangle \mbox{ intertwines } J \mbox{ and } \uI
\end{array}  \right\}.
\end{equation*}
Elements of $\Sas(N,\uI,\xi)$ are caracterized by their contact form. Indeed, given a $1$-form $\eta\in \Omega^1(N)^\xi$ such that $\eta(\xi)$ never vanishes, there exists a unique endormorphism, say $J^\eta$, lifting $\uI$ on $\Ds^\eta\coloneqq \ker \eta$. Now $(\Ds^\eta,J^\eta, \eta)$ is in $\Sas(N,\uI,\xi)$ if and only if $\eta$ is a contact form with Reeb vector field $\xi$ (i.e.\ $\eta\wedge (d\eta)^n>0$, $\eta(\xi)=1$, $\mL_\xi\eta=0$) such that $d\eta$ (which is thus $\xi$-basic) is of type $(1,1)$ with respect to $\uI$ on the local quotient. Alternatively, to check that $(\Ds,J)=(\Ds^\eta,J^\eta)$ is CR one needs to verify that for any sections $A,B \in \Gamma(\Ds)$ we have
\begin{equation}\label{eqCRcondition}
    [A,B]- [J A,J B] \in \Gamma(\Ds) \; \mbox{and} \;\; [A, J B]- [J A,B] = J([A,B]- [J A,J B]).  
\end{equation}
The second condition is ensured by the fact that $\uI$ is locally integrable as an almost complex structure on the local quotient by the flow of $\xi$, while the first is equivalent to the fact that $\eta([A,B])= \eta ([J A,J B])$ or, equivalently, that $d\eta(A,B)= d\eta(JA,JB)$. That is, $d\eta_{|_{\Ds}}$ is $(1,1)$ with respect to $J$. 
   \begin{notation}\label{notationSASsubsetFORMS} In view of the previous discussion, we consider $\Sas(N,\uI,\xi)$ as a subset of $\Omega^1(N)^\xi$ where for $\eta \in \Sas(N,\uI,\xi)$ the corresponding Sasaki structure is $(\ker \eta, J^\eta, \eta)$ where $J^\eta \in \End(\ker\eta)$ is the unique lift of $\uI$ to $\ker \eta$.  
  \end{notation}
  
Recall from \cite[ch. 8]{BG:book} and \cite[\S 2]{ACL} that picking an element $\eta_o \in \Sas(N,\uI,\xi)$, and denoting $\Omega^1_{\xi,{\rm cl}}(N)$ the space of closed $\xi$-basic $1$-forms on $N$, we have
\begin{equation}
    \Sas(N,\uI,\xi) = \left\lbrace \eta_o+d^c_\xi \phi +\alpha \mid \alpha \in \Omega^1_{\xi,{\rm cl}}(N), d(\eta_o+ d^c_\xi\phi) (\cdot, \uI \cdot )>0 \right\rbrace,
\end{equation}
and we define a slice of the action of $\Omega^1_{\xi,{\rm cl}}(N)$ on $\Sas(N,\uI,\xi)$ by
\begin{equation}
    \holSAS(N,\uI,\xi,[\eta_o]) \coloneqq \{ \eta_o+d^c_\xi \phi \,|\,  d(\eta_o+ d^c_\xi\phi) (\cdot, \uI \cdot )>0 \}.
\end{equation}

We consider the following space of CR-contact structures
\begin{equation} \label{eqSliceCmetDEFN}
\begin{split}
    \Cmet (N,\uI,\xi, [\eta_o]) &\coloneqq \{ (\Ds,I, f^{-1}\eta )\,|\, \eta = \eta_\Ds^\xi \in \holSAS(N,\uI,\xi,[\eta_o]),\, ,\,  f\in C^\infty(N,\R_{>0})^\xi  \} \\
&= \left\{ (\Ds,I, \eta ) \,\left|\, (\Ds,I, \eta) \mbox{ compatible with } (\uI,\xi), \,\frac{\eta}{\eta(\xi)} \in \holSAS(N,\uI,\xi,\eta_{o})\right.\right\}
\end{split}
\end{equation}
which only depends on the slice $\holSAS(N,\uI,\xi,[\eta_{o}])$ in which $\eta_o$ lies. More generally, we introduce the following space
\begin{equation}
\begin{split}
\Cmet (N,\uI,\xi) &\coloneqq \{ (\Ds,I, f^{-1}\eta )\,|\, \eta = \eta_\Ds^\xi \in \Sas(N,\uI,\xi),\,,\,  f\in C^\infty(N,\R_{>0})^\xi  \} \\
& = \left\{ (\Ds,I, \sigma ) \,\left|\, (\Ds,I, \sigma) \mbox{ compatible with } (\uI,\xi), \,\frac{\sigma}{\sigma(\xi)} \in \Sas(N,\xi, \uI)\right.\right\}.
\end{split}
\end{equation}

\begin{remark}
    Again, we denote a structure $(\Ds,I, \eta) \in \Cmet (N,\uI,\xi, [\eta_o])$ by its contact form, i.e.\ $\eta \in \Cmet (N,\uI,\xi, [\eta_o])$ means $(\ker \eta,I, \eta)\in\Cmet (N,\uI,\xi, [\eta_o])$ with $I$ lifting $\uI$. Since $(\uI,\xi)$ is fixed, this should not cause any confusion. Therefore, we see $\Cmet(N,\uI,\xi, [\eta_o])$ as a subset of $\Omega^1(N)^\xi$, the space of $\xi$-invariant $1$-forms on $N$.
\end{remark}

 There is a natural submersion making $\Cmet(N,\uI,\xi, [\eta_o])$ into a (trivial) bundle over $\holSAS(N,\uI, \xi, [\eta_o])$ with fibre $C^{\infty}(N,\R_{>0})^\xi$, that is   
 \begin{equation}\label{eq:BdlZoverXi}\begin{array}{rccl}
  \kappa:& \Cmet(N,\uI,\xi, [\eta_o]) &\longrightarrow & \holSAS(N,\uI, \xi,[\eta_o])\\
 & \tilde{\eta}&\mapsto & \tilde{\eta}(\xi)^{-1}\tilde{\eta}. 
 \end{array}\end{equation} Given $\eta \in  \holSAS(N,\uI, \xi,[\eta_o])$, $\kappa^{-1}(\eta) = \Gamma((\ker \eta)_+^0)^\xi$ and each $\tilde{\eta}\in \kappa^{-1}(\eta)$ gives the same contact structure $\ker \tilde{\eta}= \ker \kappa(\tilde{\eta})$ and, thus, the same CR-structure lifting $\uI$.

 \begin{notation}
    Our main purpose in this paper is to study the functionals $\bfV,\bfS$, and $\EH$ defined in \S\ref{sec:TWconnection}, on the space $\Cmet(N,\uI,\xi)$. As $(N,\uI,\xi)$ is fixed, for $\sigma \in \Cmet(N,\uI,\xi)$ and $F= \bfV,\bfS \mbox{ or } \EH$, there will be no confusion to write $F(\sigma)$ for $F(\ker \sigma, I^{\sigma}, \sigma)$.      
 \end{notation}

The action of $\Omega_{\xi,{\rm cl}}(N)$ on $\Sas(N,\uI,\xi)$ leaves invariant the transversal K\"ahler structure induced by elements of $\Sas(N,\uI,\xi)$, and for any $\eta\in \Sas(N,\uI,\xi)$ wedging with $(d\eta)^n$ kills $\xi$-basic forms. Based on this, one can deduce this basic result.
\begin{lemma}\label{lem:RestrictSlice}\cite{BG:book,FOW}
    The functionals $\bfV, \bfS$ and $\EH$ are constant along orbits of the action of the additive group $\Omega_{\xi,{\rm cl}}(N)$ on $\Cmet(N,\uI,\xi)$ given by $(\alpha, \sigma) \mapsto \sigma + \sigma(\xi)\alpha$.   
\end{lemma}
Hence, we will mostly restrict our attention to the slices $\Cmet(N,\uI,\xi,[\eta_o])$.

\subsubsection{The Sasaki geometry of a polarized manifold}\label{sss:SASpolarizedManif}

Throughout this paper, $M= (M,\uI)$ is a compact (smooth) complex manifold and $L= (L, \hat{I})$ is a holomorphic ample line bundle over $M$, with $\pi \colon L\ra M$ the holomorphic bundle map. The letters $\uI$ and $\hat{I}$ denote respectively the integrable almost complex structures on $M$ and on the total space of $L$. It will not be confusing to denote as well $\pi \colon L^{-1}\ra M$ the dual holomorphic bundle map or its restriction to any of its subsets, and $\hat{I}$ also denotes the integrable almost complex structure on $L^{-1}$, induced by the one on $L$. Also, we identify the $0$-sections of these bundles with the manifold $M$ itself.

Any positively curved Hermitian metric $h$ on $L$, provides a Hermitian metric on $L^{-1}$ and a norm function $r_h\colon L^{-1}\ra\R_{\geq 0}$ so that we can define a Sasaki manifold as
\begin{equation*}
\left(N_h\coloneqq r_h^{-1}(1),\ \Ds^h\coloneqq TN_h\cap \hat{I}TN_h,\ \hat{I}_{\Ds^h},\ \eta_h\coloneqq\iota_{N_h}^*d^c\log r_h,\ \xi\right),
\end{equation*}
see for example \cite[Chapter 9]{BG:book}. Here $\iota_{N_h}$ is the inclusion of $N_h$ in $L^{-1}$ and $\xi$ is the generator of the fiberwise circle action on $\pi\colon L^{-1}\ra M$, the form $\eta_h$ is contact, $\Ds^h= \ker \eta_h$ and $\hat{I}_{\Ds^h}\in \End(\Ds^h)$ is the restriction of $\hat{I}$ to ${\Ds^h}$.

The construction depends only on the Hermitian metric $h$ and the complex structure on the total space of $L$. As in \cite[\S 2]{ACL}, when varying the Hermitian metric on $L$ it is often convenient to compare the Sasaki structures obtained on the common quotient by the radial $\R_{>0}$-action $N\coloneqq (L^{-1}\setminus M)/ \R_{>0}$ where, when fixing a Hermitian metric $h$, any other (positively curved) Hermitian metric $\tilde{h}=e^{\phi}h$, with $\phi\in \pi^*C^{\infty}(M,\R)$ gives a Sasaki structure
\begin{equation*}
    \left(N,\ \Ds^{\tilde{h}}= \ker  ( d^c\log r_h + d^c_\xi\phi),\  \hat{I}_{\Ds^{\tilde{h}}},\ \eta_{\tilde{h}} =d^c\log r_h + d^c_\xi\phi,\  \xi\right)
\end{equation*}
where $d^c_\xi\phi = - d\phi \circ I^\xi$ and $I^\xi\in \End(TN)$ is defined by the conditions $I^\xi(a) = \hat{I}a$ if $a\in \Ds^h$ and $I^\xi(\xi)=0$. Note that the CR structures obtained as above all induce the same transversal holomorphic structure, say $(\uI, \xi)$, on $N$. As the notation suggests, the transversal holomorphic structure here is essentially the complex structure on $M$.

\smallbreak

We denote by $\holSAS(M,L)$ the set of Sasaki structures on $(N,\uI, \xi)$ obtained when varying the positively curved Hermitian metrics on $L$. Elements in $\holSAS(M,L)$ are determined by their contact form, since the holomorphic transversal structure is fixed, and we consider $\holSAS(M,L)$ as a subset of the $\xi$-invariant $1$-forms $\Omega^1(N)^\xi$.

The space $\holSAS(M,L)$ is connected, \cite[ch. 8]{BG:book}, \cite[\S 2]{ACL}, and for $\eta_o \in \holSAS(M,L)$ $$\holSAS(M,L)\simeq \{ \eta_o+d^c_\xi \phi\,|\, d(\eta_o+ d^c_\xi\phi) (\cdot, \uI \cdot )>0 \}=\colon \holSAS(N,\uI,\xi,[\eta_o])$$ for $(N,\uI)$ determined as above, $N=(L^{-1}\setminus M)/ \R_{>0}$. That is, given $\eta_1 \in \holSAS(M,L)$ there exists a unique $\alpha\in \Omega^1_\xi(N)$ such that $\eta_1-\alpha \in \holSAS(N,\uI,\xi,[\eta_o])$. Therefore, there is a one to one correspondence between $\holSAS(M,L)$ and the space of K\"ahler metrics in $2\pi c_1(L)$ since for any $\eta \in \holSAS(M,L)$ there exists a unique $\omega \in 2\pi c_1(L)$ such that $$\pi^* \omega= d\eta.$$

\begin{notationremark}\label{notrem:holSas}
    Conversely, given a Sasaki manifold $(N,\Ds, I,\eta_o, \xi)$ which is quasi-regular, i.e.\ when the orbits of the flow of $\xi$ are closed and thus $\xi$ is the generator of a circle action $\bS^1_\xi$, then $\holSAS(N,\uI,\xi,[\eta_{o}])$ is in bijection with the K\"ahler metrics lying in the class $[\omega_0] \in H^{1,1}(M,\R)$ where $\pi^*\omega_0=d\eta_o$, on the orbifold K\"ahler quotient $M=N/\bS^1_\xi$. From the previous discussion, we get that $\holSAS(N,\uI,\xi,[\eta_{o}])$ is parametrised by the positively curved Hermitian metrics on a holomorphic ample line bundle $(M,L)$. Therefore, when $\xi$ is quasi-regular, we denote $\holSAS(M,L)=\holSAS(N,\uI,\xi,[\eta_{o}])$ where it should be understood that the holomorphic structure on $L$ is fixed. In view of this correspondence, we denote $\Cmet(M,L)=\Cmet(N,\uI,\xi, [\eta_o])$.
\end{notationremark}

By Lemma~\ref{lem:RestrictSlice}, Theorem \ref{theoWEBSTER} and equation \eqref{eq:TannoFormula}, for any of the functionals $F=\bfV,\bfS$ or $\EH$ and for any $\sigma \in \Cmet(M,L)$, the value $F(\sigma)$ only depends on the function $f=\sigma(\xi)^{-1}$ and the transversal K\"ahler structure $( \pi^*\omega =d (f\sigma), \uI)$. Therefore, for convenience, we often consider the functionals we study as functional on pairs $(\omega,f)$ such that $\omega$ is a K\"ahler metric in $2\pi c_1(L)$ on the compact complex orbifold $(M,\uI)$ and $f\in C^\infty(M,\R_{>0})$. When this precision is needed we denote
\begin{equation}\label{eq:notationWidecheck}
    \widecheck{\bfV}(\omega, f) \coloneqq  \frac{1}{2\pi} \bfV(f^{-1}\eta),\;\; \widecheck{\bfS}(\omega, f) \coloneqq  \frac{1}{\pi} \bfS(f^{-1}\eta)\; \mbox{ and } \; \widecheck{\EH}(\omega, f) \coloneqq  \frac{\widecheck{\bfS}(\omega, f)}{\widecheck{\bfV}(\omega, f)^{n/n+1}}
\end{equation} where $\pi^* \omega= d\eta$ and the Convention~\ref{conventionTK} is taken into account. Of course, this makes sense only when $\xi$ is quasi-regular but using continuity with respect to the Reeb vector field of the various geometric quantities we study, it is often enough to treat the quasi-regular case.

\subsection{Variational formulas along the transversal directions}

The \emph{transversal directions} refer to the variations which are transversal to the fibre of the bundle map \eqref{eq:BdlZoverXi}, that is $\kappa: \Cmet(N,\uI,\xi, [\eta_o]) \longrightarrow  \holSAS(N,\uI, \xi, [\eta_o])$. In particular, for such transversal variation $t\mapsto\tilde{\eta}_t$ the contact structure $\ker{\tilde{\eta}_t}$ varies.

\begin{lemma}\label{lem:varTransvDir}
    Let $(N,\uI, \eta_o)$ be a compact Sasaki manifold with Sasaki-Reeb vector field $\xi$ and, for $\epsilon>0$, let $(-\epsilon,\epsilon) \ni t\mapsto \phi_t \in C^\infty(N,\R)^\xi$ be a smooth path with $\phi_0=0$ such that $\eta_{t} \coloneqq  \eta+ d^c_\xi\phi_t \in \Xi(\uI,\xi, [\eta_o])$ for all $t \in (-\epsilon,\epsilon)$. For any $f\in C^\infty(N,\R_{>0})^\xi$, denoting $\eta_{t,f} \coloneqq  f^{-1}\eta_t$, $\eta_f= f^{-1}\eta$ and $dv_f= \eta_f \wedge (d\eta_{f})^{[n]}$, we have 
    \begin{equation}\label{e:ScalvariationTRANS}
        \frac{d}{dt}\!\!\left(\bfS(\eta_{t,f})\right)_{t=0} = \!n\!\int_N\!\! f^{-2} \left(\nabla^-d f , \nabla^-d \dot{\phi}\right)_{\eta_f}dv_f +n\!\int_N \!\frac{f^{-2}}{2}\scal(\eta_{f}) (df,d\dot{\phi})_{\eta_f}dv_f
    \end{equation} 
    \begin{equation}\label{e:EHvariationTRANS}
        \frac{d}{dt}\!\!\left(\EH(\eta_{t,f})\right)_{t=0} =\frac{n}{\bfV_{\eta_f}^{\frac{n}{n+1}}}\int_N\!\! f^{-2}\left[ \left(\nabla^-d f , \nabla^-d \dot{\phi}\right)_{\eta_f} + \frac{1}{2}\left(\scal(\eta_{f}) - \frac{\bfS(\eta_f)}{\bfV(\eta_f)} \right) (df,d\dot{\phi})_{\eta_f}\right]dv_f
    \end{equation} 
    where, for any $\xi$-basic $1$-form $\alpha$, $\nabla^-\alpha$ denotes the $I$-anti-invariant part of the bilinear form $\nabla\alpha$, and $\nabla$ is the Tanaka-Webster connection of $(\ker \eta, I,\eta)$. 
\end{lemma}

\begin{remark}
    Assuming, for a moment, that $f$ is a Killing potential for $(N,\uI, \eta_o)$, the first term \eqref{e:EHvariationTRANS} vanishes identically. In particular, in the classical weighted cscK setting of \cite{Lahdili,Inoue_NA_PerelmanEntropy} (cf.\ Remark \ref{rem:NONinvTANNO}), Lemma \ref{lem:varTransvDir} shows that the critical points of $\EH$ coincide with the vanishing of the transversal Futaki invariant as proved in \cite{BHLTF3}. 
\end{remark}

\begin{proof}
    For simplicity and to rely on K\"ahler formalism we assume that $\xi$ is quasi-regular and thus associated to a polarized K\"ahler orbifold $(M,L, \omega)$. All the terms appearing in \eqref{e:ScalvariationTRANS} and in \eqref{e:EHvariationTRANS} are well-defined for more general CR-contact structure such that, at $t=0$, $\eta \in \Xi(\uI,\xi, [\eta_o])$ is Sasaki. Moreover, they depend continuously on the data (in particular on $\xi$) and the space of quasi regular vector fields is dense in the Sasaki-Reeb cone of a Sasaki manifolds, therefore the Lemma will follow from the quasi-regular case.

    We work on the K\"ahler quotient $(M,\omega)$ where we consider $f,\phi\in C^\infty(M,\R)$ with $f>0$ fixed and smooth variation of the form $\omega_t=\omega + t\,dd^c\phi$. Up to a positive constant, see the notation ~\eqref{eq:notationWidecheck}, the total Tanaka-Webster scalar curvature is  
    $$\widecheck{\bfS}(\omega_t,f) = \int_M f^{-n} (\scal_{\omega_t} +n(n+1)f^{-2}|df|_{\omega_t}^2 ) \omega_t^{[n]}$$ where $\scal_{\omega_t}$ is the scalar curvature of the K\"ahler structure $(\omega_t,\uI)$ on $M$. 

    We use the variational formulas in \cite[Chapter 3]{PGbook} and $|df|_{\omega}^2\omega^{[n]} = df\wedge d^cf \wedge \omega^{[n-1]}$. Taking the derivative of $\widecheck{\bfS}(\omega_t,f)$, we get 
    \begin{equation}\label{eq:derivStransv}
    \begin{split}
        \frac{d}{dt} \widecheck{\bfS}(\omega_t,f) =  &\int_M \left( -2(\nabla^-d)^* \nabla^-d \phi  +(d\scal_{\omega_t} ,d\phi)_{\omega_t} - n(n+1)f^{-2}(df\wedge d^cf, dd^c\phi)_{\omega_t} \right) \frac{\omega_t^{[n]}}{f^{n}}\\
        &\;\; -  \int_M   \left(\scal(\omega_t) +n(n+1)f^{-2}|df|_{\omega_t}^2\right)\Delta \phi \, \frac{\omega_t^{[n]}}{f^{n}}
    \end{split}
    \end{equation}

    First we simplify the terms where $\scal_\omega =\scal(\omega)$ appears. Using \eqref{eq:TannoFormula} we get 
    \begin{equation}\label{eqderivStransvScalterm}
    \begin{split}
        &\int_M f^{-n} (d\scal(\omega) ,d\phi) \omega_t^{[n]}  - \int_M f^{-n} \scal(\omega)\,\Delta \phi\,\omega_t^{[n]} \\
        &=\int_M  f^{-n-1}\left(f^{-1}\scal(f^{-1}\eta) +2(n+1)f^{-1}\Delta^{\omega} f +(n+1)(n+2) f^{-2}|df|^2_{\omega} \right)(d f,d\phi)) \omega^{[n]}.
    \end{split}
    \end{equation} 

    For the first term of \eqref{eq:derivStransv} we compute  
    $$-2 (\nabla^-d)^* \nabla^-d f^{-n} = 2n (\nabla^-d)^* \nabla^- f^{-n-1}d f = 2n (\nabla^-d)^*  [f^{-n-1} \nabla^-d f + (df^{-n-1} \otimes df)^{-}]$$
    where $(df^{-n-1} \otimes df)^{-}$ is the $\uI$-anti-invariant part. Then writing $h=f^{-n-1}$ and using the definition of $\delta$ with respect to an orthonormal frame we have 
    \begin{equation}
        2n(\nabla^-d)^*[(dh \otimes df)^{-}] = 2n\delta\delta[(dh \otimes df)^{-}] =n\delta[\Delta h df + (J\nabla h)\intprod\! dd^cf] 
    \end{equation}
    The remaining terms in \eqref{eq:derivStransv} are the third and the last, that give
    \begin{equation}\label{eq:derivStransv3}
    \begin{split}
        &- n(n+1)\int_M  [\delta^c\delta ({f^{-2-n}}df\wedge d^cf) + \Delta(f^{-2-n}|df|_{\omega}^2)] \phi\,  \omega^{[n]}\\
        & =-n\int_M  \delta[ (J\nabla h)\!\intprod\! dd^cf -\Delta f dh] \phi\,  \omega^{[n]}
    \end{split} 
    \end{equation}
    using the K\"ahler identity $-\delta^c=[\Lambda,d]$. So, the first, third and last terms of \eqref{eq:derivStransv} sum up \begin{equation}\label{eq:derivStransv4}
    \begin{split}
        &\int_M f^{-n} \left[ (-2(\nabla^-d)^* \nabla^-d \phi - n(n+1){f^{-2}}(df\wedge d^cf, dd^c\phi) -n(n+1)f^{-2}|df|^2\Delta \phi \right] \omega_t^{[n]} \\
        &= n\int_M \phi [2(\nabla^-d)^*  (f^{-n-1} \nabla^-d f) +\delta((\Delta f^{-n-1}) df + \Delta f df^{-n-1})] \omega^{[n]}.
    \end{split}
    \end{equation}
    The last term here cancels the last term of \eqref{eqderivStransvScalterm} since $$(\Delta f^{-n-1}) df + \Delta f df^{-n-1} = [-2(n+1)f^{-n- 2} \Delta f  -(n+1)(n+2)f^{-3-n}|df|^2] df.$$
    We obtain \eqref{e:ScalvariationTRANS} and \eqref{e:EHvariationTRANS} using also $\frac{d}{dt} \widecheck{\bfV}(\omega_t,f) = (n+1)\int_M f^{-n-2}(df,d\phi)\omega^{[n]}$.
    
    To get back to $N$ we express the inner products appearing in the last formulas in terms of the Riemannian metric associated to $(\eta_f, I)$, we observe that the bilinear forms and $1$-forms involved are $\xi$-basic. That is, they can be written locally in terms of dual basis of $\Ds=\ker\eta$ on which the Riemannian metric associated to $\eta_f$ is conformal to $(g_\eta)_{|_\Ds}$ with a ratio $f^{-1}$, thus on $k$-forms $(\alpha,\beta)_{\eta} = f^{-k}(\alpha,\beta)_{\eta_f}$. Finally $\pi^*g_\omega =2(g_\eta)_{|_\Ds}$ (cf.\ \eqref{eqRiemCRcontacMetric}) gives the formula.
\end{proof}

\subsection{The Einstein-Hilbert functional on \texorpdfstring{$\Cmet$}{Z(N,I,xi)} and cscS metrics}

Let $(N,\uI, \xi)$ be a transversal holomorphic manifold of Sasaki type such that $\xi \in \mbox{Lie} \bT$, for $\bT$ a compact torus.

\begin{theorem}\label{lem:CRITEH} ({\bf Theorem}~\ref{theoEH=MabOnZ})
Let $\tilde{\eta} \in \Cmet(N,\uI,\xi){{}^\xi}$ be a critical point of the Einstein-Hilbert functional $\EH: \Cmet(N,\uI,\xi){{}^\xi}\ra\R$. Then, $f\coloneqq \tilde{\eta}(\xi)^{-1} \in C^\infty(N,\R_{>0})^\xi$ is a Killing potential and $(\tilde{\eta}, \ker\tilde{\eta}, I^{\tilde{\eta}})$ is a Sasaki structure with constant transversal scalar curvature (cscS). Conversely, any cscS structure in $\Cmet(N,\uI,\xi){{}^\xi}$ is a critical point of $\EH$.
\end{theorem}

\begin{proof}
 If $\tilde{\eta} \in \Cmet(N,\uI,\xi){{}^\xi}$ is a critical point of the Einstein-Hilbert functional then $\tilde{\eta}$ is a critical point of the restriction of $\EH$ on the vertical direction in the fibre of $\kappa$ containing $\tilde{\eta}$, that is $\EH: \Gamma((\ker \tilde{\eta})^0_+)^\xi \ra \R$. Thus  $(\tilde{\eta}, \ker\tilde{\eta}, I^{\tilde{\eta}})$ is cscTW (i.e.\ $\scal(\tilde{\eta}) = \bfS(\tilde{\eta})/\bfV(\tilde{\eta})$) by Corollary~\ref{coroJLcscTW}. 
 
 Now, we consider the variation in the transversal directions and use the formula of Lemma~\ref{lem:varTransvDir}. Since $\tilde{\eta}$ is cscTW the second term in \eqref{e:EHvariationTRANS} is zero and since $\tilde{\eta} \in \mbox{crit}(\EH)$ the first term must vanishes for any variation $\dot{\phi}$. Thus $\nabla^-df \equiv 0$ where $\nabla^-d$ is the $I^{\eta}$-anti-invariant part of the Hessian of $f$ with respect to the transversal K\"ahler structure associated to $(I^{\eta},\eta = f\tilde{\eta})$. This implies that $\mL_{\nabla^{g_{\eta}f}} \uI=0$. Therefore the Reeb vector field of $\tilde{\eta}$, which is $\Reebmap^{\tilde{\eta}}= f\xi+\grad_{d\eta}f$ (where $\grad_{d\eta}$ is the unique section of $\ker \eta$ such that $d\eta(\grad_{d\eta}, b) =-df(b)$, $\forall b\in \ker\eta$) is a transversally holomorphic contact vector field. Thus $\Reebmap^{\tilde{\eta}}$ is a CR vector field and $(\tilde{\eta}, \ker\tilde{\eta}, I^{\tilde{\eta}})$ is Sasaki, see Definition/Proposition~\ref{defnPropDEFsasak}.   
\end{proof}

\begin{remark}\label{rmk:Perelmanentropy}
    The Einstein-Hilbert functional has some similarities with Perleman's entropy functional considered in \cite{inoue-ent}. More precisely, the \emph{normalised W-entropy} $\check{\mathcal{W}}^\lambda$, for $\lambda=0$ is defined \cite[\S$1.2.1$]{inoue-ent} as, up to a normalisation,
    \begin{equation}\label{eq:Perelmanentropy}
        \check{\mathcal{W}}^0(\omega,\varphi)=-\frac{1}{2}\frac{\int_M\mathrm{e}^\varphi\left(\scal(\omega)+\lvert d\varphi\rvert^2_\omega\right)\omega^n}{\int_M\mathrm{e}^\varphi\omega^n}.
    \end{equation}
    The variables in \eqref{eq:Perelmanentropy} are a K\"ahler metric $\omega$ in a fixed class, and a smooth function $\varphi$ on $M$. Now, interpreting the Einstein-Hilbert energy on $M$ as we did in this section we have
    \begin{equation}\label{eq:EH_changevariables}
        \widecheck{EH}(\omega,\mathrm{e}^{-\varphi/n})=\frac{\int_M \mathrm{e}^{\varphi}(\scal(\omega) +\frac{n+1}{n}\lvert d\varphi\rvert_{\omega}^2 ) \omega^{[n]}}{\left(\int_M\mathrm{e}^{\frac{n+1}{n}\varphi} \omega^{[n]}\right)^{\frac{n}{n+1}}}.
    \end{equation}
    The two functionals \eqref{eq:Perelmanentropy} and \eqref{eq:EH_changevariables} differ in the exponent of the denominator, and they differ in the curvature part by the weighted $L^2$-norm factor $n^{-1}\int\mathrm{e}^\varphi\lvert d\varphi\rvert^2\omega^{[n]}$. Of course, this makes the Euler-Lagrange equations of the two functionals quite different: the critical points in the $\varphi$-direction of the Einstein-Hilbert functional are cscTW CR-contact forms, while for the W-entropy, critical points in the $\varphi$-direction are solutions of (compare this with \eqref{eq:TannoFormula})
    \begin{equation}
        \scal(\omega)+\Delta\varphi-\lvert d\varphi\rvert^2=\mathrm{const}.
    \end{equation}
    Moreover, critical points of the $\EH$-functional on the whole space are cscS metrics, while critical points of the $W$-entropy are $\mu$-cscK metrics. However, the two functionals share some important properties: in both cases $(\omega,0)$ is a critical point for the vertical direction if and only if $\omega$ is cscK.
\end{remark}

\begin{remark}\label{rmk:EinsteinHilbert_weighted}
Consider, for real numbers $p,q$, a K\"ahler metric $g$, and a positive function $f\in C^{\infty}(M,\R_{>0})$, the function
\begin{equation*}
    \mathrm{Scal}_{f^p}(g)=f^p\,\mathrm{Scal}(g)+2pf^{p-1}\Delta_g(f)-p(p-1)f^{p-2}\lvert df\rvert^2_g
\end{equation*}
and for a constant $c$ consider the equation
\begin{equation}\label{eq:Scal_pq}
    \mathrm{Scal}_{f^p}(g)=c\,f^q.
\end{equation}
Each equation \eqref{eq:Scal_pq} is the Euler-Lagrange equation, with respect to variations in $f$, of a \emph{$(p,q)$-Einstein-Hilbert functional}, defined as
\begin{equation}\label{eq:EH_pq}
    \EH^{p,q}(f,\omega)\coloneqq \frac{\int_M\mathrm{Scal}_{f^p}(g)\,f\,\omega^{[n]}}{\left(\int_Mf^{q+1}\,\omega^{[n]}\right)^{\frac{p+1}{q+1}}}=
    \frac{\int_M\left(f^{p+1}\mathrm{Scal}(\omega)+p(p+1)f^{p-1}\lvert df\rvert^2_\omega\right)\omega^{[n]}}{\left(\int_Mf^{q+1}\,\omega^{[n]}\right)^{\frac{p+1}{q+1}}}.
\end{equation}
Clearly, choosing $p=-(2n-1)$ and $q=-(2n+1)$ gives the classical Einstein-Hilbert functional, studied in the Riemannian Yamabe problem, while our CR-Einstein-Hilbert functional corresponds to the choice $p=-(n+1)$ and $q=-(n+2)$.

Following the computations for the CR case, the critical points $(f,\omega)$ of \eqref{eq:EH_pq} (with respect to variations in both $f$ and $\omega$) are characterised by the conditions
\begin{equation*}
    \begin{cases}
        f\mbox{ is a Killing potential for the metric defined by }\omega\\
        \mathrm{Scal}_{f^p}(g)=c\,f^q.
    \end{cases}
\end{equation*}
Notice that these are all examples of weighted cscK equations $\mathrm{Scal}_v(g)=w$, for the special choice of weights $v=f^p$, $w=f^q$. We expect that much of the techniques in our paper can be directly applied to the existence of critical points of these ``weighted Einstein-Hilbert functionals'' $\EH^{p,q}$. Notice however that the behaviour of \eqref{eq:Scal_pq} depends dramatically on the choice of $p$ and $q$: as an example, the equation corresponding to contact forms with constant Tanaka-Webster curvature is sub-critical while the Yamabe problem has critical Sobolev exponent.
\end{remark}

\section{The Einstein-Hilbert functional and test configurations}\label{s:EHtestconfig}

\subsection{Sasaki test configurations and their ribbons of contact CR-structures}\label{sec:test_conf}

The notion of test configuration for polarized varieties goes back to \cite{donaldson-toric,Tian} and has been thoroughly studied since. We refer to \cite{BHJ} for a recent account, although the perspective taken in this paper relies more on the differential-geometric approach of \cite{DR,zakarias}. We recall in this subsection the basic definitions and facts of this theory to fix some notation and highlight some aspects we will use below.   

\begin{definition}\label{def:TC}
Let $(M,L)$ be a polarized K\"ahler manifold.
 An \emph{ample test configuration} (of exponent $1$) over $(M,L)$ is the following set of data :
\begin{itemize}
    \item[(i)] a normal polarized variety $(\tstM, \tstL)$ together with a linearized $\C^*$-action on $(\tstM,\tstL)$, that we denote for $\tau \in \C^*$, by $\biho^\zeta_\tau : \tstL \ra \tstL$;
    \item[(ii)] a (flat) surjective holomorphic map $\TCmap : \tstM\ra \pr^1= \C\sqcup \{+\infty\}$ which is $\C^*$-equivariant with respect to the above action and the standard action on the target. For $\tau \in \pr^1$, we denote $M_\tau\coloneqq  \TCmap^{-1}(\tau)$;  
  \end{itemize}
satisfying the following conditions :
\begin{itemize}
    \item[(iii)] $\forall \tau \in \pr^1\backslash \{0\}$, $(M_\tau, \tstL_{|_{M_\tau}}) \simeq (M,L)$ as projective varieties; 
    \item[(iv)] the following (equivariant) biholomorphism 
    \begin{equation}\label{eq:biholTC}
    \begin{array}{rccl}
        \biho: & \tstM\backslash (M_0\sqcup M_{+\infty}) &\longrightarrow &M\times \pr^1\backslash \{0,+\infty\} \\
        & x &\mapsto & ( \biho^\zeta_{\TCmap(x)^{-1}}(x), \TCmap(x)) 
    \end{array}
    \end{equation}
    extends to $\biho: \tstM\backslash M_0  \longrightarrow  M\times \pr^1\backslash \{0\}$.
\end{itemize}
\end{definition}
The actions on $\tstL$ and on $\tstM$ are denoted by the same symbol and often considered as maps between the $\TCmap$ level sets, that is $\biho^\zeta_\tau : M_1 \ra M_\tau$, for $\tau \in \C^*$. Thanks to condition (iii), we set $(M,L)=(M_1,L_1)$, without loss of generality.

In this context, following Odaka~\cite{odaka} and Wang~\cite{wang}, the {\it Donaldson-Futaki} invariant of $(\tstM,\tstL)$, can be defined as an intersection product
\begin{equation}\label{eqDFinv}
\DF(\tstM,\tstL) \coloneqq  \frac{ n}{n+1} \cst_{M,L}\, c_1(\tstL)^{n+1} - (c_1(\tstM)- \TCmap^*c_1(\pr^1)).c_1(\tstL)^n 
\end{equation}
where $\cst_{M,L} \coloneqq  c_1(M).c_1(L)^{n-1}/ c_1(L)^n.$ The polarized manifold $(M,L)$ is said {\it K-semistable} if for any ample test configuration $(\tstM,\tstL)$ over it, we have $\DF(\tstM,\tstL) \geq 0$.

A test configuration $({\tstM},\tstL)$ is \emph{smooth} if $\tstM$ is smooth and \emph{dominant} if \eqref{eq:biholTC} extends to a $\C^*$-equivariant bimeromorphic morphism 
\begin{equation}\label{dominate}
 \Pi : {\tstM} \to M \times \mathbb{P}^1,
\end{equation}
that we still denote by $\Pi$. As explained in \cite{DR,zakarias}, see also \cite{BHJ}, in terms of $K$-(semi)stability it is not restrictive to consider dominant and \emph{smooth} ample test configuration with a reduced central fibre.

\smallskip

Following \cite{zakarias}, for a smooth dominant test configuration $({\tstM},\tstL)$, we have
\begin{equation}
    \tstL=\Pi^* {\proj_1}^* L+D,
\end{equation}
for a unique $\mathbb{Q}$-Cartier divisor supported on the central fiber ${\tstM}_0$ and where ${\proj_1}$ is the projection on the first factor. If $D=\sum_j a_j D_j$ is the decomposition of $D$ into irreducible components $D_j\coloneqq\{f_j=0\}$, where $f_j$ are the local defining functions, then by the Poincar{\'e}-Lelong formula, there exist an $\mathbb{S}^1$-invariant Green function $\gamma_D$ on $\tstM\setminus\tstM_0$ defined up to adding a smooth function by $\gamma_D\coloneqq \sum_j a_j\log|f_j|$, so that the current of integration $\delta_D$ along $D$ satisfies
\begin{equation}
   \delta_D=dd^c \gamma_D+\Theta_D 
\end{equation}
where $\Theta_D$ is a smooth $\mathbb{S}^1$-invariant representative of the fundamental class $[D]$ of $D$.

Let $\mathbb{D}\subset \mathbb{P}^1\setminus\{\infty\}=\mathbb{C}$ be the unit disk centered at the origin and $(\varphi_t)_{t>0}$ a ray of locally bounded $\omega$-psh functions on $(M,\omega)\cong {M}_1$, such that the corresponding $\mathbb{S}^1$-invariant function $\Phi(\cdot,\tau)=\varphi_{-\log|\tau|}$ on $M\times\mathbb{D}^*$ is a ${\proj_1}^*\omega$-psh function on $M\times\mathbb{D}^*$. Following the terminology of \cite[section 4.3]{zakarias}, the ray $(\varphi_t)_{t\geq 0}$ is \emph{$C^{1,1}$-compatible} with the test configuration $({\tstM},\tstL)$ if $\Pi^*\Phi+\gamma_D$ is $C^{1,1}$ on $\pi^{-1}(\mathbb{D})$.

Any smooth dominant test configuration uniquely determines a $C^{1,1}$-compatible geodesic ray \cite[Theorem 1.2]{CTW2}, unique once a starting point $\omega_{\varphi_0}$ is fixed. We recall the construction from \cite[Lemma 4.6]{zakarias}: assuming that $\tstL$ is ample, we can take an $\mathbb{S}^1$-invariant K\"ahler form $\Omega\in 2\pi c_1(\tstL)$ such that $\Omega=\Pi^*{\proj_1}^*\omega+\Theta_D$. By \cite[Corollary 1.3]{CTW1}, on the compact manifold with non-empty boundary ${\tstM}_{\mathbb{D}}\coloneqq \TCmap^{-1}(\mathbb{D})$ the boundary value problem
\begin{equation}\label{BVP:G}
\begin{cases}
(\Omega+dd^c\Gamma)^{n+1}=0\\
\Gamma_{|\partial {\tstM}_{\mathbb{D}}}=\varphi_0+\gamma_D
\end{cases}    
\end{equation}
admits a unique $\mathbb{S}^1$-invariant solution $\Gamma\in C^{1,1}({\tstM}_{\mathbb{D}})$. Letting $\Phi$ be the function on $M\times\mathbb{D}^{*}$ given by $\Pi^*\Phi\coloneqq \Gamma-\gamma_D$, we get $\Pi^*({\proj}_1^*\omega+dd^c\Phi)=\Omega+dd^c\Gamma-\delta_D$, and taking the $(n+1)$-wedge power of both sides shows that $\Phi$ solves
\begin{equation}
    \begin{cases}
        (\Omega+dd^c\Phi)^{n+1}=0 \quad \text{on }M\times\mathbb{D}^*\\
        \Phi(\cdot,1)=\varphi_0.
    \end{cases} 
\end{equation}
Hence $\Phi$ is the $\mathbb{S}^1$-invariant function corresponding to a $C^{1,1}$ geodesic ray $(\varphi_t)_{t\geq 0}$, which is $C^{1,1}$ compatible with the test configuration $({\tstM},\tstL)$.
 
\subsubsection{Sasaki test configurations}\label{ss:testconfigSASAK} As pointed out in \cite[\S 6]{ACL}, any smooth ample compact test configuration $(\tstM,\tstL)$ gives rise to a test configuration $(\tstY,\xi)$ of polarized cones over $(Y ,\xi)$ where $$Y = L^{-1}\backslash 0\mbox{-section and } \tstY= \tstL^{-1}\backslash 0\mbox{-section}$$ and $\xi$ denotes the Reeb vector field on both $Y$ and $\tstY$ which is induced by the $\bS^1$ action on the fibre of $L$ and $\tstL$ respectively. Indeed, the equivariant map $\tilde{\TCmap} : \tstY \ra \pr^1$ is the composition $\TCmap \circ \pi$ where $\pi : \tstL^{-1} \ra \tstM$ is the bundle map.
 
The test configuration $\C^*$-action includes a circle action on $\tstY$ whose generator will be denoted $\zeta$. We denote $\bT\coloneqq \bS^1_\xi \times \bS^1_\zeta$ the induced torus action and $\kt \coloneqq  \mbox{Lie}(\bT)$. Thus, the test configuration $\C^*$-action on $\tstM$ is induced by $\underline{\zeta}\coloneqq  [\zeta]\in \mbox{Lie}(\bT/\bS^1_\xi)$. 
 
\begin{notation}\label{notationSUBSCRPITtau}
    The subscript $\tau$ on a form, tensor, etc. means the pullback by the inclusion of the fibre over $\tau$ which, itself, is denoted $\iota_\tau : M_\tau \hookrightarrow \tstM$, $\iota_\tau : Y_\tau \hookrightarrow \tstY$, etc.  When needed, the inclusion of a subset $A\subset B$ could be denoted $\iota_A$. We will denote by $J$ the integrable almost complex structures of $\tstY$, $\tstL$, $L$, $\tstM$, $M$ and $Y$. Since one restricts to the other via the inclusion, this should not cause any confusion.
\end{notation}
 
In \cite[\S 6]{ACL}, given a smooth ample test configuration $(\tstM,\tstL)$ and a positively curved Hermitian metric on $\tstL$, encoded by its norm function $\tilde{r}: \tstY \ra \R_{>0}$, called cone or \emph{radial potential} \cite[Definition 2.2]{ACL}, the Sasaki manifold  
\begin{equation}\label{eq:DefsasakiTC}
    \left( \tstN = \{\tilde{r}=1\}, \tstD= T\tstN\cap J T\tstN, J =J_{|_\tstD}, \widetilde{\eta} \coloneqq  \iota_\tstN^*(d^c\log \tilde{r}), \xi \right)
\end{equation}
plays the role of a Sasaki test configuration over the Sasaki manifold
\begin{equation}
    \left( N = \{r=1\} , \Ds= TN\cap JTN, J=J_{|_\Ds}, \eta = \iota_N^*(d^c\log r), \xi \right).
\end{equation}
They are both regular Sasaki manifolds, over $\tstM= \tstN / \bS^1_\xi$ and $M= N / \bS^1_\xi$ respectively. Moreover, $N =\tilde{\TCmap}_{|_\tstN}^{-1}(1) \subset \tstN \subset \tstY$.

\smallskip

The key, although simple, observation we need is the following and it holds for (non-ample) smooth compact test configurations of the form $(\hat{\tstM},\hat{\pi}^*\tstL)$ as in the last Remark.  
\begin{lemma}\label{lem:RibboncontactCR}
    There exists $\epsilon >0$ such that $f_s\coloneqq \tilde{\eta}( \xi-s\zeta)$ is a \emph{positive} function on $\tstN$ for all $s\in (-\epsilon,\epsilon)$. In particular, \begin{equation}\label{eq:RibboncontactCR}
        \forall (s,\tau) \in  (-\epsilon,\epsilon) \times \pr^1\backslash \{0\}, \quad \iota^*_\tau(f_s^{-1}\widetilde{\eta}) \in \Cmet(M_\tau,L_\tau).
    \end{equation}
    We call the resulting $2$-parameters family a \emph{ribbon} of CR-contact structures.  
\end{lemma}
According to Notation~\ref{notationSUBSCRPITtau}, we let $\iota^*_\tau(f_s^{-1}\widetilde{\eta}) = f_{s,\tau}^{-1}\widetilde{\eta}_\tau$. The $\bT$-invariant function on $\tstN$
\begin{equation}
    f_s\coloneqq \tilde{\eta}( \xi-s\zeta)
\end{equation}
is the pullback of a function on $\tstM$, that we denote by the same symbol.
\begin{remark}
    Note that $f_s=1-s\mu$, where $\mu$ is a Hamiltonian function for the action of the vector field $\uzeta$ induced by $\zeta$ on $(\tstM,\Omega)$ for $\pi^*\Omega=d\tilde{\eta}$. Of course $f_s$ is positive as long as $s$ is smaller than $\mu_{\max}^{-1}$; this maximum is always achieved (on the fibre over $\infty\in\mathbb{P}^1$) as the test configuration is dominant, see also Lemma \ref{lemma:NORMtc} below. 
\end{remark}

This section is devoted to the study of the Einstein-Hilbert functional, the total scalar curvature and the volume along the ribbon of CR-contact structures $f_{s,\tau}^{-1}\widetilde{\eta}_\tau$ for a given test configuration $(\tstM,\tstL)$ and a chosen $\tilde{\eta}\in \holSAS(\tstM,\tstL)$. Of course, the values of these functionals on $\iota^*_\tau(f_s^{-1}\widetilde{\eta})$ depend on the choice of the contact form and on $s$ but we will fix such a pair $(\tilde{\eta},s) \in \holSAS(\tstM,\tstL) \times (-\epsilon,\epsilon)$ satisfying \eqref{eq:RibboncontactCR} for the remaining subsections and consider the functions  
\begin{equation}\label{eq:3functonRIBBONS}
    \pr^1\backslash \{0\} \ni \tau \longmapsto \EH(f_{s,\tau}^{-1}\widetilde{\eta}_\tau),\  \bfV(f_{s,\tau}^{-1}\widetilde{\eta}_\tau),\  \bfS(f_{s,\tau}^{-1}\widetilde{\eta}_\tau).
\end{equation}
One of the main results of this section, Theorem \ref{theo:Globalformula}, gives explicit expressions for the limits at $\tau=0$ of these quantities on ribbons of CR-contact structures obtained from smooth ample test configurations.  

\begin{definition}\label{def:Globalformula}
    Let $(\tstM,\tstL)$ be a smooth ample dominating test configuration for $(M,L)$ with reduced central fibre and generating vector field $\zeta$. Fix $\tilde{\eta}\in\holSAS(\tstM,\tstL)$ corresponding to the K\"ahler form $\Omega$ on $\tstM$, and let $\mu\coloneqq \tilde{\eta}(\zeta)$. For $0 < s \ll 1$, we define
    \begin{equation}
    \begin{split}\label{eq:EHinv}
        &\bfS_s(\tstM,\tstL) \coloneqq   \frac{2\pi c_1(M).c_1(L)^{[n-1]}}{(1-s\mu_{\max})^n} -2ns\int_{{\tstM}}f_s^{-n-1}(\Ric(\Omega)-\pi^*\omega_{\rm FS}) \wedge \Omega^{[n]}\\
    & \phantom{\bfS_s(\tstM,\tstL) \coloneqq } -n(n+1)s^2\int_{{\tstM}}f_s^{-n-2}(\Delta_{\Omega}(\mu)-\Delta_{\omega_{\rm FS}}(\mu_{\rm FS})) \Omega^{[n+1]}\\
        &\bfV_s(\tstM,\tstL) \coloneqq   \frac{(2\pi)^{n+1}}{(1-s\mu_{\max})^{n+1}} c_1(L)^{[n]} -  (n+1) s \int_{{\tstM}} f_s^{-n-2} \Omega^{[n+1]}\\
        &\EH_s(\tstM,\tstL) \coloneqq   \frac{\bfS_s(\tstM,\tstL)}{\bfV_s(\tstM,\tstL)^{\frac{n}{n+1}}}
    \end{split}
    \end{equation}
    where $\mu_{\max}$ is the maximum of $\mu$, seen as a function on $\tstM$.
\end{definition}

\begin{remark}
    These quantities do not depend on the choice of $\tilde{\eta}\in\holSAS(\tstM,\tstL)$, so that $\EH_s(\tstM,\tstL)$ in particular is an algebraic object. We will show this in two different ways: first in the proofs of Lemmas \ref{lem:AsymSlopeVol} and \ref{lem:AsymCurvPSI} where we show that these formulas can be expressed as some weighted cscK functional on the total space of the test configuration. The second way is through the algebraic characterisation of $\EH_s(\tstM,\tstL)$, given in Section \ref{sectionTOWalgebraic}, where we can also observe that these invariants are combinations of equivariant closed forms, evaluated on $\xi-s\zeta$, on the associated complex cones.
\end{remark}

\begin{lemma}\label{lem:derivEH=DF}
    Under the hypothesis of Definition~\ref{def:Globalformula}, we have 
    \begin{equation}
        \frac{d}{ds}\Bigr|_{s=0}\EH_s(\tstM,\tstL)=\frac{2n\,\DF(\tstM,\tstL)}{(2\pi)^{n+1}\bfV_0^{\frac{n}{n+1}}n!}.
    \end{equation}
\end{lemma}
\begin{proof}
    This is a straightforward consequence of the Odaka-Wang formula \eqref{eqDFinv},
   $$
    \left(\frac{d}{ds}\Bigr|_{s=0}\bfS_s(\tstM,\tstL)-\frac{n}{n+1}\frac{\bfS_0}{\bfV_0}\frac{d}{ds}\Bigr|_{s=0}\bfV_s(\tstM,\tstL)\right)=\frac{2n\,\DF(\tstM,\tstL)}{(2\pi)^{n+1}n!}.\qedhere
    $$
\end{proof}

\begin{theorem}\label{theo:Globalformula}
Let $(\tstM,\tstL)$ be a smooth ample test configuration for $(M,L)$ with reduced central fibre and generating vector field $\zeta$. Fix $\tilde{\eta}\in\holSAS(\tstM,\tstL)$ and let $\mu\coloneqq \tilde{\eta}(\zeta)$. Then, along a compatible ribbon of CR-structures,
\begin{equation}
    \lim_{\tau\to 0}\bfS(f_{s,\tau}^{-1} \eta_\tau)= \bfS_s(\tstM,\tstL),\quad
    \lim_{\tau\to 0}\bfV(f_{s,\tau}^{-1} \eta_\tau) = \bfV_s(\tstM,\tstL).
\end{equation}
In particular, for a positive dimensional constant $c_n$,
\begin{equation}
    \lim_{\tau\to 0}\EH(f_{s,\tau}^{-1} \eta_\tau)=\EH_s(\tstM,\tstL)=\EH_0+c_n\,s\,\DF(\tstM,\tstL)+O(s^2).
\end{equation}
\end{theorem}
Section \ref{ss:globalformulaEH} is dedicated to the proof of Theorem \ref{theo:Globalformula}.

\subsubsection{From ribbons of CR-structures to weighted K\"ahler functionals}

In the next subsections \ref{ss:RibbonGeod0} and \ref{ss:RibbonGeod1}, we work away from the central fibre and in general on the trivial locus $\TCmap^{-1}(\pr^1\backslash \{0,+\infty\}) \overset{\Pi}{\simeq} M\times \C^*$. By assumption the test configuration action is linearized and condition (iii) of Definition \ref{def:TC} implies that
$$(Y_\tau \coloneqq  (\TCmap\circ \pi)^{-1}(\tau), \xi_\tau)\simeq (Y,\xi)$$
as complex cones. Here $\xi_\tau$ is the fibrewise Reeb vector field of $\tstY$ along $Y_\tau$, this makes sense since this vector field is vertical with respect to the bundle map $\pi$. Condition (iii) combined with the biholomophism \eqref{eq:biholTC} already uses that $\mO(-1)$ is trivial over $\pr^1\backslash \{0\}$  and yields a biholomorphism
\begin{equation}\label{eqTCdecompRADIALcones}
    \tstY\backslash Y_0 \simeq Y\times \pr^1\backslash \{0\}
\end{equation}
which is equivariant with respect to the complex torus action $\C_\xi^*\times \C_\zeta^*$ on the left and $\C_\xi^*$ times the standard action on $\pr^1$.

Pulling back a $\xi$-radial potential of $(\tstY,\xi)$ using \eqref{eqTCdecompRADIALcones} it gives a family of radial potentials on $(Y,\xi)$ parametrized by $\tau \in \pr^1\backslash \{0\}$. By \cite[Proposition 2.9]{ACL}, the pullback has the form
\begin{equation}\label{eqTCdecompRADIALpot}
    \tilde{r}(y,\tau) = e^{\phi_t(y)}r(y)
\end{equation}
where as always $t\coloneqq-\log\lvert\tau\rvert$ and $\phi_t \in C^\infty(M)\simeq C^\infty(Y)^{\C^*_\xi}\simeq C^\infty(N)^{\bS^1_\xi}$. This implies that through (the restriction of) the above equivariant biholomophism, the contact CR manifold of \eqref{eq:RibboncontactCR},
$$ \left( N_\tau, \Ds_\tau, \CR_{|_{\Ds_\tau}}, \iota^*_\tau(\tilde{f}_s^{-1}\widetilde{\eta}) \right),$$
is sent to the contact CR manifold determined by \begin{equation}
    f_{s,t}^{-1}(\eta + d^c_\xi\phi_t) \in \Cmet(M,L),
\end{equation}
where $f_{s,t}$ is the pull-back (by the biholomophism \eqref{eq:biholTC}) of the restriction to $N_\tau$ of the function $f_s = \tilde{\eta}( \xi-s\zeta)$ and $\eta + d^c_\xi\phi_t \in \Xi(N,\uI, \xi, [\eta])$ is a Sasaki structure compatible with the holomorphic structure $(\uI, \xi)$ induced by $(Y,\xi)$ (equivalently $(M,L)$) on $N$, see \cite[\S 6]{ACL}. 

\begin{notation}
    In what follows, for any tensor $\alpha_{s,t}$ depending of the parameters $(t,s)\in \R\times (-\epsilon,\epsilon)$, we denote by $\dot{\alpha}_{s,t}$ the variation with respect to the variable $t$.
\end{notation}
 
\begin{lemma}\label{lemmaDERIVATIVE}
    $f_{s,t} = 1+ s\dot{\phi_t}$. 
\end{lemma}
\begin{proof}
    To clarify the argument, let us denote $\widetilde{J}$ the complex structure on $\tstY$ and consider, in the decomposition \eqref{eqTCdecompRADIALpot}, the function $\phi(\tau,y)= \phi_t(y)$ as a function on $Y\times \C^*$.  

    Observe that with the identification $\tstY \backslash Y_0 \simeq Y\times \C^*$, the vector field $\widetilde{J}\zeta$ is sent to $\frac{\partial}{\partial t}$. Therefore, since $\widetilde{\eta}$ is the restriction of $d^c \log \tilde{r} =  d^c \phi + d^c \log r$ on $\tstN$, we have that
    \begin{equation}
        \widetilde{\eta} (\zeta) = - d\phi(\widetilde{J}\zeta) - d\log r(\widetilde{J}\zeta) = - \dot{\phi_t}.\qedhere
    \end{equation}
\end{proof}

From the above discussion and notation \eqref{eq:notationWidecheck}, we conclude that studying the functionals \eqref{eq:3functonRIBBONS} boils down to the study of the functionals  
\begin{equation}\label{eq:3functonTCweightedK0}
    \R \ni t \longmapsto  \widecheck{\bfV}(\omega + dd^c\phi_t, (1+s\dot{\phi}_t)),\quad\widecheck{\bfS}(\omega + dd^c\phi_t, (1+s\dot{\phi}_t)).
\end{equation}
   
\subsubsection{The Einstein-Hilbert functional on ribbons of contact CR-structures}\label{ss:RibbonGeod0}

Consider a ribbon of contact forms $f^{-1}_{s,t}\eta_t= (1+s\dot{\phi}_t)^{-1}(\eta +d^c_\xi\phi_t)$, for a \emph{smooth} path of K\"ahler potentials $\phi_t$. The direct computation gives the following formulas for the variation of the total volume and total scalar curvature.
\begin{lemma}\label{lem:Vol-variation}
    For a path of K\"ahler potentials $\varphi_t$, the variation of the volume along the ribbon of CR-contact forms $f^{-1}_{s,t}\eta_t\coloneqq  (1+s\dot{\phi}_t)^{-1}(\eta +d^c_\xi\phi_t)$ satisfies       
    \begin{equation}\label{eq:smooth_variations}
    \begin{split}
        \frac{d}{dt}\bfV(f_{s,t}^{-1}\eta_t) =& - s(n+1) \int_N {f_{s,t}^{-n-2}} \left(\ddot{\phi_{t}}- \frac{1}{2} |d\dot{\phi}_t|^2_{\eta_t}\right)     \eta_t \wedge (d\eta_t)^{[n]}\\
        \frac{d}{dt}\bfS(f_{s,t}^{-1}\eta_t) =& sn \int_N f_{s,t}^{-n-1} \left( |\nabla^-d \dot{\phi}_t|^2_{\eta_t} - \frac{\scal_{s,t}}{ f_{s,t}}\left( \ddot{\phi}_t - \frac{1}{2} |d\dot{\phi}_t|^2_{\eta_t} \right)\right)\eta_t \wedge (d\eta_t)^{[n]}.
    \end{split}
    \end{equation}
    where ${\eta_t}$ as a subscript means that the quantity is computed with respect to the Riemannian metric \eqref{eqRiemCRcontacMetric} associated to $(\uI,\xi,\eta_t)$.    
\end{lemma}
These two computations show that, if $\varphi_t$ is a smooth K\"ahler geodesic, the volume functional is constant along the ribbon $(1+s\dot{\phi}_t)^{-1}(\eta +d^c_\xi\phi_t)$, while the total scalar curvature (hence, the Einstein-Hilbert functional) is non-decreasing, and it is in fact increasing unless $\varphi_t$ is a trivial geodesic, induced by the flow of a biholomorphism. However, to apply this result to study test-configurations, we first need to generalise them to non-smooth geodesics. This is essentially the content of the next subsection, in which we extending the above computation for the variation of the total scalar curvature, ultimately leading to the proof of Theorem \ref{theo:ActionPointwiseCvxC0}.

\subsection{The Einstein-Hilbert functional on ribbons of weak geodesics}\label{ss:RibbonGeod1}

Let $(M,\omega)$ be a compact K\"ahler manifold with a fixed reference K\"ahler metric $\omega$, and let $\mathbb{D}^{*}$ be the punctured unit disc. There is a natural correspondence between rays of K\"ahler potentials $\varphi_t$ on $(M,\omega)$ and functions $\Phi$ on $M\times\mathbb{D}^*$, invariant under the $\mathbb{S}^1$-action on $\mathbb{D}^*$, given by
\begin{equation}\label{S1_inv_funct}
     \Phi(x,\tau)=\varphi_t(x),\quad \tau=e^{-t+i\theta}.
\end{equation}
Let $\proj_1:M\times \mathbb{D}^*\to M$ be the projection on the first factor. We recall the definition of (sub-)geodesic rays in the space of K\"ahler potentials of $(M,\omega)$.
\begin{definition}
    A ray of functions $(\varphi_t)_{t\geq 0}\in {\rm PSH}(M,\omega)$ is a \emph{subgeodesic ray} if the associated $\mathbb{S}^1$-invariant function $\Phi$ is $\proj_{1}^*\omega$-plurisubharmonic function on $M\times \mathbb{D}^*$. A locally bounded ray $(\varphi_t)_{t\geq 0}\in {\rm PSH}(M,\omega)$ is called a \emph{weak geodesic ray} if it is a subgeodesic ray and satisfies $\big(\proj_1^*\omega+dd^{c}\Phi\big)^{n+1}=0$, in the sense of Bedford-Taylor.
\end{definition}

For a weak (sub)geodesic ray $(\varphi_t)_{t\geq 0} \in {\rm PSH}(M,\omega)$, the function $f_{s,t}\coloneqq 1+s\dot{\varphi}_t$ is positive for $s>0$ small enough, which gives rise to what we call a \emph{ribbon of weak geodesics} on $(N,D,J,\xi)$, $(f_{s,t})^{-1}\eta_{\varphi_t}$. Note that if $\varphi_t$ is $\mathcal{C}^{1,1}$ as a function on $X\times\R_+$, the derivative $\dot{\varphi}_t$ is uniformly bounded, so that $s$ can be chosen uniformly in $t$.

In fact, there is an explicit range for $s$ for which $f_{s,t}$ is positive.
\begin{lemma}\label{lemma:NORMtc}
    Let $\varphi_t$ be the $C^{1,1}$-geodesic ray associated to the test configuration $(\tstM,\tstL)$. Then, the function $f_{s,t}$ is positive as long as $\lvert s\rvert<\lVert(\tstM,\tstL)\rVert_{L^\infty}^{-1}$.
\end{lemma}
\begin{proof}
    Note that the function $f_{s,t}=1+s\dot{\varphi}_t$ is positive as long as $\lvert s\rvert < \lVert \dot{\phi}_t\rVert_{L^\infty}^{-1}$. When $\phi_t$ is a geodesic of K\"ahler potentials, it is known from \cite[eq. $(4.14)$]{Hisamoto_norm} (see also \cite[Lemma 2.1]{BoB}) that $\lVert \dot{\phi}_t\rVert_{L^\infty}$ does not depend on $t$ and coincides with the $L^\infty$ norm of the test configuration. See \cite{Hisamoto_norm} and also \cite[Definition 3.10]{DervanRelK} for the definition of $\lVert(\tstM,\tstL)\rVert_{L^\infty}$.
\end{proof}

One of the goals of this Section is to show how one can define the Einstein-Hilbert functional on a weak geodesic ribbons of contact forms. It is then natural to enlarge the spaces $\Xi$ and $\Cmet$ to include CR-contact forms with $ C^{1,1}$-regularity.

\begin{definition}
Given a smooth Sasaki manifold $(N,\Ds,J,\eta,\xi)$ compatible with a given transversally holomorphic structure $(\uI,\xi)$, we denote by $\Xi^{1,1}(N,\uI,\xi,[\eta])$ the slice of compatible $C^{1,1}$-Sasaki structures, that is
\begin{equation*}
    \Xi^{1,1}(N,\uI,\xi,[\eta]) \coloneqq  \{ \eta_\phi\coloneqq  \eta+ d^c_\xi \phi \,|\,  \phi\in C^{1,1}(N,\R)^\xi, \eta\wedge (d\eta_\phi)^n>0 \mbox{ (as a current)}\}.
\end{equation*}
Similarly, we can define
\begin{equation*}
    \Cmet^{1,1}(N,\uI,\xi,[\eta]) \coloneqq  \{ f^{-1}\eta_\phi\,|\, \eta_\phi \in \Xi^{1,1}(\uI,\xi,[\eta]),\, f\in C^{0,1}(N,\R_{>0})^\xi\},
\end{equation*}
and we have analogous versions of these spaces for $\Sas^{1,1}(N,\uI,\xi)$ and $\Cmet^{1,1}(N,\uI,\xi)$. When $(N,\uI,\xi)$ is the transversally holomorphic circle bundle over a polarized manifold $(M,L)$, we denote these spaces by $\Sas^{1,1}(M,L)$ and $\Cmet^{1,1}(M,L)$, respectively. 
\end{definition}
Notice that the contact forms in $\Cmet^{1,1}(N,\uI,\xi,[\eta])$ just have bounded coefficients on $N$. It is clear that the volume functional is well-defined on $\Cmet^{1,1}(N,\uI,\xi,[\eta])$. Extending the total scalar curvature functional to the space of less regular functions is instead much more delicate. Indeed, the total Tanaka-Webster scalar curvature may not be well-defined on the whole of $\Cmet^{1,1}$. However we can prove, in a slightly indirect way, that the $t$-primitive of the total scalar curvature functional along a ribbon of weak geodesics is a well-defined continuous and convex function, which will allow us to deduce important properties of $\bfS$. Given a ribbon of CR-contact forms $f_{s,t}^{-1}\eta_t$ we consider the $t$-primitive of $\widecheck{\bfS}$, i.e.\ the \emph{action functional}
\begin{equation}\label{eq:actionfunctional_def}
    \widecheck{\actscal}_s(t)\coloneqq \int_0^t \widecheck{\bfS}(\omega_{\varphi_u},f_{s,u})du.
\end{equation}
A similar functional appears in Inoue's work \cite{inoue-ent}. It is easy to see why $\widecheck{\actscal}_s(t)$ might be useful in the weighted-cscK context, when $f_{s,t}$ is a Killing potential: in this case,
\begin{equation*}
    \frac{d}{dt}\widecheck{\actscal}_s(t)=\int_M \scal_{f_{s,t}^{-(n+1)}}(\omega_{\varphi_t})\,\omega_{\varphi_t}^{[n]} +s\int_M \scal_{f_{s,t}^{-(n+1)}}(\omega_{\varphi_t})\,\dot{\varphi}_t\omega_{\varphi_t}^{[n]}
\end{equation*}
and the first integral would be a constant independent from $t$, and the second one would be the derivative of the $f_{s,t}^{-(n+1)}$-weighted Mabuchi energy, that is pointwise convex along $C^{1,1}$-geodesics \cite{Lahdili}. Contrary to the weighted-cscK case, in our case $f_{s,t}$ is \emph{not} a Killing potential. It turns out however that the Killing hypothesis is not necessary to establish the following analogue of \cite[Theorem $3.8$]{inoue-ent}.
\begin{theorem}\label{theo:ActionPointwiseCvxC0}
    Let $(\phi_t)_{t\geq 0}$ be a weak geodesic ray. Then the function $t\mapsto \widecheck{\actscal}_s(t)$ is pointwise convex and continuous on $[0,\infty)$.
\end{theorem}
Another important result, inspired by \cite[Lemma $3.7$]{inoue-ent} is a \emph{slope inequality} for $\widecheck{\actscal}_s(t)$.
\begin{proposition}\label{prop:A>EH}
Along a $C^{1,1}$ geodesic ray in the space of K\"ahler potentials of $(M,\omega)$, we have
\begin{equation*}
    \frac{\frac{d}{dt}_{|0^+}\widecheck{\actscal}_s(t)}{\widecheck{\V}(\omega,1+s\dot{\varphi}_0)^{\frac{n}{n+1}}}\geq \widecheck{\EH}(\omega,1+s\dot{\varphi}_0).
\end{equation*}
\end{proposition}
As an application of Theorem \ref{theo:Globalformula}, Theorem \ref{theo:ActionPointwiseCvxC0}, and Proposition \ref{prop:A>EH}, we see that the Einstein-Hilbert functional detects K-semistability of cscK manifolds.
\begin{corollary}[={\bf Corollary}~\ref{coro:Crit=>Kss}]\label{cor:Ksemistab}
    Assume that $\EH$ admits a critical point in $\holSAS(M,L)$, i.e.\ that there is a cscK metric in $c_1(L)$ (cf.\ Theorem~\ref{lem:CRITEH}). Then $(M,L)$ is K-semistable.
\end{corollary}
\begin{proof}
    Let $\eta\in\holSAS(M,L)$ be a critical point of $\EH$, and let $\omega\in 2\pi c_1(L)$ the corresponding cscK form. Let $({\tstM},\tstL)$ be a smooth test configuration for $(M,\omega)$ with $(\varphi_t)_{t\geq 0}$ the associated weak geodesic ray. By definition of $\EH_s$, we have
    \begin{equation*}
        \widecheck{\EH}_s({\tstM},\tstL)=\frac{\lim_{t\to+\infty}\frac{d}{dt}\widecheck{\actscal}_s(t)}{\lim_{t\to+\infty}{\widecheck{\V}(\omega_t,1+s\dot{\varphi}_t)^{\frac{n}{n+1}}}}=\frac{\lim_{t\to+\infty}\frac{d}{dt}\widecheck{\actscal}_s(t)}{{\widecheck{\V}(\omega,1+s\dot{\varphi}_0)^{\frac{n}{n+1}}}}
    \end{equation*}
    since $\widecheck{\V}$ is constant along geodesics, cf.\ Lemma \ref{lem:Vol-variation}. As $\widecheck{\actscal}_s(t)$ is convex in $t$ by Theorem \ref{theo:ActionPointwiseCvxC0}, using Proposition \ref{prop:A>EH} we obtain
    \begin{equation*}
        \widecheck{\EH}_s({\tstM},\tstL)\geq\frac{\frac{d}{dt}_{|0^+}\widecheck{\actscal}_s(t)}{\widecheck{\V}(\omega,1+s\dot{\varphi}_0)^{\frac{n}{n+1}}}\geq{\EH}(\omega,1+s\dot{\varphi}_0)
    \end{equation*}
    which for $s>0$ readily implies
    \begin{equation}
        \frac{\widecheck{\EH}_s({\tstM},\tstL)-{\EH}(\omega,1)}{s} \geq \frac{\widecheck{\EH}(\omega,1+s\dot{\varphi}_0)-\widecheck{\EH}(\omega,1)}{s}.
    \end{equation}
    By assumption $s=0$ is a critical point of $s\mapsto \widecheck{\EH}(\omega,1+s\dot{\varphi}_0)$, and taking the limit for $s\to 0^+$ of the previous inequality we obtain $\partial_{s=0}\widecheck{\EH}_s({\tstM},\tstL)\geq 0$. The thesis now follows from Lemma \ref{lem:derivEH=DF}.
\end{proof}
The remainder of this Section is devoted to the proofs of Theorem~\ref{theo:ActionPointwiseCvxC0} and Proposition~\ref{prop:A>EH}. The first step is to obtain an expression for this action functional along a smooth ribbon of CR-contact forms $(1+s\dot{\varphi}_t)^{-1}(\eta+d^c\varphi_t)$. As mentioned above however, here and until the end of this section we work on the K\"ahler quotient and on the product $M\times\mathbb{D}^{*}$ where $\mathbb{D}^{*}=\{0 <|\tau|\leq 1\}\subset \mathbb{C}$ is the punctured unit disc corresponding to $t \coloneqq  -\log |\tau| \geq 0$ for the variable $\tau\coloneqq e^{-t+i\theta} \in \C^*$. We will repeatedly use the correspondence between families of functions $(v_t)_{t\geq 0}$ on $M$ and $\mathbb{S}^1$-invariant functions on $M\times\mathbb{D}^*$ given by $\mathbf{v}(x,\tau)\coloneqq v_{-\log|\tau|}(x)$.
\begin{lemma}\label{lem:chS-eq}
    With the previous notation, let $\Psi(x,\tau)$ be the  $\mathbb{S}^1$-invariant function induced on $M\times\mathbb{D}^{*}$ by $\psi_t\coloneqq \log\left(\omega_{\varphi_t}^n/\omega^n\right)$, and let $f_{s,t}=1+s\dot{\varphi}_t$. The total Tanaka-Webster scalar curvature $\widecheck{\bfS}(\omega_{\varphi_t},f_{s,t})$ (see Notation~\ref{eq:notationWidecheck}) admits the following expression
    \begin{equation}\label{Eq:wbfS}
    \begin{split}
        \widecheck{\bfS}(\omega_{\varphi_t},f_{s,t})=& ns\frac{d}{dt} \left( \int_M\log\left(\frac{\omega_{\varphi_t}^n}{\omega^n}\right) \frac{\omega_{\varphi_t}^{[n]}}{f_{s,t}^{n+1}} \right)-(n+1)s\int_{M} \Psi\frac{\partial_\theta \intprod  (\omega+dd^c\Phi)^{[n+1]}}{(1+s\dot{\Phi})^{n+2}} \\
        &+ 2\int_M \Ric(\omega)\wedge \frac{\omega_{\varphi_t}^{[n-1]}}{f_{s,t}^{n}},
    \end{split}
    \end{equation}
    where $\partial_\theta$ is the generator of the standard $\mathbb{S}^1$-action on $\mathbb{D}^{*}$.
\end{lemma}
This expression for the total scalar curvature is conceptually similar to the Chen-Tian formula for the K-energy. The first integral is clearly an entropy term; the main difference with the Chen-Tian formula is that we do not give an explicit primitive for the other, ``energy'', terms.
\begin{proof}
    First, we can rewrite the total scalar curvature functional $\widecheck{\bfS}(\omega_{\varphi_t},f_{s,t})$ as
    \begin{equation}\label{Eq:wbfS:1}
        \int_M 2\Ric(\omega)\wedge \frac{\omega_{\varphi_t}^{[n-1]}}{f_{s,t}^n} + ns\int_M \left( (n+1)s |d\dot{\varphi}_t|^2\frac{\omega_{\varphi_t}^{[n]}}{f_{s,t}^{n+2}} - ( d\psi_t, d\dot{\varphi}_t)_{\varphi_t} \frac{\omega^{[n]}_{\varphi_t}}{f_{s,t}^{n+1}}\right)
    \end{equation}
    and we just need to find a primitive for the second integral. A direct computation shows
    \begin{equation}
    \begin{split}
        \frac{d}{dt}\left(\psi_t \frac{\omega_{\varphi_t}^{[n]}}{f_{s,t}^{n+1}}\right)=&\left(-\Delta_{\varphi_t}(\dot{\varphi}_t) -\psi_t \Delta_{\varphi_t}(\dot{\varphi}_t)-(n+1)s\psi_t \ddot{\varphi}_{t} f_{s,t}^{-1}\right)\frac{\omega_{\varphi_t}^{[n]}}{f_{s,t}^{n+1}},  
    \end{split}
    \end{equation}
    and integrating by parts shows that
    \begin{equation}
    \begin{split}
        \frac{d}{dt}\left(\int_M\psi_t \frac{\omega_{\varphi_t}^{[n]}}{f_{s,t}^{n+1}}\right)
        =&\int_M \left( (n+1)s |d\dot{\varphi}_t|^2\frac{\omega_{\varphi_t}^{[n]}}{f_{s,t}^{n+2}} - ( d\psi_t, d\dot{\varphi}_t)_{\varphi_t} \frac{\omega^{[n]}_{\varphi_t}}{f_{s,t}^{n+1}}\right)\\
        & - (n+1)s\int_M \psi_t \left(\ddot{\varphi}_{t} - |d\dot{\varphi}_{t}|^2_{\varphi_t}\right) \frac{\omega_{\varphi_t}^{[n]}}{f_{s,t}^{n+2}}
    \end{split}
    \end{equation}
    The thesis then follows by recalling that 
    $(\omega+dd^c\Phi)^{[n+1]}=-(\ddot{\varphi}_{u} -|d\dot{\varphi}_{u}|^2_{\varphi_u})\omega_{\varphi_u}^{[n]}\wedge du\wedge d\theta$ on $M\times \mathbb{D}^*$, so that the last integral becomes
    \begin{equation}
        (n+1)s\int_0^t\int_{M} \Psi\frac{\partial_\theta \intprod  (\omega+dd^c\Phi)^{[n+1]}}{(1+s\dot{\Phi})^{n+2}}\qedhere
    \end{equation}
\end{proof}
Lemma \ref{lem:chS-eq} gives us an explicit expression for the action functional \eqref{eq:actionfunctional_def}
\begin{equation}\label{ch-A}
\begin{split}
    \widecheck{\actscal}_s(t)\coloneqq &ns\int_M\psi_t \frac{\omega_{\varphi_t}^{[n]}}{(1+s\dot{\varphi}_t)^{n+1}}-n(n+1)s^2\int_0^t\int_{M} \Psi\frac{\partial_\theta \intprod  (\omega+dd^c\Phi)^{[n+1]}}{(1+s\dot{\Phi})^{n+2}}\\
    &+ 2\int_0^t\int_M \Ric(\omega)\wedge \frac{\omega_{\varphi_u}^{[n-1]}}{(1+s\dot{\varphi}_u)^{n}} du.
\end{split}
\end{equation}  
It is easy to check that this action functional is well-defined along a $C^{1,1}$ \emph{geodesic} ray $(\varphi_t)_{t\geq 0}$ in the space of K\"ahler potentials on $(M,\omega)$. Note first that as $\dot{\Phi}$ is uniformly bounded on $M\times\mathbb{D}^*$, we can ensure that $1+s\dot{\Phi}>0$ on $M\times\mathbb{D}^*$ for $s$ small enough. From $(\omega+dd^c\Phi)^{[n+1]}=0$, the second integral in \eqref{ch-A} vanishes and it is clear that the last Ricci integral is $C^1$. We can write the first integral in \eqref{ch-A} as
\begin{equation}\label{Ent}
\int_M\psi_t \frac{\omega_{\varphi_t}^{[n]}}{(1+s\dot{\varphi}_t)^{n+1}}=\int_M \log\left(\frac{d\lambda_{t,s}}{d\lambda_{0,s}}\right)\frac{d\lambda_{t,s}}{d\lambda_{0,s}} d\lambda_{0,s}+(n+1)\int_M \log\left(\frac{1+s\dot{\varphi}_t}{1+s\dot{\varphi}_0}\right) \frac{\omega_{\varphi_t}^{[n]}}{(1+s\dot{\varphi}_t)^{n+1}}.
\end{equation}
The first term is the entropy of the measure $ d\lambda_{t,s}\coloneqq \frac{\omega_{\varphi_t}^{[n]}}{(1+s\dot{\varphi}_t)^{n+1}}$ relative to the reference measure $d\lambda_{0,s}$. Using \cite[Proposition 2.2]{BoB}, the remaining integral is given by
\begin{equation}\label{eq:constant_log}
    \int_M \log\left(\frac{1+s\dot{\varphi}_t}{1+s\dot{\varphi}_0}\right) \frac{\omega_{\varphi_t}^{[n]}}{(1+s\dot{\varphi}_t)^{n+1}}=C(s)-\int_M \log\left(1+s\dot{\varphi}_0\right) \frac{\omega_{\varphi_t}^{[n]}}{(1+s\dot{\varphi}_t)^{n+1}}
\end{equation}
where $C(s)$ is a constant depending only on $s$ and the remaining part is clearly well-defined. The lower semi-continuity of the entropy with respect to the weak topology on the space of measures (see \cite{BB}) yields the following
\begin{corollary}\label{cor:A-def}
Along a $C^{1,1}$ geodesic ray in the space of K\"ahler metrics on $(M,\omega)$, the function $\widecheck{\actscal}_s(t)$ is well defined and lower semi-continuous on $[0,\infty)$.
\end{corollary}

We now turn to the proof of Theorem \ref{theo:ActionPointwiseCvxC0}. We follow the proof of the convexity of the K-energy along weak geodesics of Berman and Berndtsson \cite{BB}, and its adaptation to Perelman's entropy by Inoue \cite{inoue-ent}. As our setting has many commonalities with \cite{BB} and \cite{inoue-ent}, we will only highlight the main points of the proof. We start by introducing a modified version of $\widecheck{\actscal}_s(\tau)\coloneqq \widecheck{\actscal}_s(e^{-t+i\theta})$, seen as a function on $\mathbb{D}^*$ and we compute its derivative in the sense of currents.

Let $\Theta_\tau$ be an $\mathbb{S}^1$-invariant family of volume forms on $M$ and denote by $\Psi=(\psi_\tau)_{\tau\in\mathbb{D}^*}$ the corresponding metric on the relative canonical bundle $K_{M\times\mathbb{D}^*/\mathbb{D}^*}$. Along a weak geodesic ray $(\varphi_t)_{t\geq0}$ on $(M,\omega)$ we define a modified action $\widecheck{\actscal}^\Psi(s,\tau)$ on $\mathbb{D^*}$ as (recall $t=-\log|\tau|$)
\begin{equation}\label{eq:A_Psi-general}
    \begin{split}
        \widecheck{\actscal}^\Psi(s,\tau)=&ns\int_M\log\left(\frac{e^{\psi_\tau}}{\omega^{[n]}} \right)\frac{\omega_{\varphi_t}^{[n]}}{(1+s\dot{\varphi}_t)^{n+1}} + 2\int_0^{t}\int_M \Ric(\omega)\wedge \frac{\omega_{\varphi_u}^{[n-1]}}{(1+s\dot{\varphi}_u)^{n}} du\\
        &-n(n+1)s^2\int_0^{t}\int_{M} \log\left(\frac{e^{\psi_u}}{\omega^{[n]}} \right)\frac{\partial_\theta \intprod  (\omega+dd^c\Phi)^{[n+1]}}{(1+s\dot{\Phi})^{n+2}}.
    \end{split}
\end{equation}
Note that $\widecheck{\actscal}^\Psi(s,\tau)$ coincides with $\widecheck{\actscal}(s,\tau)$ defined in \eqref{ch-A} if $\Psi$ is induced by $\omega_{\varphi_t}^n$. It will be convenient, however, to first compute the second variation of $\widecheck{\actscal}^\Psi$ assuming that $\Psi$ and $\Phi$ are \emph{smooth}. We will then exploit the approximation argument of \cite[Theorem $3.3$]{BB} to deduce the desired properties of $\widecheck{\actscal}$.
\begin{lemma}\label{lem:ddc_A^Psi-smooth}
    Let $(\varphi_t)_{t\geq 0}$ be a smooth ray of K\"ahler potentials on $(M,\omega)$ with $\Phi$ the associated $\mathbb{S}^1$-invariant function on $M\times\mathbb{D}^*$ and let $(\Theta_\tau)_{\tau\in \mathbb{D}^*}$ be a family of $\mathbb{S}^1$-invariant volume forms on $M$ such that the metric $\Psi\coloneqq \log(\Theta_\tau)$ on $K_{M\times\mathbb{D}^*/\mathbb{D}^*}$ is smooth. The following identity holds in the weak sense of currents
    \begin{equation}\label{eq:ddc-APsi0}
        \frac{1}{ns}dd^c\widecheck{\actscal}^\Psi_s=\int_{M} \frac{dd^c\Psi\wedge(\omega+dd^c\Phi)^{[n]}}{(1+s\dot{\Phi})^{n+1}} + (n+1)s (d^c\Psi)(\partial_\theta)\frac{(\omega+dd^c\Phi)^{[n+1]}}{(1+s\dot{\Phi})^{n+2}}.
    \end{equation}
    In this expression, $\proj_1:M\times\mathbb{D}^*\to M$ is the projection on the first factor and $\int_M$ stands for the integration along the fibres of $M\times\mathbb{D}^*\to \mathbb{D}^*$.
\end{lemma}
The proof follows the lines of \cite[proof of Proposition $3.10$]{inoue-ent}.
\begin{proof}
     We compute the differential of $\widecheck{\actscal}^\Psi_s(\tau)$ by integrating against test functions $\chi(\tau)$ on $\mathbb{D}^*$. As all components of $\widecheck{\actscal}^\Psi_s(\tau)$ are $\mathbb{S}^1$-invariant,
    \begin{equation}\label{eq:ddc-APsi1}
        \begin{split}
            \int_{\mathbb{D}^*} \widecheck{\actscal}^\Psi_s(\tau) dd^c \chi= &
            ns\int_{M\times\mathbb{D}^*} d^c\chi \wedge d\left[ \log\left(\frac{e^{\Psi}}{\omega^{[n]}} \right) \frac{(\omega+dd^c\Phi)^{[n]}}{(1+s\dot{\Phi})^{n+1}} \right] \\
            &+n(n+1)s^2\int_{\mathbb{D}^*} d^c\chi \wedge \left[\int_{M} \log\left(\frac{e^{\Psi}}{\omega^{[n]}} \right) \frac{\partial_\theta \intprod  (\omega+dd^c\Phi)^{[n+1]}}{(1+s\dot{\Phi})^{n+2}} \right]\\
            & + 2 \int_{\mathbb{D}^*} \chi \left[ \partial_t \int_M \Ric(\omega)\wedge \frac{\omega_{\varphi_t}^{[n-1]}}{(1+s\dot{\varphi}_t)^{n}} \right] dt\wedge d\vartheta.
        \end{split}
    \end{equation}
    To simplify the three integrals, recall that with respect to the coordinate $\tau=e^{-t + i \vartheta}$
    \begin{equation}
        \omega+dd^c\Phi=\omega_{\varphi_t}-d^c\dot{\varphi}_t\wedge dt-d\dot{\varphi}_t\wedge d\theta-\ddot{\varphi}_t dt\wedge d\theta
    \end{equation}
    hence $(\omega+dd^c\Phi)^{[n+1]}=-\left(\ddot{\varphi}_t-|d\dot{\varphi}_t|^2_{\varphi_t}\right)\omega_{\varphi_t}^{[n]}\wedge dt \wedge d\vartheta$ and
    \begin{equation}\label{eq:dtheta_contraction}
        \partial_\theta\intprod ({\proj_1}^*\omega+dd^c\Phi)^{[n+1]}= d\dot{\Phi}\wedge ({\proj_1}^*\omega+dd^c\Phi)^{[n]}.
    \end{equation}
    Moreover, for any $(1,1)$-form $\alpha$ on $M$
    \begin{equation}
        \alpha \wedge (\omega+dd^c\Phi)^{[n]} =\left[ ( \alpha, d\dot{\varphi}_t\wedge d^c\dot{\varphi}_t )_{\varphi_t}\omega_{\varphi_t}^{[n]} + (\ddot{\varphi}_t-|d\dot{\varphi}_t|_{\varphi_t}^2) \alpha \wedge  \omega_{\varphi_t}^{[n-1]}\right]\wedge dt\wedge d\theta.    
    \end{equation}
    This last equation for the choice $\alpha=\Ric(\omega)$ shows that the last integrand in \eqref{eq:ddc-APsi1} is
    \begin{equation}\label{eq:ddc_Ric}
        \partial_t \int_M \Ric(\omega)\wedge \frac{\omega_{\varphi_t}^{[n-1]}}{(1+s\dot{\varphi}_t)^{n}} \, dt\wedge d\theta = n s \int_{M} \Ric(\omega)\wedge\frac{(\omega+dd^c\Phi)^{[n]}}{(1+s\dot{\Phi})^{n+1}}.
    \end{equation}
    For the first integral of \eqref{eq:ddc-APsi1} we find, using \eqref{eq:dtheta_contraction}:
    \begin{equation}\label{eq:log_volform_d}
    \begin{split}
        d\left[ \log\left(\frac{e^{\Psi}}{\omega^{[n]}} \right) \frac{(\omega+dd^c\Phi)^{[n]}}{(1+s\dot{\Phi})^{n+1}} \right] = & d\log\left(\frac{e^{\Psi}}{\omega^{[n]}} \right) \wedge \frac{(\omega+dd^c\Phi)^{[n]}}{(1+s\dot{\Phi})^{n+1}} \\
        & - (n+1)s \log\left(\frac{e^{\Psi}}{\omega^{[n]}} \right) \frac{\partial_\theta \intprod  (\omega+dd^c\Phi)^{[n+1]}}{(1+s\dot{\Phi})^{n+2}}
    \end{split}
    \end{equation}
    and the second term will cancel out the second integral in \eqref{eq:ddc-APsi1}. For the remaining piece of this first integral, using the relation $dd^c\log\left(e^{\Psi}/\omega^{[n]} \right)=dd^c\Psi-2\Ric(\omega)$ we obtain
    \begin{equation}\label{eq:log_volform_ddc}
    \begin{split}
        &\int_{M\times\mathbb{D}^*} d^c\chi \wedge d\log\left(\frac{e^{\Psi}}{\omega^{[n]}} \right) \wedge \frac{(\omega+dd^c\Phi)^{[n]}}{(1+s\dot{\Phi})^{n+1}}=\\
         =& \int_{M\times\mathbb{D}^*} \chi \left( dd^c\Psi-2\Ric(\omega) - (n+1)s\,  d\log\left(\frac{e^{\Psi}}{\omega^{[n]}} \right) \frac{d^c\dot{\Phi}}{1+s\dot{\Phi}} \right) \wedge \frac{(\omega+dd^c\Phi)^{[n]}}{(1+s\dot{\Phi})^{n+1}} .
    \end{split}
    \end{equation}
    Finally, a small modification of \eqref{eq:dtheta_contraction} shows
    \begin{equation}\label{eq:ddc_exchange}
        d\log\left(\frac{e^{\Psi}}{\omega^{[n]}} \right) \wedge d^c\dot{\Phi}\wedge(\omega+dd^c\Phi)^{[n]} = - d^c\Psi(\partial_\theta)\,(\omega+dd^c\Phi)^{[n+1]}.
    \end{equation}
    Substituting \eqref{eq:ddc_Ric}, \eqref{eq:log_volform_d}, \eqref{eq:log_volform_ddc}, \eqref{eq:ddc_exchange}, back into \eqref{eq:ddc-APsi1} gives the thesis.
\end{proof}

We are now in a position to show the continuity and the weak subharmonicity of the function $\widecheck{\actscal}_s$, proving Theorem \ref{theo:ActionPointwiseCvxC0}. We proceed along the lines of \cite[proof of Theorem $3.3$]{BB}, see also \cite{inoue-ent}.
\begin{proof}[Proof of Theorem \ref{theo:ActionPointwiseCvxC0}]
    We use the same notation as in the proof and statement of Lemma \ref{lem:ddc_A^Psi-smooth}.
    We start by claiming that if $\Psi$ is only Lipschitz and $\Phi$ is a $\mathcal{C}^{1,1}$-geodesic on $M\times\mathbb{D}^*$ such that 
    \begin{equation}\label{positivity}
        dd^c\Psi\wedge (\omega+dd^c\Phi)^{[n]} \geq 0
    \end{equation}
   as a current, then the following equality holds in the weak sense of currents
    \begin{equation}\label{eq:ddc-APsi-weak}
        dd^c\widecheck{\actscal}^\Psi_s = ns\int_{M} \frac{dd^c\Psi\wedge(\omega+dd^c\Phi)^{[n]}}{(1+s\dot{\Phi})^{n+1}},
    \end{equation}
    where $(1+s\dot{\Phi})^{-(n+1)}\,dd^c\Psi\wedge(\omega+dd^c\Phi)^{[n]}$ is a well defined Radon measure as $dd^c\Psi\wedge(\omega+dd^c\Phi)^{[n]}$ is a current of order $0$ by the positivity assumption \eqref{positivity}. In the case when both $\Psi$ and $\Phi$ are smooth, \eqref{eq:ddc-APsi-weak} is just the statement of Lemma \ref{lem:ddc_A^Psi-smooth}. The proof of \eqref{eq:ddc-APsi-weak} in this more general setting follows the same arguments as \cite[Theorem 3.9]{inoue-ent}, that only make use of the Lipschitz continuity of $\Psi,\dot{\Phi}$ and are independent from the weight $e^{-\dot{\Phi}}$, which we can replace by $(1+s\dot{\Phi})^{-(n+1)}$. We only recall the main points.
    
    Assume first that $\Psi$ is smooth, and $\Phi$ is $\mathcal{C}^{1,1}$. Then \eqref{eq:ddc-APsi-weak} follows from the fact that $\Phi$ admits an approximation by smooth $\varepsilon$-geodesic rays. More precisely, for $0<\varepsilon \ll 1$ let $\Phi^\varepsilon$ be the $\mathbb{S}^1$-invariant solution of the elliptic problem
    \begin{equation}\label{eq:epsilon_geod}
        (\omega+dd^c\Phi^\varepsilon)^{[n+1]}=\varepsilon\left(\omega+\frac{id\tau\wedge d\bar{\tau}}{2|\tau|^2}\right)^{[n+1]},\quad \Phi^\varepsilon(1,\cdot)=\varphi_0.
    \end{equation}
    Then each $\Phi^\varepsilon$ is smooth, and the family $(\Phi^\varepsilon)_{\varepsilon>0}$ is decreasing in $\epsilon$ and converges uniformly in $C^{1,1}(M\times\mathbb{D}^*)$ to $\Phi$ for $\varepsilon \to 0$, see \cite{CTW1}. The convergence $\Phi^\varepsilon\to\Phi$ then readily gives \eqref{eq:ddc-APsi-weak}, using Lemma \ref{lem:ddc_A^Psi-smooth} for the family of K\"ahler potentials $\varphi^\varepsilon_t$ corresponding to $\Phi^\varepsilon(\cdot,\tau)$. The case when $\Psi$ is only assumed to be Lipschitz now follows by taking a $\mathcal{C}^1$-approximation $(\Psi_j)$ of $\Psi$ such that each $\Psi_j$ is smooth, for which we have convergence of $\widecheck{\actscal}^{\Psi_j}_s$ and $dd^c\widecheck{\actscal}^{\Psi_j}_s$ to $\widecheck{\actscal}^{\Psi}_s$ and $dd^c\widecheck{\actscal}^{\Psi}_s$, respectively. 
    
    The metric $\Psi$ defined on $K_{M\times\mathbb{D}^*/\mathbb{D}^*}$ by $\omega_{\varphi_t}^n$ is however not even locally bounded; the problem can be circumvented by considering instead the families of approximating metrics $(\Psi_A)_{A>0}$, $(\Psi_{k,A})_{k,A>0}$ constructed in \cite[proof of Theorem $3.3$]{BB}, for which \eqref{eq:ddc-APsi-weak} gives $dd^{c}\widecheck{\actscal}^{\Psi_A}_s\geq 0$. 
    The dominated convergence theorem gives the pointwise limit $\widecheck{\actscal}^{\Psi_{A}}_s\to \widecheck{\actscal}^{\Psi}_s=\widecheck{\actscal}_s$ as $A\to\infty$. It follows that $dd^{c}\widecheck{\actscal}_s\geq 0$, as a limit of positive currents $dd^{c}\widecheck{\actscal}^{\Psi_A}_s$, showing that $\tau\mapsto\widecheck{\actscal}(s,\tau)$ is a weakly subharmonic function. We refer to \cite{BB} for the details.

    The continuity of $\widecheck{\actscal}_s$ also follows from the existence of the regularizing sequence $\Psi_{A,k}$  (see \cite[Theorem 3.4]{BB}). Indeed, by continuity of $\Psi_{A,k}$, it is clear that $\widecheck{\actscal}^{\Psi_{A,k}}_s$ is also continuous on $\mathbb{D}^*$ and weakly subharmonic, hence it is pointwise convex as a function of $t=-\log|\tau|$. Since the sequence $\widecheck{\actscal}^{\Psi_{A,k}}_s$ converges pointwise to $\widecheck{\actscal}_s$ as $k,A\to\infty$ then $t\mapsto\widecheck{\actscal}_s(t)$ is also a convex function. As a convex function on $[0,\infty)$, $\widecheck{\actscal}_s$ is continuous on $(0,+\infty)$ and upper-semi continuous on $[0,+\infty)$. From Corollary \ref{cor:A-def}, $\widecheck{\actscal}_s$ is in fact continuous.
\end{proof}

We now turn to the proof of the slope inequality, Proposition \ref{prop:A>EH}. Combining \eqref{ch-A},  \eqref{Ent} \eqref{eq:constant_log} we can write
\begin{equation}
    \begin{split}\label{ch-A-Ent}
    \widecheck{\actscal}_s(t) = & ns\widecheck{\V}(\omega,1+s\dot{\varphi}_0)\int_M \log\left(\frac{d\tilde{\lambda}_{t,s}}{d\tilde{\lambda}_{0,s}}\right)\frac{d\tilde{\lambda}_{t,s}}{d\tilde{\lambda}_{0,s}} d\tilde{\lambda}_{0,s}
    + 2 \int_0^t\int_M \Ric(\omega)\wedge \frac{\omega_{\varphi_u}^{[n-1]}}{(1+s\dot{\varphi}_u)^{n}} du\\
    & -n(n+1)s \int_M \log\left(1+s\dot{\varphi}_0\right) \frac{\omega_{\varphi_t}^{[n]}}{(1+s\dot{\varphi}_t)^{n+1}}+C(s),
  \end{split}
\end{equation}
where $d\tilde{\lambda}_{t,s}\coloneqq d\lambda_{t,s}(\int_M d\lambda_{t,s})^{-1}$ is the probability measure induced by the volume form $d\lambda_{t,s}=\frac{\omega_{\varphi_t}^{[n]}}{(1+s\dot{\varphi}_t)^{n+1}}$ and $C(s)$ is a constant depending only on $s$. Note that $\int_M d\lambda_{t,s}$ is independent of $t$, as $\varphi_t$ is a geodesic (cf.\ Lemma \ref{lem:Vol-variation}), so that $\int_M d\lambda_{t,s}=\widecheck{\V}(\omega,1+s\dot{\varphi}_0)$.

Similarly to \cite[Lemma 3.7]{inoue-ent}, Proposition \ref{prop:A>EH} now follows from the positivity of the entropy on the space of probability measures.
\begin{proof}[Proof of Proposition~\ref{prop:A>EH}]
    As the entropy is positive on the space of probability measures, we have
    \begin{equation}
        \frac{d}{dt}_{|0^+}\int_M \log\left(\frac{d\tilde{\lambda}_{t,s}}{d\tilde{\lambda}_{0,s}}\right)\frac{d\tilde{\lambda}_{t,s}}{d\tilde{\lambda}_{0,s}} d\tilde{\lambda}_{0,s}=\underset{t\to0^+}{\lim}\frac{1}{t}\int_M \log\left(\frac{d\tilde{\lambda}_{t,s}}{d\tilde{\lambda}_{0,s}}\right)\frac{d\tilde{\lambda}_{t,s}}{d\tilde{\lambda}_{0,s}} d\tilde{\lambda}_{0,s}\geq 0.
    \end{equation}
    We should compute the derivative of the other two terms in \eqref{ch-A-Ent} along a geodesic ray. This is straightforward for the term involving $\Ric(\omega)$; for the other one, we consider a smooth $\epsilon$-geodesic ray $\varphi^{\epsilon}$, and integrating by parts one obtains
    \begin{equation*}
    \begin{split}
        \frac{d}{dt}\int_M & \log\left(1+s\dot{\varphi}^{\epsilon}_0\right) \frac{\omega_{\varphi^{\epsilon}_t}^{[n]}}{(1+s\dot{\varphi}^{\epsilon}_t)^{n+1}}=\\
        = & -(n+1)s\epsilon\int_M \log\left(1+s\dot{\varphi}^{\epsilon}_0\right)  \frac{\omega_{\varphi^{\epsilon}_t}^{[n]}}{(1+s\dot{\varphi}^{\epsilon}_t)^{n+2}}
        -s \int_M \frac{d\dot{\varphi}^{\epsilon}_0}{(1+s\dot{\varphi}^{\epsilon}_0)} \wedge \frac{d^c\dot{\varphi}^{\epsilon}_t \wedge \omega_{\varphi^{\epsilon}_t}^{[n-1]}}{(1+s\dot{\varphi}^{\epsilon}_t)^{n+1}}.
    \end{split}
    \end{equation*}
    Since $\Phi^{\epsilon}\to \Phi$ uniformly in $C^{1,1}$, it follows that
    \begin{equation}\label{eq:der_limit}
    \begin{split}
        & \frac{d}{dt}\int_M \log\left(1+s\dot{\varphi}^{\epsilon}_0\right) \frac{\omega_{\varphi^{\epsilon}_t}^{[n]}}{(1+s\dot{\varphi}^{\epsilon}_t)^{n+1}} \rightarrow -s\int_M d\dot{\varphi}_0\wedge d^c\dot{\varphi}_t \frac{\omega_{\varphi_t}^{[n-1]}}{(1+s\dot{\varphi}_0)(1+s\dot{\varphi}_t)^{n+1}}
    \end{split}
    \end{equation}
    uniformly for $\varepsilon\to 0$ on each segment contained in $[0,+\infty)$ (the range of $t$). It follows that the function $\int_M \log\left(1+s\dot{\varphi}_0\right) \frac{\omega_{\varphi_t}^{[n]}}{(1+s\dot{\varphi}_t)^{n+1}}$ is $C^1$ in $t$, with derivative given by the right hand side of \eqref{eq:der_limit}.
    It follows that
    \begin{equation}
    \begin{split}
        \frac{d}{dt}_{|0^+}\widecheck{\actscal}_s(t)\geq& 2\int_M \Ric(\omega)\wedge \frac{\omega^{[n-1]}}{(1+s\dot{\varphi}_0)^{n}}+s^2n(n+1)\int_M\lvert d\dot{\varphi}_0\rvert_{\varphi_0}^2\frac{\omega^{[n]}}{(1+s\dot{\varphi}_0)^{n+2}}\\
        =&\int_M\left((1+s\dot{\varphi}_0)^2\scal_\omega+n(n+1)|d(1+s\dot{\varphi}_0)|_\omega^2\right)\frac{\omega^{[n]}}{(1+s\dot{\varphi}_0)^{n+2}}\\
        =&\widecheck{\EH}(\omega,1+s\dot{\varphi}_0)\widecheck{\V}(\omega,1+s\dot{\varphi}_0)^{\frac{n}{n+1}}.
    \end{split}
    \end{equation}
\end{proof}

\subsection{The Einstein-Hilbert functional near the central fibre of a test configuration}\label{ss:globalformulaEH}

We will now prove Theorem \ref{theo:Globalformula}, by computing the value of $\bfS$ and $\bfV$ near the central fibre of a test configuration. It will however be convenient to slightly change the setting: we will compute the limit of these functionals along the compatible geodesic ray of a smooth dominant test configuration (see Section \ref{sec:test_conf}). It will be convenient then to use the action functional \eqref{eq:actionfunctional_def}. We first consider the volume functional, for which the argument is simpler.

\subsubsection{The volume functional on the central fibre of a test configuration}
For a smooth dominant test configuration $({\tstM},\tstL)$ with $\Phi=(\varphi_t)_{t\geq 0}$ the $C^{1,1}$-compatible geodesic ray on $M\times\mathbb{D}^*$, let $\Gamma=\Pi^*\Phi+\gamma_D$ be the solution of the boundary value problem \eqref{BVP:G}. As $\Gamma$ is $\Omega$-PSH on $\tstM_{\mathbb{D}}$, it extends to a $\mathbb{S}^1$-invariant $\Omega$-psh function on $\tstM$, see \cite[Proposition 8.8]{GZ_Book}, that we denote by the same symbol. In particular, $\Pi^*\Phi$ extends to a $\mathbb{S}^1$-invariant function on $\tstM\setminus M_0$. We consider the anti-derivative of the volume $\widecheck{\actvol}_s(\tau)$ as an $\mathbb{S}^1$-invariant function on $\mathbb{C}^*\subset\mathbb{P}^1$ defined by
\begin{equation}
    \widecheck{\actvol}_s(\tau)=\int_0^t\int_M\left(\frac{(\omega+dd^c\Phi)^{[n]}}{(1+s\dot{\Phi})^{n+1}}\right)du.
\end{equation}
In the following Lemma, we compute $dd^c\widecheck{\actvol}_s$ as a current on $\mathbb{P}^1$.
\begin{lemma}\label{lem:ddcV_TC}
    The following equality holds in the weak sense of currents
    \begin{equation}\label{eq:ddcV_TC}
        \frac{1}{(n+1)s} dd^c\widecheck{\actvol}_s=-\TCmap_\star\left( \frac{(\Omega+dd^c\Gamma)^{[n+1]}}{\left(1-s\mu_{\Gamma}\right)^{n+2}}\right)
    \end{equation}
    where $\Gamma$ is the solution of the boundary value problem \eqref{BVP:G} extended to the total space ${\tstM}$ and $\mu_\Gamma\coloneqq \mu+(d^c\Gamma)(\zeta)$.
\end{lemma}
\begin{proof}
    We start by showing the equation \eqref{eq:ddcV_TC} for $(\varphi_t)_{t\geq 0}$ smooth such that $\Phi$ is ${\proj_1}^*\omega$-psh.
    
    By \eqref{eq:smooth_variations}, for any smooth subgeodesic ray $\Phi$ we have
    \begin{equation}
        \frac{1}{(n+1)s}dd^c\widecheck{\actvol}_s=-\proj_{2,\star}\left( \frac{({\proj_1}^*\omega+dd^c\Phi)^{[n+1]}}{(1+s\dot{\Phi})^{n+2}}\right).
    \end{equation}
    where ${\proj_1}:M\times\mathbb{C}^*\to M$ is the projection on the first factor and $\proj_{2,\star}$ stands for the integration along the fibres of $\proj_2: M\times\mathbb{C}^*\to \mathbb{P}^1$. We can pull this equation back to $\tstM\setminus M_0$. As $\Pi^*(\proj_1\omega+dd^c\Phi)=\Omega+dd^c\Gamma$ away from $M_0$, the following equality holds on ${\tstM}\setminus M_0$ in the weak sense of currents
    \begin{equation}
        \begin{split}
        \frac{1}{(n+1)s} dd^c\widecheck{\actvol}_s=-\TCmap_\star\left( \frac{(\Omega+dd^c\Gamma)^{[n+1]}}{\left(1+s\,\Pi^*\dot{\Phi}\right)^{n+2}}\right).
    \end{split}
    \end{equation}
    From \eqref{eq:dtheta_contraction}, we see that $-\Pi^*\dot{\Phi}$ is a Hamiltonian for the $\mathbb{S}^1$-action on $\tstM\setminus M_0$ with respect to $\Omega+dd^c\Gamma$, and this gives the thesis for a smooth subgeodesic ray. The general case follows through pointwise approximation by smooth subgeodesics.
\end{proof}
Through a simple integration by parts, Lemma \ref{lem:ddcV_TC} allows us to compute the asymptotic slope $\widecheck{\actvol}_s$, and so the limit of the volume functional.
\begin{lemma}\label{lem:AsymSlopeVol}
    For any smooth dominant test configuration $({\tstM},\tstL)$, with $(\varphi_t)_{t\geq 0}$ the $C^{1,1}$-compatible geodesic ray, the limit of the volume functional $\bfV_s(\tau)=\bfV(f_{s,t}^{-1}\eta_t)$ on the central fibre exists and is independent on the choice of $\tilde{\eta}\in\holSAS(\tstM,\tstL)$. Moreover,
    \begin{equation}\label{AsymSlopeVol}
        2\pi\underset{\tau\to 0}{\lim}\widecheck{\V}_s(\tau)=\bfV_s(\tstM,\tstL).
    \end{equation}
\end{lemma}
\begin{proof}
    From Lemma \ref{lem:ddcV_TC}, we find
    \begin{equation}
    \begin{gathered}
        \int_{\tstM}\frac{(\Omega+dd^c\Gamma)^{[n+1]}}{\left(1-s\mu_{\Gamma}\right)^{n+2}} =  
        \lim_{\tau\to 0}\int_{\tstM\setminus \TCmap^{-1}(\mathbb{D}_{|\tau|})}\frac{(\Omega+dd^c\Gamma)^{[n+1]}}{\left(1-s\mu_{\Gamma}\right)^{n+2}} = \frac{-1}{(n+1)s}\lim_{\tau\to 0}\int_{\mathbb{C}\setminus \mathbb{D}_{|\tau|}} dd^c\widecheck{\actvol}_s =\\
        =\frac{-2\pi}{(n+1)s}\lim_{\tau\to 0}\int^{-\log|\tau|}_{-\infty} \frac{d^2 \widecheck{\actvol}_s}{dt^2} dt = \frac{-2\pi}{(n+1)s}\left(\lim_{t\to +\infty}\frac{d}{dt}\widecheck{\actvol}_s(t) - \lim_{t\to -\infty}\frac{d}{dt}\widecheck{\actvol}_s(t)\right).
    \end{gathered}
    \end{equation}
    Since the total volume function $\widecheck{\V}_s$ is well defined along the weak geodesic ray $(\varphi_t)_{t\geq0}$ induced from the test configuration and $\frac{d}{dt}\widecheck{\actvol}_s(t)=\widecheck{\V}_s(t)$, we obtain
    \begin{equation}\label{eq:intV=slop}
    \begin{split}
        \frac{2\pi}{(n+1)s}\left(\underset{t\to +\infty}{\lim}\widecheck{\V}_s(t) - \underset{t\to -\infty}{\lim}\widecheck{\V}_s(t)\right)=-\int_{{\tstM}}\frac{(\Omega+dd^c\Gamma)^{[n+1]}}{\left(1-s\mu_{\Gamma}\right)^{n+2}}.
    \end{split}
    \end{equation}
    For a K\"ahler metric $\Omega'\in[\Omega]$ we let $\mu'$ be a Hamiltonian function for $\uzeta$ with respect to $\Omega'$. By the discussion in \cite[\S $1$]{FutakiMabuchi_extremal}, see also \cite[Lemma 2]{Lahdili}, the quantity $\int\left(1-s\mu'\right)^{-n-2}\Omega'{}^{[n+1]}$ is independent of the choice of the K\"ahler metric $\Omega'\in[\Omega]$. In particular, \eqref{eq:intV=slop} can be rewritten as
    \begin{equation}
        \frac{2\pi}{(n+1)s}\left(\underset{t\to +\infty}{\lim}\widecheck{\V}_s(t) - \underset{t\to -\infty}{\lim}\widecheck{\V}_s(t)\right)=-\int_{{\tstM}}\frac{\Omega^{[n+1]}}{\left(1-s\mu\right)^{n+2}}
    \end{equation}
    Since the fiber $M_{\infty}$ is fixed by the $\mathbb{S}_{\uzeta}^1$-action on ${\tstM}$ we have $\mu_{|M_{\infty}}=\mu_{\max}$ and
    \begin{equation}
        \underset{\tau\to\infty}{\lim}\widecheck{\V}_s(\tau)= \int_{M_\infty}\frac{(\Omega+dd^c\Gamma)^{[n]}}{(1-s\mu_{\max})^{n+1}}=\frac{1}{(1-s\mu_{\max})^{n+1}}\int_{M_\infty}\Omega^{[n]}.
    \end{equation}
    Combining this with \eqref{eq:intV=slop}.
\end{proof}

\subsubsection{The slope of the action functional along a test configuration} Since the $\mathbb{S}^1$-invariant metric $\log \omega_{\varphi_\tau}^n$, seen as a metric on $K_{\tstM/\mathbb{P}^1}$ by \eqref{dominate}, may blow up near the central fibre, we need to employ the twisted action functional $\widecheck{\actscal}_s^\Psi$ introduced in \eqref{eq:A_Psi-general}. Here, we choose $e^\Psi$ to be the $\mathbb{S}^1$-invariant smooth Hermitian metric on $K_{\tstM/\mathbb{P}^1}=K_{\tstM}-\TCmap^*K_{\mathbb{P}^1}$ induced from $\Omega$ by $\Omega^{[n+1]}=e^{\Psi+\TCmap^*\log\omega_{\rm FS}}$, whose curvature $2$-form satisfies
\begin{equation}\label{eq:choosing_Psi}
    -\frac{1}{2}dd^c\Psi=\Ric(\Omega)-\TCmap^*\omega_{\rm FS}.
\end{equation}
We will then show in Lemma \ref{lem:equalslopes} that $\widecheck{\actscal}_s^\Psi$ and $\widecheck{\actscal}_s$ have the same asymptotic slope.

We will need an immediate extension of Lemma \ref{lem:ddc_A^Psi-smooth} to a non-trivial test configuration, which can be obtained as in Lemma \ref{lem:ddcV_TC}.

\begin{lemma}\label{lem:ddcA_TC}
    Let $({\tstM},\tstL)$ be a smooth dominant test configuration with $(\varphi_t)_{t\geq 0}$ the $C^{1,1}$-compatible geodesic ray. The following equality holds in the weak sense of currents
    \begin{equation}\label{eq:ddcA_TC}
        \frac{1}{ns} dd^c\widecheck{\actscal}^\Psi_s=\TCmap_\star\left( \frac{dd^c\Psi\wedge(\Omega+dd^c\Gamma)^{[n]}}{\left(1-s\mu_{\Gamma}\right)^{n+1}}+ (n+1)s(d^c\Psi)(\uzeta) \frac{(\Omega+dd^c\Gamma)^{[n+1]}}{\left(1-s\mu_{\Gamma}\right)^{n+2}}\right)
    \end{equation}
    where $\Gamma$ is the solution of the boundary value problem \eqref{BVP:G} extended to the total space ${\tstM}$ and $\mu_\Gamma\coloneqq \mu+(d^c\Gamma)(\uzeta)$.
\end{lemma}

We can mimic the proof of Lemma \ref{lem:AsymSlopeVol} to compute the asymptotic slope of $\widecheck{\actscal}^\Psi(s,\cdot)$.

\begin{lemma}\label{lem:AsymCurvPSI}
    With the previous notation, the asymptotic slope of $\widecheck{\actscal}^\Psi(s,\cdot)$ near the central fibre of a test configuration satisfies
    \begin{equation}
        2\pi\lim_{t\to +\infty}\frac{d}{dt}\widecheck{\actscal}^\Psi_s=\bfS_s(\tstM,\tstL)
    \end{equation}
\end{lemma}
\begin{proof}
    We follow the proof of Lemma \ref{lem:AsymSlopeVol}. Lemma \ref{lem:ddcA_TC} gives
    \begin{equation}\label{eq:intM=slop}
    \begin{split}
        \int_{{\tstM}}\frac{dd^c\Psi\wedge(\Omega+dd^c\Gamma)^{[n]}}{\left(1-s\mu_\Gamma\right)^{n+1}} + (n+1)s \int_{{\tstM}} (d^c\Psi)(\uzeta) \frac{(\Omega+dd^c\Gamma)^{[n+1]}}{\left(1-s\mu_\Gamma\right)^{n+2}}=&\\
        =\frac{1}{ns}\lim_{\tau \to 0}\int_{\mathbb{C}\setminus \mathbb{D}_{|\tau|}} dd^c\widecheck{\actscal}^\Psi_s=\frac{2\pi}{ns}\left(\lim_{t\to +\infty}\frac{d}{dt}\widecheck{\actscal}^\Psi_s - \lim_{t\to -\infty}\frac{d}{dt}\widecheck{\actscal}^\Psi_s\right)&.
    \end{split}
    \end{equation}
    For a K\"ahler metric $\Omega'\in[\Omega]$, let $\mu'$ be a Hamiltonian function for $\uzeta$ with respect to $\Omega'$. Using \cite[Lemma 2]{Lahdili}\footnote{We take $\theta=dd^c\psi$ and ${\rm v}(\mu')=\left(1-s\mu'\right)^{-(n+1)}$ in the second integral $B_{\rm v}^{\theta}$ of \cite[Lemma 2]{Lahdili}}, the following integral 
    \begin{equation}\label{intM-const}
        \int_{{\tstM}} \left( (dd^c\Psi)\wedge\frac{(\Omega')^{[n]}}{\left(1-s\mu'\right)^{n+1}}+(n+1)s(d^c\Psi)(\uzeta) \frac{(\Omega')^{[n+1]}}{\left(1-s\mu'\right)^{n+2}} \right).
    \end{equation}
    is independent of the choice of the K\"ahler metric $\Omega'\in[\Omega]$.
    As in the proof of Lemma \ref{lem:AsymSlopeVol}, we can then rewrite \eqref{eq:intM=slop} as
    \begin{equation}\label{eq:slopes_accessory}
    \begin{split}
        \lim_{t\to +\infty}\frac{d}{dt}\widecheck{\actscal}^\Psi_s - \lim_{t\to -\infty}\frac{d}{dt}\widecheck{\actscal}^\Psi_s=\frac{ns}{2\pi}\int_{{\tstM}} \left(dd^c\Psi\wedge\frac{\Omega^{[n]}}{\left(1-s\mu\right)^{n+1}} + (n+1)s (d^c\Psi)(\uzeta) \frac{\Omega^{[n+1]}}{\left(1-s\mu\right)^{n+2}}\right).
    \end{split}
    \end{equation}
    Contracting both sides of \eqref{eq:choosing_Psi} by $\uzeta$ and using the equivariance of \eqref{dominate}, we see that there is a constant $C$ such that
    \begin{equation}
        -(d^c\Psi)(\uzeta)=\Delta_{\Omega}(\mu)-\Delta_{\omega_{\rm FS}}(\mu_{\rm FS})+C.
    \end{equation}
    Since both functions $\mu$ and $\mu_{\rm FS}$ are momenta for an $\mathbb{S}^1$-action, on ${\tstM}$ and $\pr^1$ respectively, and that both actions have the same weight on $M_\infty$ and $+\infty \in \pr^1$ respectively, we get from \cite[Lemma 4.5]{BHL} that $C=0$. Therefore, we have
    \begin{equation}\label{eq:dPsi-Ric}
        - (d^c\Psi)(\uzeta)=\Delta_{\Omega}(\mu)-\Delta_{\omega_{\rm FS}}(\mu_{\rm FS}).
    \end{equation}
    This identity, together with \eqref{eq:choosing_Psi}, can be used to rewrite the right hand side of \eqref{eq:slopes_accessory}.
    Let's compute the asymptotic slope of $\widecheck{\actscal}^{\Psi}_s$ as $t\to-\infty$. Computing as in Lemma \ref{lem:ddc_A^Psi-smooth}, the $t$-derivative of $\widecheck{\actscal}^\Psi_s$ is
    \begin{equation}
    \begin{split}
        \frac{d}{dt}\widecheck{\actscal}^\Psi_s = & ns \int_M\dot{\psi}_t\frac{\omega_{\varphi_t}^{[n]}}{(1+s\dot{\varphi}_t)^{n+1}}  + 2 \int_{M_\tau}\iota_{M_\tau}^*\left(\Pi^*\Ric(\omega)\wedge \frac{(\Omega+dd^c\Gamma)^{[n-1]}}{(1-s\mu_\Gamma)^{n}}\right)\\
        &+2ns\int_{M_\tau}\iota_{M_\tau}^*\left(d\log\left(\frac{e^{\Psi+\TCmap^*\omega_{\rm FS}}}{\omega^{[n]}\wedge\omega_{\rm FS}} \right)\wedge d^c\mu_{\Gamma}\wedge \frac{(\Omega+dd^c\Gamma)^{[n-1]}}{(1-s\mu_\Gamma)^{n+1}}\right).
    \end{split} 
    \end{equation}
    As $t\to-\infty$, the $\mathbb{C}^*$-action on ${\tstM}$ is trivial near $M_{\tau=\infty}$, hence $e^{\psi_t}=(\rho(\tau)^*e^\Psi)_{|M_1}$ tends to a constant with respect to $t$ and $\mu_{\Gamma|M_{\infty}}=\mu_{\max}$ since the the fiber $M_{\infty}$ is fixed by the $\mathbb{C}^*$-action on ${\tstM}$. Thus, we obtain
    \begin{equation}\label{AsymSlope_A1_infinity}
        \underset{t\to -\infty}{\lim} \frac{d}{dt}\widecheck{\actscal}^\Psi_s=2 \int_{M_\infty} \Ric(\Omega_{|M_\infty})\wedge \frac{(\Omega_{|M_\infty})^{[n-1]}}{(1-s\mu_{\max})^{n}}
    \end{equation}
    and the thesis follows from combining \eqref{eq:dPsi-Ric} and \eqref{AsymSlope_A1_infinity} with \eqref{eq:slopes_accessory}.
\end{proof}

Now, we compare the asymptotic slopes as $\tau\to 0$ of the functions $\widecheck{\actscal}^{\Psi}_s$ and $\widecheck{\actscal}_s$.
\begin{lemma}\label{lem:equalslopes}
    Let $(\varphi_t)_{t\geq 0}$ be the geodesic ray compatible with $(\mathcal{M},\mathcal{L})$. Then, the asymptotic slopes of the action functional $\widecheck{\actscal}$ and its modified version $\widecheck{\actscal}^{\Psi}$ at the central fibre of $(\tstM,\tstL)$ along the ribbon $f^{-1}_{s,t}\eta_t\coloneqq  (1+s\dot{\phi}_t)^{-1}(\eta +d^c_\xi\phi_t)$ coincide.
\end{lemma}
\begin{proof}
    First, as $\varphi_t$ is a geodesic, the difference is
    \begin{equation}\label{A-A}
    \begin{split}
        \widecheck{\actscal}_s(t)-\widecheck{\actscal}^\Psi_s(t)
        =& ns \int_M \log\left(\frac{\omega_{\varphi_t}^{[n]}}{e^{\psi_t}}\right) \frac{\omega_{\varphi_t}^{[n]}}{(1+s\dot{\varphi}_t)^{n+1}}.
    \end{split}
    \end{equation}
    Proceeding as in the proof of \cite[Theorem 5.1]{zakarias}, we can write
    \begin{equation}
        \log\left(\frac{\omega_{\varphi_t}^{[n]}}{e^{\psi_t}}\right) =  \Pi_*\left( \log\left(\frac{(\Omega+dd^c\Gamma)^{[n]}\wedge\TCmap^*\omega_{\rm FS}}{e^{\Psi+\TCmap^*\log\omega_{\rm FS}}}\right)\right)
    \end{equation}
    where $\Gamma$ is the solution of the boundary value problem \eqref{BVP:G}, which induces the subordinate geodesic ray $\Pi^*\Phi=\Gamma-\gamma_D$. Since $\Gamma$ is $\Omega$-psh on ${\tstM}$ and $e^{\Psi+\TCmap^*\log\omega_{\rm FS}}=\Omega^{[n+1]}$ (locally), the quantity $\log\left(\frac{(\Omega+dd^c\Gamma)^{[n]}\wedge\TCmap^*\omega_{\rm FS}}{e^{\Psi+\TCmap^*\log\omega_{\rm FS}}}\right)$ is bounded from above by a constant $C$. Hence,
    \begin{equation}\label{Upper_A-A}
        \widecheck{\actscal}_s(t)-\widecheck{\actscal}^\Psi_s(t)\leq C ns \int_M \frac{\omega_{\varphi_t}^{[n]}}{(1+s\dot{\varphi}_t)^{n+1}}=\tilde{C}ns    
    \end{equation}
    since along a geodesic ray we know that $\int_M \frac{\omega_{\varphi_t}^{[n]}}{(1+s\dot{\varphi}_t)^{n+1}}$ is a constant independent from $t$ (see e.g.\ \cite[Proposition 2.2.]{BoB}).

    To get a lower bound for the difference \eqref{A-A}, we let $d\tilde{\lambda}_{t,s}\coloneqq \frac{d\lambda_{t,s}}{\widecheck{\V}(\omega,1+s\dot{\varphi}_0)}$ be the probability measure induced from $d\lambda_{t,s}=\frac{\omega_{\varphi_t}^{[n]}}{(1+s\dot{\varphi}_t)^{n+1}}$. Then
    \begin{equation}\label{eq:InequalityAA_preparation}
    \begin{split}
        \frac{\widecheck{\actscal}_s(t)-\widecheck{\actscal}^\Psi_s(t)}{ns \widecheck{\V}(\omega,1+s\dot{\varphi}_0)} = &  \int_M\log\left(\frac{e^{-\psi_t} d\tilde{\lambda}_{t,s}}{ \int e^{\psi_t}d\tilde{\lambda}_{t,s}}\right) d\tilde{\lambda}_{t,s}-\log \int_M e^{\psi_t}d\tilde{\lambda}_{t,s}+\int_M\frac{\log((1+s\dot{\varphi}_t)^{n+1})}{(1+s\dot{\varphi}_t)^{n+1}}\omega_{\varphi_t}^{[n]}.
    \end{split}
    \end{equation}
    Now, note that the last term on the right hand side is constant, as $(\varphi_t)_{t\geq 0}$ is the geodesic ray compatible with $(\mathcal{M},\mathcal{L})$ (see \cite[Proposition 2.2]{BoB}).
    Hence, using the positivity of the entropy on probability measures, we find from \eqref{eq:InequalityAA_preparation}
    \begin{align}\label{eq:InequalityAA}
        \frac{1}{ns}(\widecheck{\actscal}_s(t)-\widecheck{\actscal}^\Psi_s(t)) \geq C- \widecheck{\V}(\omega,1+s\dot{\varphi}_0)\log \int_M e^{\psi_t}.
    \end{align}
    If the central fiber $M_0$ of the test configuration ${\tstM}$ is reduced, \cite[Lemma $3.11$]{BHJ2} yields
    \begin{equation}\label{Lower_A-A}
    \begin{split}
        \frac{1}{ns}(\widecheck{\actscal}_s(t)-\widecheck{\actscal}^\Psi_s(t))
        \geq& C- \widecheck{\V}(\omega,1+s\dot{\varphi}_0)C' \log t
    \end{split}
    \end{equation}
    From \eqref{Upper_A-A} and \eqref{Lower_A-A}, the asymptotic slopes at infinity of the action functional and its modified version are equal provided that the central fibre of $({\tstM},\tstL)$ is reduced.
\end{proof}

Combining Lemma \ref{lem:AsymCurvPSI} and Lemma \ref{lem:equalslopes}, we obtain
\begin{equation}\label{AsympSlope_A}
\begin{split}
    \frac{2\pi}{ns}\lim_{t\to +\infty}\frac{d}{dt}\widecheck{\actscal}_s =& \frac{4\pi}{ns}\int_{M_\infty} \Ric(\Omega_\infty)\wedge \frac{\Omega_\infty^{[n-1]}}{(1-s\mu_{\max})^{n}} -2 \int_{{\tstM}}(\Ric(\Omega)-\pi^*\omega_{\rm FS})\wedge\frac{\Omega^{[n]}}{\left(1-s\mu\right)^{n+1}}\\
    &-(n+1)s\int_{{\tstM}}(\Delta_{\Omega}(\mu)-\Delta_{\omega_{\rm FS}}(\mu_{\rm FS})) \frac{\Omega^{[n+1]}}{\left(1-s\mu\right)^{n+2}}
\end{split}
\end{equation}
which is exactly the expression of $\bfS_s(\tstM,\tstL)$ of Theorem \ref{theo:Globalformula}.

\section{An algebraic version of the Einstein-Hilbert functional} \label{sectionTOWalgebraic}

\subsection{Brief recap on K-stability of polarized cones}\label{s:towards}

In this section, $(\widehat{Y},\xi)$ denotes a {\it polarized affine scheme} as introduced by Collins and Székelyhidi \cite{CS}.
This means that $\widehat{Y}$ is equipped with an action of a complex torus $\bT_\C$, and $\xi$ is contained in a fixed maximal compact torus $\bT\subset \bT_\C$. Moreover, for some $N>0$ there is an embedding of affine schemes $\iota: \widehat{Y}\subset \C^N$ sending $\bT_\C$ to a subgroup of the diagonal torus in $GL(N, \C)$.
Denoting $R(Y)$ the structure sheaf of $\widehat{Y}$ and its weight space decomposition
\begin{equation}
    R(Y) = \oplus_{\alpha\in\kt^*}R_\alpha,
\end{equation}
we ask that the vector $\xi$ lies in the \emph{Reeb cone}
\begin{equation}
    \kt_+ \coloneqq  \left\lbrace \zeta \in \kt \mid \langle \alpha,\zeta \rangle >0,\,\forall\alpha\text{ s.t.\ }R_\alpha \neq\{0\} \right\rbrace.
\end{equation}

We will only work with polarised affine schemes arising from a polarised manifold. These are the counterpart of the Sasaki manifolds defined in \S\ref{sss:SASpolarizedManif}.

\begin{example}\label{ex1}
    Let $\pi :L \ra M$ be a very ample line bundle on a compact complex manifold. There exists a Hermitian metric $h$ on $L$ whose curvature $\omega_h$ is a K\"ahler form $\omega_h$ on $M$.  Now consider the total space of the dual line bundle with the zero section removed $Y \coloneqq  L^{-1}\setminus M$ and the fibrewise norm function $r_{h} : Y \to \R_{>0}$ induced (via duality) by $h$ on $L^{-1}$. Then $(Y ,\frac{1}{4} dd^c r_h^2)$ is a K\"ahler cone, with radial action given by dilation along the fibres, and Sasaki-Reeb vector field $\xi$ induced by the $S^1$-action on the fibres. Note that $Y$ has a natural one-point completion $\widehat{Y}=Y\cup\{0\}$, the singular space obtained by blowing down the zero section in the total space of $L^{-1}$. On the other hand, as $L$ is very ample, its space of sections provides an embedding of $M$ into $\pr^{N-1}$, and an $\C^*$-equivariant embedding of $\widehat{Y}$ into $\C^N$. This way the variety $(\widehat{Y},\xi)$ is a polarized complex cone as above.  
\end{example}
Collins and Sz\'ekelyhidi consider the \emph{$\bT$-equivariant index character of $(Y,\xi)$}, defined by
\begin{equation}\label{eq:indexchar}
    \Hf(Y,t,\xi) = \sum_{\alpha \in \kt^* } e^{-t \langle \alpha,\xi \rangle} \dim R_\alpha
\end{equation}
as a function of $t\in \C$, and point out that $\Hf(Y,t,\xi)$ converges as soon as $\text{Re}\,t>0$ and $\xi \in \kt_+$.  Moreover, they show that this function admits the following expansion near $t=0$
\begin{equation}\label{eqCScharacter}
    \Hf(Y, t,\xi)= \frac{a_0(\xi) n!}{t^{n+1}} + \frac{a_1(\xi) (n-1)!}{t^{n}} + O(t^{1-n})
\end{equation}
where $a_0, a_1:\kt_+\ra \R$ are smooth functions \cite[Theorem 3]{CS}, called the \emph{pole coefficients}. We may write $a_i(Y,\xi)$ to emphasize the dependency in $Y$.

\smallskip

Similarly to~\cite{Tian}, in \cite{CS} a test configuration of an affine cone $(\widehat{Y},\xi)$ is defined as the orbit of a $\C^*$-action $\C^*_\zeta \subset GL(\C^N)$ commuting with $\xi$, together with its flat limit across $0$. The central fibre $(\widehat{Y}_0,\xi)$ of the test configuration inherits the $\bT_\C$-action of $\widehat{Y}$, but also has an extra $\C^*$-action induced by $\zeta$. In our notation $\zeta$ is a real vector field, the generator of the action of $S^1 \subset \C^*$.

Note that for $s \in \R$ small enough the vector field $\xi_s\coloneqq \xi-s \zeta$ is in the Reeb cone of $Y_0$, $\kt_+(Y_0) \subset \kt\oplus \R$ and thus for $t\in \C$ with $\text{Re } t >0$ we can consider the index character
\begin{equation}
    \Hf(Y_0,t,\xi_s )= \frac{a_0(\xi_s ) n!}{t^{n+1}} + \frac{a_1(\xi_s ) (n-1)!}{t^{n}} + O(t^{n-1}).
\end{equation}
Thus the derivative $D_{-\zeta} a_i = \frac{d}{ds} a_i(\xi_s)_{|{s=0}}$ exists and they define
\begin{equation}\label{eqDFofCS}
    DF(Y_0,\xi, \zeta) \coloneqq   \frac{a_0(\xi)}{n} D_{-\zeta} \left(\frac{a_1}{a_o}\right) (\xi) + \frac{a_1(\xi) (D_{-\zeta} a_0)_\xi}{n(n+1)}. 
\end{equation}
\begin{definition}(\cite[Definition 5.3]{CS}) We say that $(Y, \xi)$ is K-semistable if, for every torus $\bT \ni \xi$,
and every $\bT$-equivariant test configuration with central fibre $Y_0$, we have
$$DF(Y_0,\xi, \zeta) \geq 0$$
where $\zeta$ denotes the generator of the induced $\C^*$-action on the central fibre.
\end{definition}

\subsection{An affine version of the Einstein-Hilbert functional}
As noticed in \cite{BHLTF1}, a convenient way to write the Donaldson-Futaki invariant \eqref{eqDFofCS} is 
\begin{equation}\label{eqDFofCS2}
    DF(Y_0,\xi, \zeta) = \frac{a_0(\xi)^{n/n+1}}{n}\frac{d}{ds}\left( \frac{a_1(\xi-s\zeta)}{a_0(\xi-s\zeta)^{n/n+1}} \right)_{|s=0}.
\end{equation}
\begin{definition}
    Let $Y_0$ be a $T$-invariant affine cone with non-empty Reeb cone $\kt_+$ and admitting an asymptotic expansion \eqref{eqCScharacter} with smooth coefficients $a_0, a_1:\kt_+ \to \R$ such that $a_0>0$ on $\kt_+$. The \emph{affine Einstein-Hilbert functional} is
\begin{equation}
    \EH^a(Y_0,\xi) =16\pi \, a_1\,a_0^{-\frac{n}{n+1}}
\end{equation}
\end{definition}
To motivate this version of the Einstein-Hilbert functional we consider the case where the central fibre has at worst orbifold singularities and $\xi_s=\xi-s\zeta$ is quasi-regular, i.e.\ it induces a $\C^*$-action. Therefore, the quotient $M_{\xi_s } \coloneqq  Y_0/\C^*_{\xi_s }$ comes with an ample line bundle, say $L_{\xi_s }$, and the pole coefficients can be expressed in terms of intersection classes, see e.g.\ \cite[Theorem A.1]{BHJ}, as
\begin{equation}
    a_0(\xi_s ) =c_1(L_{\xi_s  })^{[n]},\quad a_1(\xi_s  ) = \frac{1}{2}\,c_1(M_{\xi_s  }).c_1(L_{\xi_s  })^{[n-1]}.
\end{equation}
To connect with the previous sections, pick any $\eta \in \holSAS(M_{0},L_{0})$. Then we have
\begin{equation}
    a_0(\xi_s ) = \frac{1}{(2\pi)^{n+1}}\bfV\left(\frac{\eta}{\eta(\xi_s )}\right),\quad a_1(\xi_s ) =\frac{1}{8(2\pi)^{n+1}}\bfS\left(\frac{\eta}{\eta(\xi_s )}\right).
\end{equation}
Thus, in this case, the affine Einstein-Hilbert functional equals the usual one evaluated on $\eta(\xi_s)^{-1}\eta$.
This holds for any CR contact form $\eta \in \holSAS(M_{0},L_{0})$ of Sasaki type since $\xi_s$ is CR-holomorphic, see \cite{FOW}. It is then straightforward to prove the following.
\begin{lemma}\label{lem:alg_limit} 
    Let $(M,L)$ be a smooth polarized manifold and $(\tstM,\tstL)$ a smooth ample test configuration whose central fibre has at worst orbifold singularities. For any compatible Sasaki-contact form 
    $\tilde{\eta} \in \Cmet(\tstM,\tstL)$ and for any $s>0$  such that $\xi_s \coloneqq \xi-s\zeta$ is in the Reeb cone of $(\tstY, \xi)$ we have 
    \begin{equation}\label{eq:limEH}
        \lim_{\tau\ra 0} \EH\left(\iota^*_\tau\left(\tilde{\eta}(\xi_s )^{-1}\tilde{\eta}\right) \right) = \EH^a(Y_0,\xi_s ).
    \end{equation}     
\end{lemma}

\subsection{A global expression for the affine Einstein-Hilbert functional}

Let $(\tstM,\tstL)$ be a compact ample test configuration over $(M,L)=(M_\infty,L_\infty)$, and denote by $(\tstY,\xi)$ the test configuration of the polarized complex cone over $(Y,\xi)$, see \S\ref{ss:testconfigSASAK}, with central fibre $(Y_0,\xi)$, a polarized cone over the central fibre $(M_0,L_0)$. As before, $\bT_\C\coloneqq \C^*_\xi \times \C^*_\zeta$ is the $2$-torus acting on $\tstY$ and on $Y_0$. The ampleness condition is used to embed equivariantly the cones as affine cones (with the apex removed) in $\C^N$. The $\bT$-equivariant index character depends on the torus action and thus on the linearization we have picked on $\tstL$. Observe that the linearization of $\C^*_\zeta$ induces a non-necessarily trivial action on the fibre $L_\infty \ra M_\infty$ with weight we denote $\mu_{\max}$.           
\begin{proposition}\label{prop:GlobalExpAffEH}
    Let $(\tstY,\xi)$ be a test configuration for the polarized complex cone  $(Y,\xi)$ with reduced central fibre $(Y_0,\xi)$, and let $\zeta$ denote the vector field inducing the test configuration action. Then, for any $s\in \R$ such that $\xi_s =\xi-s\zeta$ is in the Reeb cone of $(\tstY,\xi)$, the coefficients of the $\bT$-equivariant index character satisfy     
    \begin{equation}
    \begin{split}
        a_0(Y_0,\xi_s)= & \frac{a_0(Y,\xi)}{(1-s\mu_{\max})^{n+1}}+ s(n+1)\,a_0(\tstY,\xi_s)\\
        a_1(Y_0,\xi_s)=&\frac{a_1(Y,\xi)}{(1-s\mu_{\max})^{n}} -  \frac{s\,n\,a_0(Y,\xi)}{(1-s\mu_{\max})^{n+1}} - \frac{s^2n(n+1)}{2} a_0(\tstY,\xi_s) + s\,n\, a_1(\tstY,\xi_s)
    \end{split}
    \end{equation}
    where $\mu_{\max}\in \R$ is the weight of the induced action of $\C^*_\zeta$ on the fibre $L_\infty \ra M_\infty$.  

    If, moreover, $(\tstY,\xi)$ is the polarized complex cone associated to a smooth ample test configuration $(\tstM,\tstL)$ over a polarized manifold $(M,L)$ then we have
    \begin{equation}
        a_0(\tstY,\xi_s ) = \frac{\bfV(\tstN, f_s^{-1}\tilde{\eta})}{(2\pi)^{n+2}}  \quad \mbox{ and } \quad a_1(\tstY,\xi_s ) = \frac{\bfS(\tstN, f_s^{-1}\tilde{\eta})}{8(2\pi)^{n+2}}
    \end{equation}
    for any $\tilde{\eta} \in \holSAS(\tstM,\tstL)$, where $f_s\coloneqq \tilde{\eta}( \xi_s )$ and $\tstN$ denotes the circle bundle over $\tstM$ associated to $\tstL$.
\end{proposition}
Recall that $\mR(\widehat{\tstY}) = \bigoplus_{k\in\N} H^0(\tstM,\tstL^k)$
is identified with the space of continuous functions on $\widehat{\tstY}$ which are holomorphic on $\tstY= L^{-1}\backslash \tstM$ and polynomial on each fibre of $\tstY$ over $\tstM$ (which  are the $\C^*_\xi$ orbits). Via this identification, $H^0(M,L^k)$ is the $k$-weight space for the $\C^*$-action, i.e.\ the set of functions satisfying $\mL_{-J\xi} f = kf$. The action of $\C^*_\zeta$ yields another weight space decomposition
\begin{equation}\label{e:homRingDECOMP}
    H^0(\tstM,\tstL^k) = \bigoplus_{m\in \Z} H^0(\tstM,\tstL^k)_m
\end{equation}
where $H^0(\tstM,\tstL^k)_m$ is the $m$-weight pace of $\C^*_\zeta$ in $H^0(\tstM,\tstL^k)_m$.

We will compute the $\bT$-equivariant index character $\Hf(Y_0, t,\xi_s)$ of the central fibre $(\widehat{Y}_0, \xi)$ in terms of the dimension of the spaces $H^0(\tstM,\tstL^k)_m$. Recall from \eqref{eq:indexchar} that the $\bT$-equivariant index character is
\begin{equation}\label{eq:HilbFctY0}
    \Hf(Y_0, t,\xi_s) = \sum_{(k,m)\in \Z^2} e^{-(tk-tsm)}   \dim H^0(M_0,L_0^k)_m.
\end{equation} 
\begin{proof}[Proof of Proposition~\ref{prop:GlobalExpAffEH}] 
    We adapt a strategy of \cite[p.315]{donaldson-toric}. Observe first that to obtain the coefficients of the poles of the $\bT$-equivariant index character, it is sufficient to consider the truncated series $k\geq k_0$ for some fixed $k_0$, since the difference is a smooth function which does not contribute to the pole coefficients. So we pick $k_0\gg 0$ such that the following sequences are exact for all $k\geq k_0$
    \begin{equation}\label{eq:ExactSeq}
    \begin{split}
        0&\rightarrow H^0(\tstM, \tstL^k\otimes\TCmap^*\mO(-1)) \xrightarrow{\otimes\,\sigma_0} H^0(\tstM, \tstL^k)   \xrightarrow{{\rm r}_0} H^0(M_0, L_0^k) \rightarrow 0 \\
        0&\rightarrow H^0(\tstM, \tstL^k\otimes\TCmap^*\mO(-1))  \xrightarrow{\otimes\,\sigma_\infty} H^0(\tstM, \tstL^k) \xrightarrow{{\rm r}_\infty} H^0(M_\infty, L_\infty^k) \rightarrow 0. 
    \end{split}
    \end{equation}
    The second arrows are given by multiplication by the sections of $\TCmap^*\mO(-1)$, say $\sigma_0$ and $\sigma_\infty$, defining $M_0$ and $M_\infty$, respectively. The third are the restriction to $M_0$ and $M_\infty$, that we denote respectively by ${\rm r}_0$ and ${\rm r}_\infty$. For the moment, denote by ${\bf A}$ the weight of the $\C^*_\zeta$-action on $\sigma_0$. Therefore, for any $m\in \Z$, we have 
    \begin{equation}\label{eq:ExactEquivSeqDon0}
        0\rightarrow H^0(\tstM, \tstL^k\otimes\TCmap^*\mO(-1))_{m}  \xrightarrow{\otimes\,\sigma_0} H^0(\tstM, \tstL^k)_{m+{\bf A}}  \xrightarrow{{\rm r}_0} H^0(M_0, L_0^k)_{m+{\bf A}} \rightarrow 0 
    \end{equation}
    The surjectivity of ${\rm r}_0$ in \eqref{eq:ExactEquivSeqDon0} follows from the surjectivity and equivariance of ${\rm r}_0$ in \eqref{eq:ExactSeq}. Similarly, if we denote the weight of the $\C^*_\zeta$-action on $\sigma_\infty$ by ${\bf B}$, we also have 
    \begin{equation}\label{eq:ExactEquivSeqDoninfty}
        0\rightarrow H^0(\tstM, \tstL^k\otimes\TCmap^*\mO(-1))_{m}  \xrightarrow{\otimes\,\sigma_\infty} H^0(\tstM, \tstL^k)_{m+{\bf B}} \xrightarrow{{\rm r}_\infty}  H^0(M_\infty, L_\infty^k)_{m+{\bf B}} \rightarrow 0. 
    \end{equation}
    Note that $M_\infty$ is fixed by the torus action but the linearization on $\mL$ induces an action on the fibres of $L_\infty$, of weight $\mu_{\max}$. Therefore,  $H^0(M_\infty, L_\infty^k)_{m+1}$ is non trivial only if $m+1=k\mu_{\max}$.  

    From \eqref{eq:ExactEquivSeqDoninfty} and \eqref{eq:ExactEquivSeqDon0} we obtain, setting $h^0(\cdot)_m = \dim H^0(\cdot)_m$,
    \begin{equation}
        h^0 (M_0, L_0^k )_m  = h^0(\tstM, \tstL^k)_m - \left(h^0(\tstM, \tstL^k)_{m+({\bf B}-{\bf A})} - h^0(M_\infty, L_\infty^k)_{m+({\bf B}-{\bf A})}\right).
    \end{equation}
    For a fixed $k\geq k_0$ we have
    \begin{equation}
    \begin{split}
        &\sum_{m\in \Z}e^{-t(k-sm)} h^0(M_0, L_0^k)_m=\\
        &=  \left(1-e^{-ts({\bf B}-{\bf A})}\right) \sum_{m\in \Z} e^{-t(k-sm)}h^0(\tstM,\tstL^k)_m  +   e^{-tk(1-s\mu_{\max}) -ts({\bf B}-{\bf A})} h^0(M_\infty,L_\infty^k).
    \end{split}
    \end{equation}
    We deduce that the (truncated) $\bT$-equivariant index character satisfies  
    \begin{equation}
        \Hf_{k_0}(Y_0,t,\xi_s)= (1-e^{-ts({\bf B}-{\bf A})}) \Hf_{k_0}(\tstY,t,\xi_s) + e^{-ts({\bf B}-{\bf A})}\Hf_{k_0}(Y,t,(1-s\mu_{\max})\xi).
    \end{equation} 
    Recall that, by definition,
    \begin{equation}
    \begin{split}
        & \Hf(\tstY, t,\xi_s) = \frac{a_0(\tstY,\xi_s) (n+1)!}{t^{n+2}} + \frac{a_1(\tstY,\xi_s) n!}{t^{n+1}} + O(t^{-n})\\
        & \Hf(Y, t,(1-s\mu_{\max})\xi) = \frac{a_0(Y,\xi) n!}{ (1-s\mu_{\max})^{n+1} t^{n+1}} + \frac{a_1(Y,\xi) (n-1)!}{(1-s\mu_{\max})^nt^{n}} + O(t^{1-n}).
    \end{split}
    \end{equation}
    The pole coefficients of $\Hf(Y_0,t,\xi_s)$ can be computed as 
    \begin{equation}
        a_0(Y_0,\xi_s)=\frac{a_0(Y,\xi)}{(1-s\mu_{\max})^{n+1}}+ s(n+1)({\bf B}-{\bf A})a_0(\tstY,\xi_s)
    \end{equation}
    \begin{equation}\label{eq:a1_conecentralfibre}
    \begin{split}
        a_1(Y_0,\xi_s)=&\frac{a_1(Y,\xi)}{(1-s\mu_{\max})^{n}} - s\,n ({\bf B}-{\bf A})\frac{a_0(Y,\xi)}{(1-s\mu_{\max})^{n+1}} \\ &+ s\,n({\bf B}-{\bf A})\, a_1(\tstY,\xi_s) - s^2 ({\bf B}-{\bf A})^2\frac{n(n+1)}{2} a_0(\tstY,\xi_s).
    \end{split}
    \end{equation}
    Now, we argue that ${\bf B}-{\bf A} =-1$. Consider any linearized $\C^*$ action on $\mO(1)$ lifting the standard action on $\pr^1$ and denote $\mu^{\rm FS}$ the associated moment map of the Fubini-Study metric normalized to represent $\mO(1)$ (rather than $\mO(2)$), so that $\mu_{\max}^{\rm FS}-\mu_{\min}^{\rm FS}=1$. As is well known, see e.g.\  \cite[Chapter 8]{PGbook}, the weight of the $\C^*$-action on the fibre of $\mO(1)$ over a fixed point is the value of the momentum map at that point. For the fixed point $0\in \pr^1$ the weight is thus the minimum value of $\mu^{\rm FS}$ and since $\sigma_{\infty}(0) \neq 0$ (and that $0\in \pr^1$ is fixed by the action), the weight of the action on $\sigma_\infty$, as a global section, must coincides with $\mu^{\rm FS}(0)$, that is ${\bf B} = \mu^{\rm FS}(0)$. Similarly for $\sigma_0$ we have ${\bf A} = \mu^{\rm FS}(\infty)$. Thus, ${\bf B}-{\bf A} =-1$. 

    Since $(M,L)$ is smooth, the first term in \eqref{eq:a1_conecentralfibre} is
    \begin{equation*}
        \frac{a_1(Y,\xi)}{(1-s\mu_{\max})^{n}}  = \frac{1}{2(2\pi)^{n}(1-s\mu_{\max})^{n}} \int_{M} \Ric(\Omega_\infty)\wedge \Omega_\infty^{[n-1]}
    \end{equation*}
    while the third term is, for any $\tilde{\eta} \in \holSAS(\tstM,\tstL)$,
    \begin{equation}
        a_1(\tstY,\xi_s) = \frac{1}{8(2\pi)^{n+2}} \bfS(\tstN, f_s^{-1}\tilde{\eta}).\qedhere
    \end{equation}
\end{proof}

\begin{proposition}\label{prop:5.5followup}
    With the hypotheses of Proposition \ref{prop:GlobalExpAffEH}, if $(\tstM,\tstL)$ is a smooth ample test configuration and $\Omega \in 2\pi c_1(\tstL)$ is a smooth K\"ahler metric, we have the relation
    \begin{equation}
    \begin{split}
        \frac{2\,a_0(Y,\xi)}{(1-s\mu_{\max})^{n+1}} -  &s (n+1)\,a_0(\tstY,\xi_s ) \\ 
        &=\frac{1}{(2\pi)^{n+1}}\int_{\tstM}\left(\TCmap^*\omega_{\rm FS}\wedge\frac{\Omega^{[n]}}{f_s^{n+1}} + \frac{s}{2}(n+1)\Delta^{\omega_{\rm FS}}\mu_{\rm FS} \frac{\Omega^{[n+1]}}{f_s^{n+2}} \right)
    \end{split}
    \end{equation}
    where $\omega_{\rm FS}$ is the Fubini-Study metric of Einstein constant $1$ and $\mu_{\rm FS}$ is its momentum map for the standard action on $\pr^1$. 
\end{proposition}
\begin{proof}
    Note that $2\mu_{\rm FS} =\Delta^{\omega_{\rm FS}}\mu_{\rm FS}$ since $\omega_{\rm FS} = \Ric(\omega_{\rm FS})$ (i.e.\ the Einstein constant is $1$).
    We first exhibit an interpretation of the expression
    \begin{equation*}
        \int_{\tstM}\left(\TCmap^*\omega_{\rm FS}\wedge\frac{\Omega^{[n]}}{f_s^{n+1}} + s(n+1) \mu_{\rm FS} \frac{\Omega^{[n+1]}}{f_s^{n+2}} \right)
    \end{equation*}
    in terms of intersection of equivariant closed form on the total space of $\tstL^{-1}$. Let $r:\tstL^{-1}\ra \R_{\geq 0}$ be the distance to the zero section induced by a dual Hermitian metric on $\tstL$ whose curvature is $\Omega$. Let $\tstN \subset \tstL^{-1}$ the unit circle bundle over $\tstM$ with the contact form $\tilde{\eta} = d^c\log r$, so that $d\tilde{\eta} =\pi^*\Omega$ where $\pi: \tstL^{-1} \ra \tstM$ is the bundle map. Denoting $\tilde{\eta}_{f_s} =f_s^{-1}\tilde{\eta}$, we have
\begin{equation}\label{eq:equivlocCONES}
    \begin{split}
        &\int_{\tstM}\left(\TCmap^*\omega_{\rm FS}\wedge\frac{\Omega^{n}}{f_s^{n+1}}  + s\TCmap^*\mu_{\rm FS}\, \frac{\Omega^{n+1}}{f_s^{n+2}} \right) \\
       & =\frac{1}{2\pi}\int_{\tstN}\left((\TCmap\circ \pi)^*\omega_{\rm FS}\wedge\tilde{\eta}_{f_s}\wedge (d\tilde{\eta}_{f_s} )^{n} +s(\TCmap\circ \pi)^*\mu_{\rm FS} \tilde{\eta}_{f_s}\wedge (d\tilde{\eta}_{f_s})^{n+1} \right).
    \end{split}
\end{equation}
We consider the $2$-form
\begin{equation*}
\coneMetric_{f_s} \coloneqq  \frac{1}{2} d(r^2\tilde{\eta}_{f_s}) = rdr\wedge \tilde{\eta}_{f_s} + \frac{r^2}{2}d\tilde{\eta}_{f_s}= f_s^{-1}dr\wedge d^cr + f_s^{-1}r^2\pi^*\Omega + df_s^{-1}\wedge d^cr^2,
\end{equation*}
which is smooth on the total space of $\tstL^{-1}$ and vanishes when pulled back to the $0$-section $\tstM$ since $dr\wedge d^cr =(dd^cr^2 -2r^2\pi^*\Omega)/4$. We use that $\int_{\R_{\geq 0}} e^{-r^2/2}r^{2m+1}dr = 2^m m!$, for any $m\in \N$, to write \eqref{eq:equivlocCONES} as   
\begin{equation*}
    \begin{split}
        &\frac{1}{2\pi }\int_{\tstL^{-1}}(\TCmap\circ \pi)^*\omega_{\rm FS}\wedge \frac{\coneMetric_{f_s}^{n+1}}{(n+1)!} + s \frac{1}{2\pi}\int_{\tstL^{-1}}(\TCmap\circ \pi)^*\mu_{\rm FS}\frac{\coneMetric_{f_s}^{n+2}}{(n+2)!}\\
        &= \frac{1}{2\pi }\int_{\tstL^{-1}}\left(\TCmap^*\left(\omega_{\rm FS} + s \mu_{\rm FS}\right)\right)\wedge e^{\coneMetric_{f_s}-r^2/2}
    \end{split}
\end{equation*}
We have $\xi_s\intprod\coneMetric_{f_s} = -rdr$ and $\xi_s\intprod(\TCmap\circ \pi)^*\omega_{\rm FS} = s d\mu_{\rm FS}$. In that case, the form 
\begin{equation}
    \left(\TCmap^*\left(\omega_{\rm FS} + s \mu_{\rm FS}\right)\right)\wedge e^{\coneMetric_{f_s}-\frac{r^2}{2}}
\end{equation}
is a $\bT$-equivariantly closed form evaluated on $\xi_s$. 

Note that the fixed points set $\mbox{Fix}_{\bT}(\tstL^{-1})$ of the action of $\bT$ lie in the zero section $\tstM$ and coincide with the fixed point set of $\C^*_{\zeta}$ on $\tstM$. Therefore, because $\tstM$ is smooth and compact, $\mbox{Fix}_{\bT}(\tstL^{-1})$ is a disjoint union of compact connected smooth submanifolds. Thanks to the test configuration equivariant map we know that one of these subsets is $M_\infty$ and the others, that we generically denote $Z$, lie in the central fibre. Then, the pullbacks through the inclusions give 
$$\iota_{M_\infty}^*\left(\TCmap^*\left(\omega_{\rm FS} +\ s \mu_{\rm FS}\right)\wedge e^{\coneMetric_{f_s}-r^2/2}\right) = s\mu^{\rm FS}_{\max}=s $$
$$\iota_{Z}^*\left(\TCmap^*\left(\omega_{\rm FS} + s \mu_{\rm FS}\right)\wedge e^{\coneMetric_{f_s}-r^2/2} \right)= s\mu^{\rm FS}_{\min}=-s $$
since the pullback of $\coneMetric_{f_s}-r^2/2$ by the inclusion $\tstM\subset \tstL^{-1}$ vanishes. 

For $s\neq 0$, $\xi_s $ is generic (i.e.\ its zero locus coincides with $\mbox{Fix}_{\bT}(\tstL^{-1})$) and the Atiyah-Bott-Berline-Vergnes equivariant localisation formula gives  
\begin{equation*}
    \begin{split}
      \int_{\tstL^{-1}} &\left((\TCmap\circ\pi)^*\left(\omega_{\rm FS} +s\mu_{\rm FS}\right)\right)\wedge e^{\coneMetric_{f_s}-r^2/2}= \sum_{Z} \int_Z\frac{ -s }{\euler_{\xi_s }\left(E_{Z}^{\tstL^{-1}}\right)} + \int_{M_\infty}\frac{ s }{\euler_{\xi_s }\left(E_{M_\infty}^{\tstL^{-1}}\right)}
    \end{split}
\end{equation*} where $E_{A}^{B}$ is the equivariant normal bundle of an equivariant inclusion $A\subset B$ and whose equivariant Euler class evaluated on $\xi_s =\xi-s\zeta$ is denoted $\euler_{\xi_s }\left(E_{A}^{B}\right)$.

A similar argument shows
\begin{equation}
    2\pi \int_{\tstM} \frac{\Omega^{n+1}}{f_s^{n+2}} = \int_{\tstL^{-1}} e^{\coneMetric_{f_s}-r^2/2} = \sum_{Z} \int_Z\frac{ 1 }{\euler_{\xi_s }\left(E_{Z}^{\tstL^{-1}}\right)} +  \int_{M_\infty}\frac{ 1 }{\euler_{\xi_s }\left(E_{M_\infty}^{\tstL^{-1}}\right)}.
\end{equation}
In sum, we get that
\begin{equation}
    2\pi \int_{\tstM}\left(\TCmap^*\omega_{\rm FS}\wedge\frac{\Omega^{n}}{f_s^{n+1}} + s\mu_{\rm FS} \frac{\Omega^{n+1}}{f_s^{n+2}} \right) = -2\pi s\int_{\tstM} \frac{\Omega^{n+1}}{f_s^{n+2}} +2s  \int_{M_\infty}\frac{ 1 }{\euler_{\xi_s }\left(E_{M_\infty}^{\tstL^{-1}}\right)}.
\end{equation}
At this point, we should compute the equivariant Euler class of $E_{M_\infty}^{\tstL^{-1}} \simeq E_{M_\infty}^{L_\infty} \oplus E_{M_\infty}^{\tstM}$. We follow the precise explanation of \cite[\S 5.2]{MSY2} and apply it for a $\bT$-equivariant line bundle $E$ on which the weight of the fiberwise action is ${\bf w} \in Lie(\bT)^*$. For $a\in Lie(\bT)$, $$\euler_{a}\left(E\right) = \frac{1}{2\pi} \left( \langle {\bf w},a  \rangle  + 2\pi c_1(E)\right).$$ In our situation, this gives $\euler_{\xi_s }\left( E_{M_\infty}^{\tstM}\right) = \frac{s}{2\pi}$ because $\mu$ achieves a maximum on $M_\infty$. For the other weight, following the line of ideas in the proof of \cite[Lemma 2.13]{ACL}, we get the expression
\begin{equation*}
    \euler_{\xi_s }\left(E_{M_\infty}^{L_\infty}\right) = \frac{1}{2\pi} \left( (1- s\mu_{\max})  + 2\pi c_1(L_\infty^{-1})\right).
\end{equation*}
With this in place, we obtain
$$\int_{M_\infty}\frac{ 1 }{\euler_{\xi_s }\left(E_{M_\infty}^{\tstL^{-1}}\right)} = \frac{(2\pi)^2}{s(1- s\mu_{\max})} \int_{M_\infty} \frac{1}{1- \frac{\Omega_\infty}{1- s\mu_{\max}}} =\frac{(2\pi)^2}{s(1- s\mu_{\max})^{n+1}} \int_{M_\infty}\Omega_\infty^n$$ from which we infer the result as a consequence of
\begin{equation*}
    \int_{\tstM}\left(\TCmap^*\omega_{\rm FS}\wedge\frac{\Omega^{n}}{f_s^{n+1}} + s \mu_{\rm FS} \frac{\Omega^{n+1}}{f_s^{n+2}} \right) = -s\int_{\tstM} \frac{\Omega^{n+1}}{f_s^{n+2}} + 2  \frac{2\pi}{(1- s\mu_{\max})^{n+1}} \int_{M_\infty}\Omega_\infty^n.\qedhere
\end{equation*}
\end{proof}

\section*{Statements and Declarations}

\paragraph{Acknowledgements.}
We thank Olivier Biquard and Zakarias Sj\"ostr\"om Dyrefelt for valuable discussions, as well as Ruadha\'i Dervan, Paul Gauduchon, Vestislav Apostolov, Jacopo Stoppa and Rémi Reboulet for their interest in this work. We thank the anonymous referees for their insightful comments, that helped us improve a previous version of this paper.

\paragraph{Funding.}
A.L.\ is supported by the Villum Young Investigator grant 0019098. E.L.\ is supported by the ANR-FAPESP grant PRCI ANR-21-CE40-0017, she is also member of ANR-21-CE40-0011 JCJC MARGE and partially supported by a BQR grant of the University of Lyon~1. C.S.\ is supported by a MSCA Postdoctoral Fellowship, funded by the Research and Innovation framework programme Horizon Europe {\normalsize\euflag} under grant agreement 101149320. This work has benefited from a visit of C.S. to the Université Lyon 1 covered by the FQRNT (Qu\'ebec, Canada).  

\paragraph{Conflicts of interest.}
On behalf of all authors, the corresponding author states that there is no conflict of interest.

\paragraph{Data management.}
Data sharing is not applicable to this article as no datasets were generated or analysed in this paper.

\paragraph{Article version.} This version of the preprint has not undergone peer review or any post-submission improvements or corrections. The Version of Record of this article is
published in \textit{Mathematische Annalen}, and is available online at \href{https://doi.org/10.1007/s00208-026-03578-3}{doi.org/10.1007/s00208-026-03578-3}.

\addcontentsline{toc}{section}{References}
\bibliographystyle{abbrv}

\end{document}